%% file: StoSubDir26_v2.tex
\documentclass{article}
\usepackage[utf8]{inputenc}
\input{preamble}

\usepackage{orcidlink}
\title{Direct search for stochastic optimization in random subspaces with zeroth-, first-, and second-order convergence and expected complexity\thanks{This work was supported by the CAMPA and ComPASS-4, projects of the U.S. Department of Energy, Office of Science, Office of Advanced Scientific Computing Research and Office of High Energy Physics, Scientific Discovery through Advanced Computing (SciDAC) program under Contract No.~DE-AC02-06CH11357.}}
\author{
	\href{mailto:kdzahini@anl.gov}{K. J. Dzahini \orcidlink{0000-0002-1515-4251}}\thanks{Argonne National Laboratory, 9700 S. Cass Avenue, Lemont, IL 60439, USA (\href{https://www.anl.gov/profile/kwassi-joseph-dzahini}{www.anl.gov/profile/kwassi-joseph-dzahini}). Mathematics and Computer Science Division.
	}
	\and
	\href{mailto:wild@lbl.gov}{S. M. Wild \orcidlink{0000-0002-6099-2772}}\thanks{Lawrence Berkeley National Laboratory, 1 Cyclotron Road, Berkeley, CA 94720, USA ( \href{https://wildsm.github.io/}{https://wildsm.github.io/}). Applied Mathematics and Computational Research Division.
	}
}
\date{\today}
\begin{document}

%\linenumbers % To make it easier for reviewers

\maketitle

%\vspace*{-0.5cm}

\noindent
{\bf Abstract:} Stochastic directional direct-search (SDDS) algorithms were recently introduced as an extension to stochastically noisy objectives of a broad class of algorithms including the well-known mesh adaptive direct-search (MADS) algorithms developed for the minimization of deterministic functions in a blackbox optimization framework. However, since SDDS methods explore the variable space via directions selected at each iteration from search sets of cardinality depending on the problem dimension, their performance quickly deteriorates as the dimension gets larger. {\bl This work introduces StoDARS}, a framework for large-scale stochastic blackbox optimization that not only is both an algorithmic and theoretical extension of the SDDS framework but also extends to noisy objectives a recent framework of direct-search algorithms in reduced spaces (DARS). Unlike SDDS, StoDARS achieves scalability by using~$m$ search directions generated in random subspaces defined through the columns of Johnson--Lindenstrauss transforms (JLTs) obtained from Haar-distributed orthogonal matrices, where the user-determined parameter~$m$ is independent of the dimension of the problem. For theoretical needs, the quality of these subspaces and the accuracy of random estimates used by the algorithm are required to hold with sufficiently large, but fixed, probabilities. In particular, the almost sure convergence to zero of the sequence of the algorithm's stepsize parameters, referred to as zeroth-order convergence, is demonstrated by using the theory of stochastic processes. Then, leveraging an existing supermartingale-based framework, the expected complexity of StoDARS is proved to be similar to that of SDDS and other stochastic full-space methods up to constants, when the objective function is continuously differentiable. By dropping the latter assumption, the ability of StoDARS to generate a dense set of subspace directions by means of the aforementioned JLTs allows its analysis to be the first of a JLT-based subspace algorithm establishing convergence to Clarke stationary points with probability one, unlike prior works on subspace methods where the use of gradient is inevitable. Moreover, the analysis of the second-order behavior of
MADS using a second-order-like extension of the Rademacher's theorem-based definition of the Clarke subdifferential (so-called generalized Hessian) is extended to the StoDARS framework, making it the first in a stochastic direct-search setting, to the best of our knowledge.

$ $

$ $\\
{\bf Mathematics Subject Classiﬁcation} 90C15 $\cdot$ 90C30 $\cdot$ 90C56 %$\cdot$ MSC4

%\tableofcontents
\clearpage

\clearpage
\section{Introduction}
The growing use of sensors and observation-driven design and discovery 
%and computer science 
has attracted interest in a number of engineering and scientific fields for solving complex optimization problems where the objective function is available only through a {\it black box}~\cite{AuHa2017}. Such situations where closed-form expression of the objective function and its derivatives are not available are addressed by derivative-free optimization (DFO). DFO methods can be broadly classified into two main categories~\cite{RobRoy2022RedSpace}: model-based methods~\cite{blanchet2016convergence,CoScVibook,CRsubspace2021,chen2018stochastic,DzaWildSub2022,shashaani2018astro}, where a model of the objective function is constructed to guide the optimization process, and direct-search methods~\cite{AbAu06,AuDe2006,audet2019stomads,CoScVibook,Ch2012,dzahini2020expected,dzahini2020constrained,DzaRinRoyZef2024Survey,rinaldi2022weak}, which do not build models but evaluate at each iteration the objective function at a collection of points and act solely based on those function values without approximating derivatives.

This work focuses on the following stochastic DFO problem:
\begin{equation}\label{probl1}
\underset{\x\in\bcX}{\min} f(\x), \qquad\mbox{with}\quad f(\x)=\Esp_{\ubar{\theta}}\left[f_{\ubar{\theta}}(\x)\right],
\end{equation}
where {\bl the nonempty set} $\bcX \subseteq\rn$ {\bl of nonzero measure} is a feasible region; the values of the function $f:\rn\to\R$ are available only through $f_{\ubar{\theta}}$, a stochastically noisy version of $f$;  ${\ubar{\theta}}$ is a random variable whose probability distribution is possibly unknown; and $\Esp_{\ubar{\theta}}$ denotes the {\it expectation} with respect to ${\ubar{\theta}}$. Problem~\eqref{probl1} arises, for example, in statistical machine learning~\cite{bottou2018optimization} where $\ubar{\theta}$ is a random variable describing a given sample set while $f_{\ubar{\theta}}$ represents the composition of a loss function and a prediction function. {\bl We particularly target settings where the problem dimension $n>100$, which is large in the context of derivative-free methods for stochastic optimization.}

Significant algorithmic and theoretical advances have recently been made in stochastic DFO with the aim of developing provable algorithms to solve Problem~\eqref{probl1}, some of which are either model-based or perform an estimation of the gradient of $f$. Direct-search approaches are an especially promising option in settings where model construction or gradient approximation can be computationally expensive, or when one might not have any information about the existence of derivatives~\cite{audet2019stomads}. Inspired by the unconstrained framework of the stochastic mesh adaptive direct-search (StoMADS) algorithm proposed by Audet et al.~\cite{audet2019stomads} (later extended to the noisy constraints case~\cite{dzahini2020constrained}), Dzahini~\cite{dzahini2020expected} recently introduced a broad class of SDDS algorithms with convergence and expected complexity analyses. At each iteration~$k$, an SDDS algorithm attempts to improve the current best solution, referred to as the {\it incumbent}, by comparing  estimates of the values of~$f$ at the incumbent solution and at most $n+1$ trial points. The accuracy of these estimates, which are required to be $\beta$-probabilistically $\ef$-accurate, is dynamically controlled by means of a {\it stepsize}~$\dk$. More precisely, in addition to the fact that the per iteration number of search directions of SDDS grows linearly with the dimension $n$, it is also proved that for convergence purposes the computation of such estimates requires $\Omega\left(\dk^{-4}/(1-\sqrt{\beta})\right)$ function evaluations~\cite{audet2019stomads,chen2018stochastic}, with a resulting overall evaluation complexity that affects the efficiency of these algorithms.

Having observed through numerical experiments that the number of function evaluations required by stochastic optimization methods such as SDDS~\cite{audet2019stomads,dzahini2020constrained} and stochastic trust-region~\cite{chen2018stochastic} algorithms to compute estimates and obtain satisfactory results in practice is significantly small{\bl er than} $\Omega\left(\dk^{-4}/(1-\sqrt{\beta})\right)$, 
an important step toward the improvement of SDDS algorithms is the significant reduction in the number of search directions, as was the case in~\cite{RobRoy2022RedSpace}. Indeed, after revisiting the analysis of direct search based on probabilistic properties~\cite{GrRoViZh2015,GraRoViZh2019}, Roberts and Royer~\cite{RobRoy2022RedSpace} recently proposed a  randomized direct-search framework for the unconstrained optimization of deterministic objectives where the search directions are chosen within random subspaces of~$\rn$. Inspired by existing derivative-based~\cite{carfowsha2022aRandomised,carfowsha2022Rand,shao2022Thesis} and derivative-free model-based~\cite{CRsubspace2021} methods that achieve scalability using random subspace strategies, the proposed class of DARS algorithms generate at each iteration exactly~$m$ search directions in~$\rn$ by combining vectors in~$\rp$ with random matrices~$\QkRandom\in\rnp$ defining the random subspaces, where~$m$ and~$p$ are significantly less than~$n$ and chosen independently of~$n$. Although DARS algorithms do not use any gradient information, their analysis to-date assumed that the objective function is continuously differentiable since in~\cite{RobRoy2022RedSpace} the gradient is central to the definition of the {\it well alignment}\footnote{We note that the {\it probabilistic well alignment} notion inspired by~\cite[Definition~3.2]{DzaWildSub2022} and introduced in Definition~\ref{wellAlignment}  not only is gradient-free compared with prior works but also is more general in that it is not specific to a given fixed direction such as a gradient vector.} of the random subspaces, a notion inherited from~\cite{CRsubspace2021} and references therein, and which is the cornerstone of all the analyses. This last observation raises the important and relevant question of whether random subspace algorithms inspired by the above methods can be analyzed without a differentiability assumption on the objective function, using, for example, Clarke nonsmooth analysis~\cite{Clar83a}.

Of the cited methods, the only ones that achieve scalability making use of random subspace strategies  not only are unconstrained but also  either are for deterministic objectives~\cite{carfowsha2022aRandomised,carfowsha2022Rand,HarRobRoy25Expected,Kozak2023,kozak2021stochastic,RobRoy2022RedSpace,NEURIPS2023_7429f4c1} or are model based~\cite{DzaWildSub2022}, utilizing the aforementioned well alignment notion and thereby gradient information. In other words, to our knowledge, no such method exists in a stochastic framework with convergence guarantees without assuming the objective function to be continuously differentiable. More particularly, there does not exist any constrained direct-search method for stochastic optimization in random subspaces. Thus, inspired by~\cite{audet2019stomads,carfowsha2022aRandomised,dzahini2020expected,DzaWildSub2022,RobRoy2022RedSpace}, we address these gaps in the present work by introducing the StoDARS class of constrained direct-search algorithms for stochastic optimization in random subspaces, in which scalability is achieved through the use of random search directions generated in random subspaces of~$\rn$. {\bl Unlike prior works allowing the selection of these subspaces by means of random matrices~$\QkRandom\in\rnp$---called Johnson--Lindenstrauss transforms (JLTs), that is, satisfying~\eqref{JLTEquat} below---with $p\ll n$, from various matrix ensembles~\cite{Database2003Ach,DzaWildHashing2022,IndykMotwani1998,KaNel2014SparseLidenstrauss}, only JLTs obtained from Haar-distributed orthogonal random matrices are allowed in StoDARS for many reasons. Indeed, the sequence of search directions generated by Haar-based JLTs is, with probability one, dense in the unit sphere of~$\rn$, which is crucial for the theoretical guarantees. Moreover, the resulting sets of search directions always form a {\it positive spanning set} (PSS; see, e.g.,~\cite[Definition~6.1]{AuHa2017}) in the selected subspace at each iteration of the algorithm---unlike those from Gaussian-based~\cite{IndykMotwani1998} in general or Hashing(-like)~\cite{Database2003Ach,DzaWildHashing2022,KaNel2014SparseLidenstrauss} JLTs---a property that is essential for practical efficiency.}
%Unlike prior works where these subspaces can be selected through the range of random matrices~$\QkRandom\in\rnp$, with $p\ll n$, from various matrix ensembles~\cite{Database2003Ach,DzaWildHashing2022,IndykMotwani1998,KaNel2014SparseLidenstrauss}, the matrices~$\QkRandom^{\top}$ used by StoDARS are not only Johnson--Lindenstrauss transforms (JLTs), that is, satisfying~\eqref{JLTEquat} below, thus allowing~$p$ to be independent of~$n$, but they are obtained from Haar-distributed orthogonal random matrices so that the sequence of resulting search directions is dense in~$\rn$ with probability one, as required by theory and needed for practical efficiency. 

Before summarizing the contributions of this work, we note that one of the results that is central to the complexity analyses of direct-search algorithms in both the deterministic~\cite{Vicente2013} and stochastic~\cite{dzahini2020expected} cases is that at least one of the directions of a PSS makes an acute angle with the negative gradient (at a given point) of a continuously differentiable function~$f$. The first contribution of this work corresponds to probabilistic extensions to random subspace directions (that do not form a PSS) of the latter property through Propositions~\ref{propKappaD1} and~\ref{propKappaD2} and Lemma~\ref{lemKappaD3}. We note that random subspace direct-search methods are particularly appealing in high-dimensional blackbox optimization because they can substantially reduce the number of polling directions, and hence expensive function evaluations required per iteration. In contrast with full-space direct-search frameworks whose poll sets must contain at least $n+1$ directions in dimension~$n$, random subspace frameworks generate only~$m$ directions inside a $p$-dimensional subspace with~$m$ and~$p$ chosen independently of~$n$. Reducing the number of polling directions required on any iteration is beneficial, for example, when one can do concurrent function evaluations but one does not have the computational resources to do a full set of $n+1$ concurrent evaluations; such cases can arise when the function evaluation requires parallel computing resources for even a single evaluation. 

Our analysis shows that this reduction in directional sampling does not come at the expense of worst-case guarantees. Indeed, by making use of the aforementioned results while assuming the function~$f$ differentiable, and leveraging a supermartingale-based framework introduced in~\cite{blanchet2016convergence}, the analysis of StoDARS demonstrates that although using the aforementioned random subspace directions instead of PSSs, the expected complexity of the method is similar to that of stochastic trust-region~\cite{blanchet2016convergence,chen2018stochastic,DzaWildSub2022}, SDDS~\cite{dzahini2020expected}, and stochastic line-search~\cite{paquette2018stochastic} up to constants {\bbl (Theorem~\ref{complexResult})}. StoDARS, which uses random estimates required to be sufficiently accurate with large but fixed probability, is therefore the first constrained random subspace variant of SDDS algorithms and an extension of DARS methods to stochastically noisy objectives. By dropping the differentiability assumption and making use of Clarke calculus~\cite{Clar83a}, this manuscript is, to our knowledge, the first to prove convergence of a subspace method to Clarke stationary points {\bbl (Theorem~\ref{ClarkeTheor})}, unlike prior analyses of random subspace algorithms~\cite{carfowsha2022aRandomised,carfowsha2022Rand,CRsubspace2021,DzaWildSub2022,RobRoy2022RedSpace} where the use of the gradient is crucial in both the selection of the subspaces and theory. In particular, by exploiting the ability of StoDARS to generate a dense set of search directions, the analysis of second-order behavior of MADS algorithms~\cite{AbAu06} using a second-order-like extension of the Rademacher's theorem-based definition of the Clarke subdifferential~\cite[Section~1.2]{Clar83a}, so-called generalized Hessian, is extended for the first time to the stochastic case {\bbl (Theorem~\ref{secondOrderClarkeTheor})}.
A separate line of development for nonsmooth optimization of a deterministic objective has recently been explored using randomized finite-differences and the Goldstein subdifferential \cite{NEURIPS2023_7429f4c1}. 

The manuscript is organized as follows. After recalling the framework of SDDS methods, Section~\ref{sec:SDSS} discusses strategies used for subspace selection and then provides theoretical properties of the sets of search directions used by StoDARS. Section~\ref{secThree} presents the general framework of StoDARS algorithms and explains how they result in a stochastic process. The almost sure convergence to zero of the sequence of stepsize parameters, referred to as {\it zeroth-order convergence}, is demonstrated in Section~\ref{sec4}, followed by an expected complexity analysis in Section~\ref{sec5}. An almost sure convergence result to Clarke stationary points is established in Section~\ref{sec6}, followed in Section~\ref{sec7} by an analysis of the second-order behavior of StoDARS. Numerical experiments are presented in Section~\ref{sec8}, followed by discussions and suggestions for future work.

\subsubsection*{Notation}
Throughout the manuscript, vectors are written in lowercase boldface (e.g., $\vb\in\rn, n\geq 2$), and matrices are written in uppercase boldface (e.g., $\Ub\in\rnp$). Unless otherwise stated, all random quantities are underlined (e.g., $\ubar{z}$, $\vrandom$, $\Urandom$) and are defined on the same probability space $(\Omega,\F,\pr)$. The nonempty set~$\Omega$ referred to as the {\it sample space} is a collection of elements~$\omega$ called {\it sample points}, while subsets of~$\Omega$ are called {\it events}. The collection, $\F$, of these events is called a $\sigma$-algebra; and the finite measure~$\pr$ defined on the {\it measurable space} $(\Omega,\F)$, with $\prob{\Omega}=1$, is a {\it probability measure}. A {\it filtration} is an increasing subsequence $\accolade{\mathcal{H}_k}_k$ of $\sigma$-algebras of~$\F$. By denoting by~$\mathscr{B}(\rn)$ the Borel $\sigma$-algebra generated by open sets of~$\rn$, a real-valued random variable~$\ubar{z}$ is a measurable map $\ubar{z}:(\Omega,\F,\pr)\to (\R,\mathscr{B}(\R))$, where measurability means that $\accolade{\ubar{z}\in \mathscr{E}}:=\accolade{\omega\in\Omega:\ubar{z}(\omega)\in\mathscr{E}}=:\ubar{z}^{-1}(\mathscr{E})\in\F$ for all $\mathscr{E}\in \mathscr{B}(\R)$~\cite{bhattacharya2007basic}. A real-valued random matrix is a matrix-valued random variable, that is, a matrix whose entries are real-valued random variables. $\normii{\ub}:=\left(\sum_{i=1}^n u_i^2\right)^{1/2}$ denotes the Euclidean norm of a vector $\ub:=(u_1,\dots,u_n)^{\top}\in\rn$.  $\norme{\cdot}$ is used for the norm of a matrix and is assumed 
%GAIL - bu "supposed to be" do you mean that it really is not? or do you mean it is assumed to be?
to be {\it consistent} with the Euclidean norm, that is, $\normii{\Q\ub}\leq \norme{\Q}\normii{\ub}$, which holds for both spectral and Frobenius matrix norms. $\ub^{\top}\vb=\scal{\ub}{\vb}:=\sum_{i=1}^n u_i v_i$ denotes the {\it scalar product} or {\it dot product} of the vectors $\ub$ and $\vb:=(v_1,\dots,v_n)^{\top}\in\rn$. $K_{\infty}$ denotes the random set of all {\it unsuccessful iterations} of a given algorithm. A realization of a random set~$K$ is denoted by~$\mathcal{K}$; that is, $K(\omega)=\mathcal{K}$ for some~$\omega\in\Omega$. {\bl For $a\leq b$, $\intbracket{a,b}:=[a,b]\cap\Z$ is the set of integers in the interval $[a,b]$.}

\section{Stochastic directional direct-search algorithm and random subspace polling}
\label{sec:SDSS}
The directional direct-search algorithm proposed in this work uses search directions constructed in random subspaces of~$\rn$, defined by means of specific random matrices. For the proposed method to be efficient and to meet for convergence needs, these subspaces and their resulting sets of search directions must be probabilistically rich with sufficiently high, but fixed, probabilities, to allow a better exploration of~$\rn$. This section briefly recalls the SDDS framework introduced in~\cite{dzahini2020expected}, which forms the foundation of the proposed method summarized in Algorithm~\ref{algoStoDARS}. We then discuss strategies used for random subspace selection and subspace {\it polling} and provide useful theoretical properties of the resulting {\it polling sets} used by Algorithm~\ref{algoStoDARS}.

\subsection{Full-space unconstrained stochastic directional direct-search method}
SDDS algorithms were introduced in~\cite{dzahini2020expected} to solve the stochastic optimization problem~\eqref{probl1} in the particular case $\bcX=\rn$. These algorithms use a framework inspired by that of StoMADS~\cite{audet2019stomads} and the broad class of methods presented in~\cite{Vicente2013} using a {\it generating set search} method~\cite{KoLeTo03a}, such as pattern search and generalized pattern search~\cite{AuDe03a}. Each iteration of SDDS algorithms is composed of two main steps: an optional {\it search}, not shown in~\cite[Algorithm~1]{dzahini2020expected}, and not detailed in Algorithm~\ref{vagueDirAlgo} for simplicity, and the ``more disciplined''~\cite{LaMeWi2019} {\it poll} on which  the theoretical analyses rely. Broadly speaking, each poll step of an SDDS method generates a finite set of {\it poll points} by taking an {\it incumbent solution}~$\xk\in\rn$ and adding terms of the form~$\dk\di^n$, where $\dk>0$ is the {\it stepsize} and $\di^n\in\rn$ is a nonzero vector from a finite set $\mathbb{D}^n_k\subset\rn$ of {\it poll directions}. More precisely, $\mathbb{D}^n_k$ is a PSS, meaning that any vector in~$\rn$ can be written as a nonnegative linear combination of elements in~$\mathbb{D}^n_k$~\cite[Definition~6.1]{AuHa2017} or, equivalently, the cosine measure $\kappa(\mathbb{D}^n_k)>0$, where 
%for a given $\mathbb{D}\subset\rn$,
\begin{equation}\label{cosineMeasureDef}
\kappa(\mathbb{D}^q):=\underset{\wb\in \R^q,\wb\neq \bm{0}}{\min}\ \underset{\di^q\in \mathbb{D}^q}{\max}\frac{\scal{\wb}{\di^q}}{\normii{\wb}\normii{\di^q}}\quad\mbox{with}\quad \mathbb{D}^q\subset\R^q.
\end{equation}
Since exact function values $f(\x)$ are unavailable in the SDDS framework, then by using the noisy function~$f_{\ubar{\theta}}$ values, so-called $\ef$-accurate estimates $\fok\approx f(\xk)$ and $\fdk\approx f(\xk+\dk\di^n)$ are computed at each iteration, the definition of which is recalled in Definition~\ref{epsAccDef0} in the context of Algorithm~\ref{algoStoDARS}. It is then demonstrated in~\cite[Proposition~1]{dzahini2020expected} (see also Proposition~\ref{decreaseProp1} below) that a sufficient decrease $\fdk-\fok\leq -\gamma\ef\dk^2$, with $\gamma>2$, leads to a decrease in~$f$. In the latter case, the iteration is called {\it successful}, resulting in an update of the incumbent solution and an increase in the stepsize. Otherwise, the iteration is {\it unsuccessful}, and the incumbent solution is not updated while the stepsize is decreased. A new iteration is then initiated unless a stopping criterion is met.

\begin{algorithm}%[H] 
\caption{Simplified SDDS framework}
\label{vagueDirAlgo}
\textbf{[0] Initialization}\\
Choice of $\gamma>2, \ef>0$, and other algorithmic parameters, initialization of stepsize $\delta_0>0$, and initial point\\
$\x_0\in\rn$. \\
\textbf{[1] Search} (Optional)\\
\textbf{[2] Poll}\\
\hspace*{10mm}Choice of a positive spanning set $\mathbb{D}^n_k\subset\rn$.\\
%and comparison of the current solution $\xk$ with\\
\hspace*{10mm}Computation of $\ef$-accurate estimates $\fok\approx f(\xk)$ and $\fdk\approx f(\xk+\dk\di^n)$ using noisy function\\
\hspace*{10mm}evaluations.\\
\hspace*{10mm}If $\fdk-\fok\leq -\gamma\ef\dk^2$ for some $\di^n\in \mathbb{D}^n_k$ \emph{(successful)}\\
\hspace*{18mm}set $\xkun\gets\xk+\dk\di^n$ and $\dkun> \dk$.\\
\hspace*{10mm}Otherwise \emph{(unsuccessful)}\\ 
\hspace*{18mm}set $\xkun\gets\xk$ and $\dkun<\dk$.\\
\textbf{[3] Termination}\\
\hspace*{10mm}Stop if a criterion is met, or set $k\gets k+1$ and go to~\textbf{[1]}.
\end{algorithm}

% Let $p\in \intbracket{1,n}$, $m\geq p+1$ and $\pollD:=\accolade{\di^1,\di^2,\dots,\di^m}\subset\spOne$ be a PSS~\cite[Definition~6.1]{AuHa2017}.  Consider the random set $\pollDQ:=\accolade{\Qrandom\di:\di\in\pollD}=\accolade{\Qrandom\di^1,\Qrandom\di^2,\dots,\Qrandom\di^m}\subset\R^n$ of images by $\Qrandom$ of the directions in $\pollD$. Recall the following definition of the {\it cosine measure}~\cite{KoLeTo03a} of a set $\pollD\subset\spOne$:
% \[\kappa(\pollD):=\underset{\wb\in \rp,\wb\neq \bm{0}}{\min}\ \underset{\di\in \pollD}{\max}\frac{\scal{\wb}{\di}}{\normii{\wb}\normii{\di}}.\]

\subsection{Random subspace selection and subspace polling}\label{sec2Point2}

By assuming that the function $f$ is continuously differentiable, and under additional assumptions including those similar to Assumption~\ref{efkfAssump} below, the expected number of iterations required by Algorithm~\ref{vagueDirAlgo} to drive the norm of the gradient of~$f$ below a given threshold $\epsilon\in (0,1)$ is proved to be $O(\epsilon^{-2})$. {\bl This result is} based on the observation that given $\vi\in\rn$ (such as the negative gradient of $f$ at some point~$\x$), at least one direction of any PSS $\mathbb{D}^n\subset\rn$ makes an acute angle with $\vi$, or equivalently (see, e.g.,~\cite[Eq.~(37) and Lemma~2]{dzahini2020expected} and~\cite[Theorem~6.5]{AuHa2017})
\begin{equation}\label{cosineDescentDkn}
\mbox{for any}\ \vi\in\rn,\quad\exists \di_{\star}^n(\vi)\in \mathbb{D}^n\subset\rn\quad\mbox{such that}\quad \kappa(\mathbb{D}^n)\normii{\vi}\normii{\di^n_{\star}(\vi)}\leq \scal{\vi}{\di^n_{\star}(\vi)}.
\end{equation}
Thus, although not using any first-order information, SDDS shares a similar worst-case complexity bound with many existing derivative-free methods, such as stochastic trust-region~\cite{blanchet2016convergence,chen2018stochastic} and stochastic line-search~\cite{paquette2018stochastic}, in terms of a dependence in~$\epsilon$, and also matches that of deterministic directional direct-search methods~\cite{Vicente2013}. Unfortunately, however, and as with all other direct-search methods using PSSs for polling,  the performance of SDDS quickly deteriorates as the dimension~$n$ of the problem gets larger; the reason is  that even the smallest PSS
%, a so-called {\it minimal positive basis} (see, e.g.,~\cite[Definition~6.3 and Corollary~6.1]{AuHa2017}), 
has a cardinality of~$n+1$, resulting in many function evaluations in Algorithm~\ref{vagueDirAlgo} as $n$ grows.

To significantly reduce the number of poll directions, Roberts and Royer~\cite{RobRoy2022RedSpace} recently introduced a new framework, for deterministic optimization, of {\it direct search in reduced spaces}, where poll directions are generated within subspaces selected through random matrices. More precisely, their method (summarized in~\cite[Algorithm~3.1]{RobRoy2022RedSpace}) uses at each iteration~$k$ a polling set $\accolade{\Qk\di^p:\di^p\in\pollD_k}\subset\rn$ that belongs to the lower-dimensional subspace $\accolade{\Qk\di:\di\in\rp}$ of~$\rn$, obtained by combining directions~$\di^p$ from a subset $\pollD_k\subset\rp$ with a realization $\Qk\in\rnp$ of a random matrix~$\QkRandom$, where $p\leq n$ can be chosen independently of~$n$. Moreover, inspired by recent works in a model-based setting~\cite{carfowsha2022Rand,CRsubspace2021,shao2022Thesis}, the matrix~$\QkRandom$ is required to be {\it probabilistically well aligned}; that is, it must be able to capture a portion, bounded away from zero, of the unknown full-space derivative information. 

Inspired by~\cite{carfowsha2022Rand,CRsubspace2021,shao2022Thesis} and the the recent work of Dzahini and Wild~\cite{DzaWildSub2022} on {\it trust-region methods in random subspaces for stochastic optimization},
%SW:I do not see why this needs to be called out, yo state preciselt the definitio below. 
%(see~\cite[Definition~3.2]{DzaWildSub2022} more precisely), 
the following definition of well alignment, a little different from~\cite[Definition~3.2]{RobRoy2022RedSpace}, is considered throughout the manuscript.

\begin{definition}\label{wellAlignment}
For fixed $\alpha_Q,\beta_Q\in (0,1/2)$, 
%a sequence $\accoladekinN{\QkRandom}$ of independent 
a random matrix $\Qrandom\in\rnp$ is called $(1-\beta_Q)$-probabilistically $\alpha_Q$-well aligned if, for all $\vb\in\rn$, $\prob{\normii{\Qrandom^{\top}\vb}\geq \alpha_Q\normii{\vb}}\geq 1-\beta_Q$. Such matrices are referred to as $WAM(\alpha_Q,\beta_Q)$.
\end{definition}

Motivated by the property in~\eqref{cosineDescentDkn}, which represents the cornerstone of the complexity analyses of directional direct-search for both deterministic~\cite{Vicente2013} and stochastic~\cite{dzahini2020expected} optimization, and in order to establish its probabilistic variants 
%of~\eqref{cosineDescentDkn} 
through Propositions~\ref{propKappaD1} and~\ref{propKappaD2} that will allow the results of Lemmas~\ref{AkLemma1} and~\ref{smallLemmaHaar}, which are crucial for the complexity result derived in Theorem~\ref{complexResult}, the polling sets used by Algorithm~\ref{algoStoDARS} are defined inspired by~\cite{RobRoy2022RedSpace},  assuming that 
$\pollD_k\subset\rp$ is a PSS. More precisely, these polling sets are defined as $\mathbb{U}_k^n:=\accolade{\Ukp\di^p:\di^p\in\pollD_k}$, where $\Ukp\in\rnp$ represents a realization of the first~$p$ columns of a Haar-distributed orthogonal random matrix, satisfying $\Qk=\sqrt{\frac{n}{p}}\Ukp$ as motivated by Theorem~\ref{JLforHaarTheor} below. 
%We note that throughout the manuscript, vectors in~$\pollD_k$ are simply denoted by~$\di$ instead of~$\di^p$. 
The result of Proposition~\ref{propKappaD1} is established under the following assumption.

\begin{assumption}\label{JltAss1}{\bl (Inspired by~\cite[Assumptions~2 and~3]{DzaWildSub2022} and~\cite{CRsubspace2021})}
The random matrix $\Qrandom\in\rnp$ is a $WAM(\alpha_Q,\beta_Q)$ for fixed $\alpha_Q,\beta_Q\in (0,1/2)$ and satisfies $\norme{\Qrandom}\leq\Qmax$ almost surely for some $\Qmax>0$. 
%and $\prob{\normii{\QrandomTranspose\bm{v}}\geq \alpha_Q\normii{\bm{v}}}\geq 1-\beta_Q$ for all $\bm{v}\in\rn$, for some $\alpha_Q\in (0,1)$ and $\beta_Q\in (0,1/2)$.
\end{assumption}

\begin{proposition}\label{propKappaD1}
Assume that $\Qrandom\in\rnp$ satisfies Assumption~\ref{JltAss1}, and let $\pollD\subset \rp$ be a PSS. For any $\bm{v}\in \rn$, 
%there exists a direction $\di_{\star}\in\pollD$ such that 
\begin{equation}\label{kdprob1}
\prob{\accolade{\exists\ubar{\di}_{\star}^p\in\pollD:\frac{\alpha_Q}{\Qmax}\kappa(\pollD)\normii{\bm{v}}\normii{\Qrandom\ubar{\di}_{\star}^p}\leq \scal{\bm{v}}{\Qrandom\ubar{\di}_{\star}^p}}}\geq 1-\beta_Q.
\end{equation}
In other words, with high probability, there always exists a random direction $\Qrandom\ubar{\di}_{\star}^p\in\accolade{\Qrandom\di^p:\di^p\in\pollD}\subset\rn$ that makes an acute angle with $\bm{v}$. The same result holds in the case of a random PSS $\pollDrandom\subset\rp$; that is,
\begin{equation}\label{kdprob1prime}
\prob{\accolade{\exists\ubar{\di}_{\star}^p\in\pollDrandom:\frac{\alpha_Q}{\Qmax}\kappa(\pollDrandom)\normii{\bm{v}}\normii{\Qrandom\ubar{\di}_{\star}^p}\leq \scal{\bm{v}}{\Qrandom\ubar{\di}_{\star}^p}}}\geq 1-\beta_Q.
\end{equation}
\end{proposition}

\begin{proof}
Let $\bm{v}\in \rn$ be given. Assume that the event $\mathscr{E}_1:=\accolade{\normii{\QrandomTranspose\bm{v}}\geq \alpha_Q\normii{\bm{v}}}$ occurs, which happens with a probability of at least $1-\beta_Q$, and let $\omega\in\mathscr{E}_1$ be arbitrary. Since $\pollD$ is a PSS, it contains a direction $\di_{\star}^p=\ubar{\di}_{\star}^p(\omega)$ that makes an acute angle with the deterministic vector $\Qrandom(\omega)^\top \bm{v}\in \rp$. Formally thanks to~\eqref{cosineDescentDkn}, this means that 
\begin{equation}\label{kdeq1}
\kappa(\pollD)\normii{\Qrandom(\omega)^\top \bm{v}}\normii{\di_{\star}^p}\leq \scal{\Qrandom(\omega)^\top \bm{v}}{\di_{\star}^p}=\scal{\bm{v}}{\Qrandom(\omega)\di_{\star}^p}.
\end{equation}
It follows from~\eqref{kdeq1} and the inequality $\normii{\Qrandom(\omega)^\top\bm{v}}\geq \alpha_Q\normii{\bm{v}}$ that $\alpha_Q\kappa(\pollD)\normii{\bm{v}}\normii{\di_{\star}^p}\leq \scal{\bm{v}}{\Qrandom(\omega)\di_{\star}^p}$, which leads to $\frac{\alpha_Q}{\Qmax}\kappa(\pollD)\normii{\bm{v}}\normii{\Qrandom(\omega)\di_{\star}^p}\leq \scal{\bm{v}}{\Qrandom(\omega)\di_{\star}^p}$, given that $\normii{\Qrandom(\omega)\di_{\star}^p}\leq \Qmax\normii{\di_{\star}^p}$. We have therefore proved that $\mathscr{E}_1\subseteq\mathscr{E}_2:=\accolade{\exists\ubar{\di}_{\star}^p\in\pollD:\frac{\alpha_Q}{\Qmax}\kappa(\pollD)\normii{\bm{v}}\normii{\Qrandom\ubar{\di}_{\star}^p}\leq \scal{\bm{v}}{\Qrandom\ubar{\di}_{\star}^p}}$, whence $\prob{\mathscr{E}_2}\geq \prob{\mathscr{E}_1}\geq 1-\beta_Q$, leading to~\eqref{kdprob1}. 

In the case of a random PSS $\pollDrandom$, the same reasoning  applies with $\pollD$ and $\di_{\star}^p$ replaced with $\pollDrandom(\omega)$ and $\di_{\star}^p(\omega)$ (which are deterministic), respectively, thus leading to~\eqref{kdprob1prime}, and the proof is completed.
\end{proof}

%-------------------------------------------------------------
%However, many other considerations motivated this choice. 
The poll directions used by Algorithm~\ref{algoStoDARS} are generated by using Haar-distributed orthogonal random matrices. Recall that a matrix $\bm{O}\in \rnn$ is {\it orthogonal} if and only if $\bm{O}^{\top}\bm{O}=\bm{O}\bm{O}^{\top}=\bm{I}_n$ or equivalently if its columns form an orthonormal basis of $\rn$~\cite[Section~1.1]{MeckesElizHaar2019}. Throughout the manuscript, the compact Lie group of $n\times n$ orthogonal matrices over~$\R$ will be denoted by $\mathbb{O}(n)$. 
%We also recall that $\mathbb{O}(n)$ is a group under multiplication, and that this group admits a compact Lie group structure, that is, it is also a differentiable manifold such that the multiplication operation $(\bm{A}, \bm{B})\mapsto \bm{A}\bm{B}$ and the inversion operation $\bm{A}\mapsto \bm{A}^{-1}$ are smooth maps. 
A {\it uniform random element} of $\mathbb{O}(n)$ is a random $\Urandom\in \mathbb{O}(n)$ whose distribution is translation invariant; that is, for any fixed $\bm{O}\in \mathbb{O}(n)$, $\bm{O}\Urandom\equald \Urandom\bm{O}\equald \Urandom$~\cite[Section~1.2]{MeckesElizHaar2019}, where ``$\equald$'' is used for the equality in distribution. In particular, there exists a unique translation-invariant probability measure on $\mathbb{O}(n)$ called Haar measure on~$\mathbb{O}(n)$~\cite[Theorem~1.4]{MeckesElizHaar2019}. Any random matrix whose distribution is Haar measure will be called Haar-{\it distributed matrix}. 
\begin{remark}
{\bl One may naturally wonder why we chose to use \emph{Haar}-distributed orthogonal random matrices.} Indeed, it is well known from probability theory that a sequence~$\accoladekinN{\vrandom_k}$ of i.i.d.\ random vectors~$\vrandom_k\sim\mathcal{U}(\snOne)$ is dense on the unit sphere~$\snOne$ with probability one. Thus, in order for Algorithm~\ref{algoStoDARS} to generate dense sets of poll directions although operating in subspaces, not only for a better exploration of the full-space~$\rn$ but especially in order to establish the results of Theorems~\ref{ClarkeTheor} and~\ref{secondOrderClarkeTheor} inspired by~\cite{AuDe2006} and~\cite{AbAu06}, respectively, \emph{Haar}-distributed orthogonal matrices are used for polling. This choice is also based on the observation from Propositions~\ref{unitHaarProp} and~\ref{densityProp} that \emph{Haar}-distributed orthogonal matrices, when combined with vectors from~$\spOne$ in the appropriate way, result in random vectors distributed according to~$\mathcal{U}(\snOne)$. On the other hand, the need for~$p$ to be chosen independently of~$n$, resulting in a minimum of~$m=p+1$ poll directions in~$\mathbb{U}_k^n:=\{\Ukp\di^p:\di^p\in \pollD_k\}$ whenever $\abs{\pollD_k}=p+1$, along with the observation from Theorem~\ref{JLforHaarTheor} that \emph{Haar}-distributed-based matrices satisfy Assumption~\ref{JltAss1} while meeting the above requirement on~$p$, justifies the specific choice of such matrices in the remainder of the manuscript.
\end{remark}

Lemma~\ref{HaarLemma} and Theorem~\ref{LeviTheorem}, stated next, provide useful results for the proof of Theorem~\ref{JLforHaarTheor}.
\begin{lemma}\label{HaarLemma}\cite[Section~2.1]{MeckesElizHaar2019}
Let $\Urandom:=(\UrandomScal_{ij})_{1\leq i,j\leq n}\in \mathbb{O}(n)$ be \emph{Haar}-distributed. Then, the entries of $\Urandom$ are identically distributed. Moreover, $\E{\UrandomScal_{11}}=0$ and $\E{\UrandomScal_{11}^2}=\frac{1}{n}$.
\end{lemma}

The following theorem states that any Lipschitz function~$F$ on the Euclidean sphere~$\snOne$ concentrates well. It is proved by arguing that ``any nonlinear function must concentrate at least as strongly as a linear function''~\cite{vershynin2018HDP}, which is shown by comparing the {\it sublevel sets} $\accolade{\x:F(\x)\leq \mathfrak{c}}$ using a  geometric principle: an {\it isoperimetric inequality}~\cite[Theorems~5.1.5 and~5.1.6]{vershynin2018HDP}.
\begin{theorem}\label{LeviTheorem}(\cite[Theorem~5.1.4 and Exercise~5.1.12]{vershynin2018HDP})
Let $F$ be an $L_F$-Lipschitz function on the unit sphere $\mathbb{S}^{n-1}$, and let $\xiRandomVec\sim \mathcal{U}(\mathbb{S}^{n-1})$. Then there exists an absolute constant $\eta>0$ such that for every $t\geq 0$
\[\prob{\abs{F\left(\xiRandomVec\right)-\E{F\left(\xiRandomVec\right)}}\geq t}\leq 2\exp{\left(-\frac{\eta n t^2}{L_F^2}\right)}. \]
\end{theorem}

The result stated below in Theorem~\ref{JLforHaarTheor} shows how random matrices satisfying Assumption~\ref{JltAss1} and used for subspace polling in Algorithm~\ref{algoStoDARS} can be made available. It also shows that the number of directions in the poll sets $\mathbb{U}_k^n$ can be chosen independently of~$n$. It is inspired by the proof of the Johnson--Lindenstrauss lemma~\cite[Lemma~6.1]{MeckesElizHaar2019}, which states that after proper rescaling, the pairwise Euclidean distances between \scalebox{0.7}{$\aleph$} points of~$\rn$ projected onto a random subspace of dimension on the order of~$\log(\scalebox{0.7}{$\aleph$})$ hardly changes. This lemma is recalled next using the notation of the manuscript to emphasize how, unlike Theorem~\ref{JLforHaarTheor}, it does not straightforwardly show  the way the resulting matrix satisfies Assumption~\ref{JltAss1}.

\begin{lemma}\cite[Lemma~6.1]{MeckesElizHaar2019}
There exist absolute constants~$c$ and~$C$ such that the following holds. Let $\x_1,\x_2,\dots,\x_{\scalebox{0.7}{$\aleph$}}\in \rn$, and $\Urandom_{p\times n}\in \rpn$ consisting of the first~$p$ rows of the \emph{Haar}-distributed random matrix $\Urandom:=(\UrandomScal_{ij})_{1\leq i,j\leq n}\in \mathbb{O}(n)$. Fix $\varepsilon>0$ and let $p:=\frac{a\log(\scalebox{0.7}{$\aleph$})}{\varepsilon^2}$. With probability at least~$1-C\scalebox{0.7}{$\aleph$}^{2-\frac{ac}{4}}$, the matrix $\QrandomTranspose:=\sqrt{\frac{n}{p}}\Urandom_{p\times n}$ satisfies
\[(1-\varepsilon)\normii{\x_i-\x_j}^2\leq \normii{\QrandomTranspose(\x_i-\x_j)}^2\leq(1+\varepsilon)\normii{\x_i-\x_j}^2\quad\mbox{for all}\ i,j\in\intbracket{1,\scalebox{0.7}{$\aleph$}}.\]
\end{lemma}

Although the proof of Theorem~\ref{JLforHaarTheor} is derived from that of~\cite[Lemma~6.1]{MeckesElizHaar2019}, {\bl we provide additional details, especially regarding the choice of~$p$ along with minor details for the sake of clarity and a better understanding, in the appendix}.
\begin{theorem}\label{JLforHaarTheor}(Inspired by~\cite[Lemma~6.1]{MeckesElizHaar2019})
Let $\varepsilon_Q, \beta_Q\in (0,1)$ and $\Urandom:=(\UrandomScal_{ij})_{1\leq i,j\leq n}\in \mathbb{O}(n)$ be a \emph{Haar}-distributed random matrix. Let $\Urandom_{p\times n}\in \rpn$ be the matrix consisting of the first $p$ rows of $\Urandom$, and define $\QrandomTranspose:=\sqrt{\frac{n}{p}}\Urandom_{p\times n}$. There exists an absolute constant $\eta>0$ such that if $p\geq  \frac{4}{\eta}\varepsilon_Q^{-2}\log\left(\frac{2}{\beta_Q}\right)$, then
\begin{equation}\label{JLTEquat}
\prob{(1-\varepsilon_Q)\normii{\x}\leq \normii{\QrandomTranspose\x}\leq(1+\varepsilon_Q)\normii{\x}}\geq 1-\beta_Q\quad\mbox{for all}\ \x\in \rn.
\end{equation} 
\end{theorem}
\begin{proof}
{\bbl See Appendix~\ref{theor22}.}
\end{proof}

Among the various constructions of Haar measure on~$\mathbb{O}(n)$, including the Riemannian perspective, the explicit geometric construction, the inductive construction and the Gauss--Gram--Schmidt approach~\cite[Section~1.2]{MeckesElizHaar2019} (see also~\cite{MezzadriHaar2007} and references therein), the latter  not only happens to be ``the most commonly used, but also the easiest to implement on a computer''~\cite{MeckesElizHaar2019}. It  is described as follows. Let $\XfracRandom = \QfracRandom\RfracRandom$ be the ``QR-decomposition'' obtained via the Gram--Schmidt process of a random matrix $\XfracRandom\in \rnn$ from the real Ginibre ensemble, that is, with i.i.d. standard Gaussian entries. From a theoretical point of view,~$\QfracRandom$ is Haar-distributed, as demonstrated in~\cite[Section~1.2]{MeckesElizHaar2019}. However, as discussed in~\cite{MezzadriHaar2007} (and  observed by Edelman and Rao~\cite{EdelmanRao2005Haar} in the case of {\it unitary matrices}), unlike Householder {\it reflections}~\cite[Theorem~4]{MezzadriHaar2007}, algorithms adopting the above Gram--Schmidt approach are not numerically stable; that is, the output~$\QfracRandom$ might not be Haar-distributed in practice. The following correction (see, e.g.,~\cite{Stewart1980Haar} for a proof) is commonly used~\cite[Algorithm~3]{BhQPMADS2015}. Let {\bl $\DfracRandom$} be the diagonal matrix with elements ${\bl\DfracRandomScal_{ii}}=\sign(\RfracRandomScal_{ii})$, where $\RfracRandom = (\RfracRandomScal_{ij})_{1\leq i,j\leq n}$. Then~$\UfracRandom:=\QfracRandom{\bl\DfracRandom}$ is Haar-distributed on~$\mathbb{O}(n)$.

The following result is a consequence of Proposition~\ref{propKappaD1} in the specific case where $\Qrandom$ is constructed by using Haar measure and given by Theorem~\ref{JLforHaarTheor}. It is used together with Lemma~\ref{substLemma} to derive in Proposition~\ref{propKappaD2} another stochastic variant of the property in~\eqref{cosineDescentDkn}.

\begin{lemma}\label{lemKappaD3}
Let $\Qrandom\in\rnp$ be given by Theorem~\ref{JLforHaarTheor}, and let $\pollD\subset \rp$ be a PSS. For any $\bm{v}\in \rn$, the random vector $\hat{\di}^p(\vb,\qrandom)\in\argmax\accolade{\scal{\vb}{\Qrandom\di^p}:\di^p\in\pollD}$ satisfies
%there exists a direction $\di^p_{\star}\in\pollD$ such that 
\begin{equation}\label{kdpr1}
\prob{\alpha_Q\kappa(\pollD)\normii{\bm{v}}\leq \langle\vb, \Qrandom\hat{\di}^p(\vb,\qrandom)\rangle}\geq 1-\beta_Q.
\end{equation}
A similar result holds for any $\bm{v}\in \rn$ with ${\hat{\di}^p}(\vb,\qrandom,\drandom)\in\argmax\accolade{\scal{\vb}{\Qrandom\ubar{\di}^p}:\ubar{\di}^p\in\pollDrandom}$ in the case of a random PSS $\pollDrandom\subset\rp$; that is,
\begin{equation}\label{kdpr2}
\prob{\alpha_Q\kappa(\pollDrandom)\normii{\bm{v}}\leq \langle\vb, \Qrandom{\hat{\di}^p}(\vb,\qrandom,\drandom)\rangle}\geq 1-\beta_Q.
\end{equation}
\end{lemma}

\begin{proof}
By construction, $\Q=\sqrt{\frac{n}{p}}\Ub_{\ntp}$, where $\ubar{\Ub}_{\ntp}$ represents the first $p$ columns of a Haar-distributed orthogonal matrix and $\normii{\di^p}=1$ for all $\di^p\in\pollD$, whence $\norme{\Q}=\sqrt{\frac{n}{p}}=:\Qmax =\normii{\Q\di^p}$ since $\normii{\Ub_{\ntp}\di^p}=1$. Consequently, $\accolade{\exists\ubar{\di}_{\star}^p\in\pollD:\frac{\alpha_Q}{\Qmax}\kappa(\pollD)\normii{\bm{v}}\normii{\Qrandom\ubar{\di}_{\star}^p}\leq \scal{\bm{v}}{\Qrandom\ubar{\di}_{\star}^p}}=\accolade{\exists\ubar{\di}_{\star}^p\in\pollD:\alpha_Q\kappa(\pollD)\normii{\bm{v}}\leq \scal{\bm{v}}{\Qrandom\ubar{\di}_{\star}^p}}$, which is included in the event $\accolade{\alpha_Q\kappa(\pollD)\normii{\bm{v}}\leq \langle\vb, \Qrandom{\hat{\di}^p}(\vb,\qrandom)\rangle}$, leading to~\eqref{kdpr1} thanks to~\eqref{kdprob1}. The result in~\eqref{kdpr2} follows from~\eqref{kdprob1prime} using similar arguments by replacing $\pollD$ with~$\pollDrandom$ and ${\hat{\di}^p}(\vb,\qrandom)$ with ${\hat{\di}^p}(\vb,\qrandom,\drandom)$, and the proof is completed.
\end{proof}

\begin{lemma}\cite[Theorem~2.7-$(\ell)$]{bhattacharya2007basic}\label{substLemma}
Let $\ubar{\vb}$ and $\Mrandom$ be random maps into $(\mathcal{S}_1, \mathscr{S}_1)$ and $(\mathcal{S}_2, \mathscr{S}_2)$, respectively. Let $\mathcal{F}$ be a $\sigma$-algebra and $\psi$ be a measurable real-valued function on $\left(\mathcal{S}_1\times \mathcal{S}_2, \mathscr{S}_1\bigotimes\mathscr{S}_2\right)$. If $\ubar{\vb}$ is $\mathcal{F}$-measurable, $\sigma(\Mrandom)$ and $\mathcal{F}$ are independent, and $\E{\abs{\psi(\ubar{\vb},\Mrandom)}}<\infty$, then $\E{\psi(\ubar{\vb},\Mrandom)|\mathcal{F}}=\wp(\ubar{\vb})$ almost surely, where $\wp(\vb):=\E{\psi(\vb,\Mrandom)}$ for all $\vb\in\mathcal{S}_1$.
\end{lemma}

The result stated next, leading to Lemma~\ref{AkLemma1}, shows that the conclusions of Lemma~\ref{lemKappaD3} still hold conditionally to any $\sigma$-algebra~$\mathcal{F}$ that is independent of~$\Qrandom$, when $\vi$ is replaced with an~$\mathcal{F}$-measurable random vector~$\ubar{\vi}$. It is motivated by the need to derive in Lemma~\ref{smallLemmaHaar} a subspace variant of~\cite[Lemma~2]{dzahini2020expected}, which is key to the expected complexity result of Theorem~\ref{complexResult}.
\begin{proposition}\label{propKappaD2}
Let $\mathcal{F}$ be a $\sigma$-algebra independent of a random matrix $\Qrandom\in\rnp$ given by Theorem~\ref{JLforHaarTheor}, and let $\ubar{\vb}\in \rn$ be an $\mathcal{F}$-measurable random vector. Let $\pollD\subset \rp$ be a PSS, and consider the random vector defined by $\hat{\di}^p(\ubar{\vb},\qrandom)\in\argmax\accolade{\scal{\ubar{\vb}}{\Qrandom\di^p}:\di^p\in\pollD}$. Then 
%There exists a direction $\di_{\star}^p\in\pollD$ such that 
\begin{equation}\label{kdprob2}
\prob{\left.\alpha_Q\kappa(\pollD)\normii{\ubar{\vb}} \leq \langle\ubar{\vb}, \Qrandom\hat{\di}^p(\ubar{\vb},\qrandom)\rangle \right|\F}\geq 1-\beta_Q\quad\mbox{almost surely.}
\end{equation}
A similar result holds for any $\ubar{\vb}\in \rn$ with $\hat{\di}^p(\ubar{\vb},\qrandom,\drandom)\in\argmax\accolade{\scal{\ubar{\vb}}{\Qrandom\ubar{\di}^p}:\ubar{\di}^p\in\pollDrandom}$ in the case of a random PSS $\pollDrandom\subset\rp$ with $\Drandom$ independent of $\mathcal{F}$, 
%and for a direction ${\ubar{\di}^p}_{\star}\in \pollDrandom$; 
that is,
\begin{equation}\label{kdprob2prime}
\prob{\left.\alpha_Q\kappa(\pollDrandom)\normii{\ubar{\vb}} \leq \scal{\ubar{\vb} }{\Qrandom{\hat{\di}^p}(\ubar{\vb},\qrandom,\drandom)} \right|\F}\geq 1-\beta_Q\quad\mbox{almost surely.}
\end{equation}
\end{proposition}

% $\ubar{\vb}=\Qrandom\di_{\star}^p$
\begin{proof}
Let $\psi$ defined on $(\rn\times\rnp, \mathscr{B}(\rn)\bigotimes\mathscr{B}(\rnp))$ by $\psi(\ubar{\vb},\Qrandom)=\mathds{1}_{\accolade{\alpha_Q\kappa(\pollD)\normii{\ubar{\vb}} \leq \scal{\ubar{\vb}}{\Qrandom\hat{\di}^p(\ubar{\vb},\qrandom)}}}$ be a real-valued measurable function. Since $\E{\abs{\psi(\ubar{\vb},\Qrandom)}}<\infty$ trivially, it follows from Lemma~\ref{substLemma} that $\E{\psi(\ubar{\vb},\Qrandom)|\mathcal{F}}=\wp(\ubar{\vb})$ almost surely, where $\wp(\vb):=\E{\mathds{1}_{\accolade{\alpha_Q\kappa(\pollD)\normii{\vb} \leq \scal{\vb}{\qrandom\hat{\di}^p(\vb,\qrandom)}}}}=\prob{\alpha_Q\kappa(\pollD)\normii{\vb} \leq \langle\vb,\Qrandom\hat{\di}^p(\vb,\qrandom) \rangle}$ for all $\vb\in\rn$. It follows from~\eqref{kdpr1} that $\wp(\vb)\geq 1-\beta_Q$ for all $\vb\in\rn$ whence $\wp(\ubar{\vb})\geq 1-\beta_Q$ almost surely. This means that $\wp(\ubar{\vb}):=\E{\psi(\ubar{\vb},\Qrandom)|\mathcal{F}}=\prob{\left.\alpha_Q\kappa(\pollD)\normii{\ubar{\vb}} \leq \langle\ubar{\vb}, \Qrandom\hat{\di}^p(\ubar{\vb},\qrandom)\rangle \right|\F}\geq 1-\beta_Q$ almost surely, proving~\eqref{kdprob2}. To prove~\eqref{kdprob2prime}, consider the measurable real-valued function $\chi$ defined on $(\rn\times\R^{p\times (n+m)}, \mathscr{B}(\rn)\bigotimes\mathscr{B}(\R^{p\times (n+m)}))$ by $\chi\left(\ubar{\vb},\left[\Qrandom^{\top}\ \Drandom\right]\right)=\mathds{1}_{\accolade{\alpha_Q\kappa(\pollDrandom)\normii{\ubar{\vb}} \leq \scal{\ubar{\vb}}{\qrandom\hat{\di}^p(\ubar{\vb},\qrandom,\drandom)}}}$, satisfying $\E{\abs{\chi\left(\ubar{\vb},\left[\Qrandom^{\top}\ \Drandom\right]\right)}}<\infty$. $\sigma\left(\left[\Qrandom^{\top}\ \Drandom\right]\right)$ is independent of~$\F$ whence $\E{\left. \chi\left(\ubar{\vb},\left[\Qrandom^{\top}\ \Drandom\right]\right) \right|\F}=\tilde{\wp}(\ubar{\vb})$, where $\tilde{\wp}(\vb):=\prob{\alpha_Q\kappa(\pollDrandom)\normii{{\vb}} \leq \scal{{\vb}}{\Qrandom\hat{\di}^p({\vb},\qrandom,\drandom)}}$, thanks to Lemma~\ref{substLemma}. It follows from~\eqref{kdpr2} that $\tilde{\wp}(\vb)\geq 1-\beta_Q$ $\forall\vb\in\rn$, whence $\tilde{\wp}(\ubar{\vb}):=\E{\left. \chi\left(\ubar{\vb},\left[\Qrandom^{\top}\ \Drandom\right]\right) \right|\F}\geq 1-\beta_Q$ almost surely, showing~\eqref{kdprob2prime}, and the proof is completed.
\end{proof}
%GAIL - "showing erquatin x" seems odd to me. Do you mean proving?

\begin{figure}[ht]
\begin{multicols}{2}
\vspace*{-0.7cm}
\begin{adjustbox}{max totalsize={.9\textwidth}{.9\textheight},center}
\begin{tikzpicture}
\def \R{3}
\coordinate (O) at (0, 0, 0);
\draw (0, 0, 0);
\begin{scope}[rotate=30]
\path[draw, fill = cyan, opacity=0.5] (\R,0) arc [start angle=0,
end angle=180,
x radius=\R cm,
y radius=1*\R cm] ;
\path[draw, fill = cyan, opacity=0.5] (-\R,0) arc [start angle=180,
end angle=360,
x radius=\R cm,
y radius=1*\R cm] ;
\end{scope}
\draw[black, thick, rotate=45, ->] (0,0,0) -- ++(0,0, 1.845*\R);
\draw[black, thick, rotate=-75, ->] (0,0,0) -- ++(0,0, 1.845*\R);
\draw[black, thick, rotate=-15, ->] (0,0,0) -- ++(0,0, -1.845*\R);
\draw (0, 2, 0) node[black]{$\mathbb{R}^p,\ p=2$};
\draw (0, -1.8, 0) node[black, right=0.3mm]{$\bm{d}^p\in \mathbb{D}^p$};
\draw (1.0, -1, 0) node[black]{$\norme{\bm{d}^p}=1$};
\end{tikzpicture}
\end{adjustbox}

\columnbreak

\begin{adjustbox}{max totalsize={.9\textwidth}{.9\textheight},center}
\begin{tikzpicture}
\def \R{3}
\filldraw[ball color=white] (0,0) circle (\R);
\fill[ball color=white, opacity=0.8] (0,0,0) circle (\R); % 3D lighting effect
\coordinate (O) at (0, 0, 0);
\begin{scope}[rotate=30]
\path[draw,dashed, fill = yellow, opacity=0.25] (\R,0) arc [start angle=0,
end angle=180,
x radius=\R cm,
y radius=0.5*\R cm] ;
\path[draw, fill = yellow, opacity=0.25] (-\R,0) arc [start angle=180,
end angle=360,
x radius=\R cm,
y radius=0.5*\R cm] ;
\end{scope}
\draw (0.1, 0.2, 0) node[black]{$\x$}; 
\draw[red, thick, dashed, rotate = 120, ->] (0,0,0) -- node[above]{} (1.1*\R,0,0);
\draw (0, 0, 0) node[black]{$\bullet$};
\draw (0, 2.5, 0) node[black]{$\mathbb{R}^n,\ n=3$};
\draw (-1.9, 2.85, 0) node[black]{$\bm{v}$};
\foreach \angle[count=\n from 1] in {-30} {
\begin{scope}[rotate=\angle]
	\path[draw,dashed, fill = cyan, opacity=0.5] (\R,0) arc [start angle=0,
	end angle=180,
	x radius=\R cm,
	y radius=0.5*\R cm] ;
	\path[draw, fill = cyan, opacity=0.5] (-\R,0) arc [start angle=180,
	end angle=360,
	x radius=\R cm,
	y radius=0.5*\R cm] ;
\end{scope}
}
\draw[red, thick, dashed, rotate = 120, ->] (0,0,0) -- node[above]{} (1.1*\R,0,0);
\draw (1.8, -0.8, 0) node[black]{\scalebox{0.6}{$\bm{U}_2$}};
\draw (-0.8, -0.8, 0) node[black]{\scalebox{0.6}{$\bm{U}_1$}};
\draw (-1.5, 0, 0) node[black]{\scalebox{0.6}{$\norme{\bm{U}_1}=1$}};
\draw (-1.7, 1.4, 0) node[black]{\scalebox{0.8}{$\textcolor{black}{\bm{U}_{\ntp}}\bm{d}^p$}};
\draw[red, thick, rotate=59, ->] (0,0,0) -- ++(0,0, 1.15*\R);
\draw[red, thick, rotate=-29, ->] (0,0,0) -- ++(0,0, -1.15*\R);
\draw[red, thick, rotate = 150, ->] (0,0,0) -- node[above]{} ( \R, 0,0);
\draw (0, 0, 0) node[black]{$\bullet$};
\draw (0.1, 0.2, 0) node[black]{$\x$}; 
\draw[black, dashed, rotate=15, ->] (0,0,0) -- ++(0,0, 0.925*\R);
\draw[black, dashed, rotate=105, ->] (0,0,0) -- ++(0,0,1.843*\R);
\end{tikzpicture}
\end{adjustbox}
\end{multicols}
\caption{Illustration of a minimal positive basis {\bl (the smallest PSS, i.e., with exactly $p+1$ directions)} $\pollD\subset\spOne$ of~$\rp$, with $p=2$, and the resulting set of poll directions $\mathbb{U}^n:=\accolade{\bm{U}_{\ntp}\di^p:\di^p\in \pollD}$, with $n=3$, where $\bm{U}_{\ntp}:=\left[\scalebox{0.7}{$\bm{U}_1$}\cdots \scalebox{0.7}{$\bm{U}_p$}\right]\in\rnp$ is a matrix with orthonormal columns. As demonstrated by Proposition~\ref{Prop2Point3}, $\pollD$ has the same positive spanning properties as $\mathbb{U}^n$ in the subspace generated by the vectors \scalebox{0.7}{$\bm{U}_j$}, $j\in\intbracket{1,p}$, illustrated by the blue hyperplane on the right. When $\vb\in\rn$ is a descent direction at~$\x$, say $\vb=-\nabla f(\x)\neq\bm{0}$ assuming~$f$ differentiable, the selection of the yellow hyperplane must be avoided since none of its poll directions makes an acute angle with~$\vb$ as desired by  complexity theory; a situation that may also impact the efficiency of an algorithm using such poll directions. Indeed, whenever the yellow hyperplane is selected, it holds that $\bm{U}_{\ntp}^{\top}\vb=\bm{0}$ even though $\vb\neq\bm{0}$ as assumed above. Fortunately such situations are avoided with high probability in Algorithm~\ref{algoStoDARS}, since the random matrix $\Qrandom:=\sqrt{\frac{n}{p}}\Urandom_{\ntp}$ satisfies $\normii{\QrandomTranspose\vb}\geq \alpha_{Q}\normii{\vb}$ for some $\alpha_{Q}\in (0,1)$ with high probability thanks to Theorem~\ref{JLforHaarTheor}, when obtained from Haar distribution. Moreover, Lemma~\ref{lemKappaD3} ensures that the illustrated behaviour of the blue hyperplane on the right (and its poll directions) with respect to~$\vb$ occurs with high probability, that is, with high probability there exists a random poll direction that makes an acute angle with $\vb$.}
\label{figTikz1}
\end{figure}
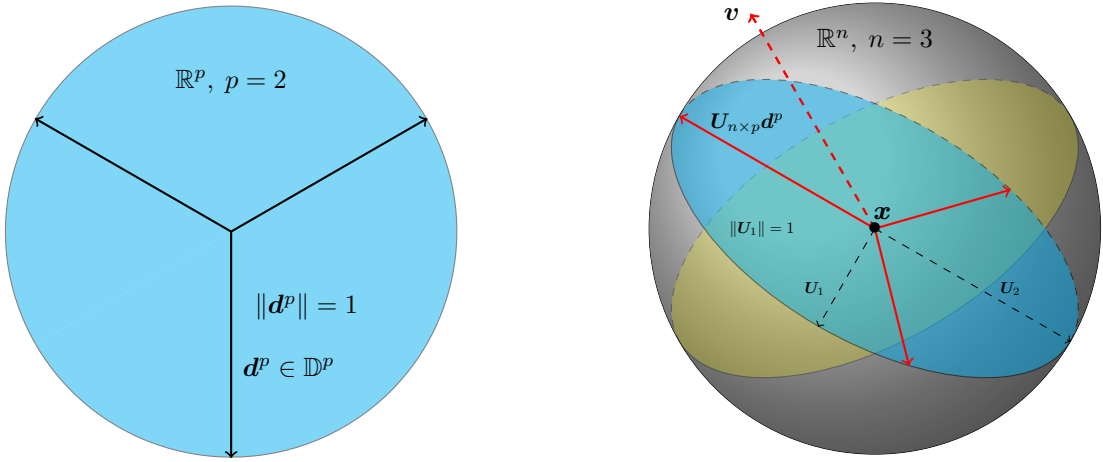
\vspace*{0.5cm}

%Fortunately, when $\bm{U}_{\ntp}$ is obtained from Haar-distributed orthogonal matrices, the latter situation occurs only with a very small probability thanks to

We conclude this section by the result stated next, which shows another advantage of generating poll directions in Algorithm~\ref{algoStoDARS} from Haar-distributed matrices combined with positive spanning sets $\pollD\subset\spOne$. Indeed, although $\mathbb{U}^n:=\accolade{\bm{U}_{\ntp}\di^p:\di^p\in \pollD}$ is not a positive spanning set of~$\rn$, it has the same positive spanning properties as $\pollD$ in the subspace defined by the columns of $\bm{U}_{\ntp}:=\left[\scalebox{0.7}{$\bm{U}_1$}\cdots \scalebox{0.7}{$\bm{U}_p$}\right]${\bl ; such properties are not necessarily guaranteed by other matrices such as Gaussian~\cite[Theorem~3.4]{DzaWildSub2022} or hashing(-like) matrices~\cite{DzaWildHashing2022,KaNel2014SparseLidenstrauss}}. Thus, since the corresponding random subspaces are generated by random vectors (with realizations {\scalebox{0.7}{$\bm{U}_j$}}) that are distributed according to
{$\mathcal{U}(\snOne)$} for all $j\in\intbracket{1,p}$ (see Figure~\ref{figTikz1} for an illustration), one may expect an efficient exploration of the variable space by Algorithm~\ref{algoStoDARS} despite not using positive spanning sets of~$\rn$ for polling.

\begin{proposition}\label{Prop2Point3}
Let $\bm{U}_{\ntp}\in\rnp$ be a matrix with orthonormal columns and assume that $\pollD\subset\spOne$ is a positive spanning set of~$\rp$. Consider the set $\mathbb{U}^n:=\accolade{\bm{U}_{\ntp}\di^p:\di^p\in \pollD}\subset\ {\emph{range}}(\bm{U}_{\ntp})\subset\rn$, where ${\emph{range}}(\bm{U}_{\ntp})$ denotes the span of the column vectors of $\bm{U}_{\ntp}$. Define the cosine measure of $\mathbb{U}^n$ with respect to ${\emph{range}}(\bm{U}_{\ntp})$ by \[\kappa(\mathbb{U}^n{|}_{{\emph{range}}(\bm{U}_{\ntp})}):=\underset{\wb\in {\emph{range}}(\bm{U}_{\ntp})}{\min}\ \underset{\ub\in\mathbb{U}^n}{\max}\frac{\scal{\wb}{\ub}}{\normii{\wb}\normii{\ub}}.\]
Then $\kappa(\mathbb{U}^n{|}_{{\emph{range}}(\bm{U}_{\ntp})})=\kappa(\pollD)=:\kappa(\pollD{|}_{\rp})>0$, that is, the directions in $\mathbb{U}^n$ form a positive spanning set in the column space of $\bm{U}_{\ntp}$. Moreover for all $\di^p_i, \di^p_j\in\pollD$, it holds that $\scal{\bm{U}_{\ntp}\di^p_i}{\bm{U}_{\ntp}\di^p_j}=\scal{\di^p_i}{\di^p_j}$ while $\normii{\bm{U}_{\ntp}\di^p}=\normii{\di^p}=1 $ for all $\di^p\in\pollD$. In other words, $\pollD$ has exactly the same positive spanning properties as $\mathbb{U}^n$ in ${\emph{range}}(\bm{U}_{\ntp})$.
\end{proposition}

\begin{proof}
We first observe that the vectors $\wb\in\text{range}(\bm{U}_{\ntp})$ and $\ub\in\mathbb{U}^n$ can be rewritten as $\wb=\bm{U}_{\ntp}\s$, for some $\s\in\rp$ and $\ub=\bm{U}_{\ntp}\di^p$ where $\di^p\in\pollD$, which implies that $\kappa(\mathbb{U}^n{|}_{{\text{range}}(\bm{U}_{\ntp})})=\underset{\s\in\rp}{\min}\ \underset{\di^p\in\pollD}{\max}\frac{\scal{\bm{U}_{\ntp}\s}{\,\bm{U}_{\ntp}\di^p}}{\normii{\bm{U}_{\ntp}\s}\normii{\bm{U}_{\ntp}\di^p}}$. Assume that $\s:=(s_1,\dots,s_p)$ and let $\scalebox{0.7}{$\bm{U}_j$}\in \rn$, with $j\in\intbracket{1,p}$, be the $j$th column of $\bm{U}_{\ntp}$. Then $\normii{\bm{U}_{\ntp}\s}^2=\scal{\bm{U}_{\ntp}\s}{\bm{U}_{\ntp}\s}=\sum_{j=1}^{p}s_j^2\normii{\scalebox{0.7}{$\bm{U}_j$}}^2=\normii{\s}^2$ where the last equality follows from the orthonormality of the columns of $\bm{U}_{\ntp}$. Likewise, $\normii{\bm{U}_{\ntp}\di^p}=\normii{\di^p}=1$. Moreover, it holds that $\scal{\bm{U}_{\ntp}\s}{\bm{U}_{\ntp}\di^p}=\scal{\s}{\bm{U}_{\ntp}^{\top}\bm{U}_{\ntp}\di^p}=\scal{\s}{\di^p}$ since $\bm{U}_{\ntp}^{\top}\bm{U}_{\ntp}=\bm{I}_p\in\rpp$. Consequently, $\kappa(\mathbb{U}^n{|}_{{\text{range}}(\bm{U}_{\ntp})})=\underset{\s\in\rp}{\min}\ \underset{\di^p\in\pollD}{\max}\frac{\scal{\s}{\di^p}}{\normii{\s}\normii{\di^p}}=\kappa(\pollD)>0$ where the inequality follows from the fact that $\pollD$ is a positive spanning set of~$\rp$. The proof of the other results trivially follows using the above arguments, which completes the proof.
\end{proof}

\section{Random subspace direct search and its resulting stochastic process}\label{secThree}
This section presents the general framework of the StoDARS algorithm and explains how the proposed method results in a stochastic process.

\subsection{StoDARS algorithm}
In the SDDS framework the algorithms operate in full space by examining poll points of the form $\xk+\dk\di^n$, with $\di^n\in \rn$ and $\dk>0$, where $\xk\in \rn$ is an incumbent solution. In contrast, the proposed StoDARS method operates in affine subspaces $\mathcal{Y}_k$ randomly chosen through the range of matrices $\Ub_{k,\ntp}\in \rnp$ with orthonormal columns, and whose entries are randomly selected. More precisely, inspired by~\cite{RobRoy2022RedSpace} considering trial points of the form $\xk+\dk\Qk\di^p$, with $\di^p\in\rp$, where the columns of $\Qk\in\rnp$ are not necessarily orthonormal, these affine subspaces are defined as
\begin{equation}\label{affSubspaceDef}
\mathcal{Y}_k:={\bl \accolade{\xk+\Ub_{k,\ntp}\di^p:\di^p\in\rp} \supset \accolade{\xk+\dk\Ub_{k,\ntp}\di^p:\di^p\in\pollD_k}=:\mathcal{P}_k},
\end{equation}
where {\bl $\mathcal{P}_k$ denotes the set of poll trial points,} $\pollD_k\subset\spOne$ is a PSS constructed by using realizations of a random matrix $\Drandom_k:=\Drandomtildep_{t_k}\in\rpm$ {\bl (whose columns form a PSS of~$\rp$)}, with $m\in {\bl \intbracket{p+1,2p}}$ and $t_k\in\N$. {\bl We note that a technique for constructing $\pollD_k$ is proposed in Section~\ref{sec8}.} $\Ub_{k,\ntp}$ is a realization of a random matrix $\Urandom_{k,\ntp}:=\Urandomtilde_{t_k,\ntp}$ representing the first~$p$ columns of a Haar-distributed matrix $\Urandomtilde_{t_k}:=\left[\Urandomtilde_{t_k,\ntp}\ \ \Urandomtilde_{t_k,\ntnmp}\right]\in\mathbb{O}(n)$, where~$p$ is chosen so that $\sqrt{\frac{n}{p}}\Urandomtilde_{t_k,\ntp}$ is $W(\alpha_Q,\beta_Q)$ for some $\alpha_Q,\beta_Q\in (0,1/2)$. The independent matrices $\Urandomtilde_{t_k,\ntp}$ and $\Drandomtildep_{t_k}$ are generated so that $\{\Urandomtilde_{t_{i},\ntp}\}_{i=0}^{k}$ are i.i.d and $\{\Drandomtildep_{t_i}\}_{i=0}^{k}$ are independent. The basic idea of combining such matrices and directions through~\eqref{affSubspaceDef} is motivated by the following result and the need for the algorithm to use a (dense) sequence of uniformly distributed vectors on the unit sphere~$\snOne$.

\begin{proposition}\label{unitHaarProp}(Inspired by~\cite[Theorem~1]{BhQPMADS2015})
Let $\Urandom:=\left[\Urandom_{\ntp}\ \, \Urandom_{\ntnmp}\right]\in \mathbb{O}(n)$ be Haar-distributed, where $\Urandom_{\ntp}\in\rnp$. Let $\diRandom\in\spOne$ be a random, possibly degenerate, unit vector {\bl independent of $\Urandom$}; that is, $\diRandom$ is either random or deterministic. Then, $\Urandom_{\ntp}\diRandom\sim\mathcal{U}(\snOne)$; that is, $\Urandom_{\ntp}\diRandom$ is uniformly distributed on the unit sphere $\snOne$.
\end{proposition}

\begin{proof}
Consider $\vrandom:=
\begin{bmatrix}
\diRandom\\
\bm{0}_{n-p}
\end{bmatrix}\in \snOne$, where $\bm{0}_{n-p}:=(0,\dots,0)^{\top}\in\R^{n-p}$. By left-translation invariance of Haar measure {\bl and because $\vrandom$ and $\Urandom$ are independent}, it holds that for all $\bm{O}\in\mathbb{O}(n)$, $\bm{O}(\Urandom\vrandom)=(\bm{O}\Urandom)\vrandom\equald\Urandom\vrandom$, whence $\Urandom\vrandom\sim\mathcal{U}(\snOne)$. Then the proof is completed by observing that $\Urandom\vrandom=\left[\Urandom_{\ntp}\ \, \Urandom_{\ntnmp}\right]\begin{bmatrix}
\diRandom\\
\bm{0}_{n-p}
\end{bmatrix}=\Urandom_{\ntp}\diRandom$.
\end{proof} 
On the other hand, the idea of assigning an integer index~$t_k$ to every iteration~$k$, as was also the case in~\cite{BhQPMADS2015,VDAs2013} inspired by~\cite{AbAuDeLe09}, is motivated by the need to demonstrate in Proposition~\ref{densityProp} the existence of an almost surely dense subsequence $\accoladekinK{\ukRandom}$ of poll directions (where $K\subset\N$ and $\ukRandom:=\Urandom_{k,\ntp}\dikRandom$ with $\dikRandom\in \pollDrandom_k$) on the unit sphere~$\snOne$, as required by the analyses presented in Sections~\ref{sec6} and~\ref{sec7}. Indeed, while it is well known from probability theory that a sequence $\accoladekinN{\ukRandom}$ of i.i.d. random vectors $\ukRandom~\sim\mathcal{U}(\snOne)$ is dense on~$\snOne$ with probability one, the latter property is not guaranteed to hold for any subsequence $\accoladekinK{\ukRandom}$ for an arbitrary random subset $K\subset \N$. However, by updating $t_k$ so that $\accoladekinK{t_k}$ covers the whole set $\accolade{0,1,2,\dots}$ when $K$ is a subset of unsuccessful iterations (the latter are defined similarly to those of SDDS) of Algorithm~\ref{algoStoDARS} as illustrated by Table~\ref{tablktk}, $\accoladekinK{\ukRandom}$ is shown in Proposition~\ref{densityProp} to be dense on~$\snOne$ with probability one.

During the poll step of each iteration~$k$, $\ef$-accurate estimates $\fok\approx f(\xk)$ and $\fuk\approx f(\xk+\dk\ub)$ (see Definition~\ref{epsAccDef0}) are computed, respectively, at the incumbent solution~$\xk$ and the trial points~{\bl $\mathcal{P}_k$.}
%$\mathcal{P}_k:=\{\xk+\dk\ub: \ub\in\mathbb{U}_k^n \}$, where $\mathbb{U}_k^n:=\{\Ukp\di^p:\di^p\in \pollD_k\}$. 
The result stated next determines the iteration type, that is, whether it is successful or not.
\begin{definition}\label{epsAccDef0}
Given $\ef>0$, $\xk\in\rn$, and $\ub:=\Ukp\di^p=\sqrt{\frac{p}{n}}\Qk\di^p\in\rn$ with $\di^p\in\pollD_k$, $\fok$ and $\fuk$ are called $\ef$-accurate estimates of $f(\xk)$ and $f(\xk+\dk\ub)$, respectively, for a given $\dk$ if
\begin{equation}\label{epsAccDef1}
\abs{\fok-f(\xk)}\leq \ef\dk^2\qquad \mbox{and}\qquad \abs{\fuk-f(\xk+\dk\ub)}\leq \ef\dk^2.
\end{equation}
\end{definition}
\begin{proposition}\label{decreaseProp1}(Inspired by~\cite[Proposition~1]{dzahini2020expected})
Let $\gamma>2$ be a constant, and let $\fok$ and $\fuk$ be $\ef$-accurate estimates of $f(\xk)$ and $f(\xk+\dk\ub)$, respectively: 
\begin{equation}\label{decreaseEqu}
\mbox{If}\quad \fuk-\fok\leq -\gamma\ef\dk^2,\qquad\mbox{then}\qquad f(\xk+\dk\ub)-f(\xk)\leq -(\gamma-2)\ef\dk^2.
\end{equation}
\end{proposition}
  
\begin{proof}
The proof trivially follows that of~\cite[Proposition~1]{dzahini2020expected} with minor changes.
\end{proof}

%\clearpage

%-----------------------------------------------------

%\begin{minipage}{20cm}
\begin{algorithm}[H]
\caption{StoDARS algorithm}
\label{algoStoDARS} 
%\nonl
% $\accolade{\Urandomtilde_t}_{t=0}^\infty$, where the $\Urandomtilde_t:=\left[\Urandomtilde_{t,\ntp}\ \ \Urandomtilde_{t,\ntnmp}\right]\in\mathbb{O}(n)$ are Haar-distributed, is a sequence of independent matrices, with $p\in\intbracket{1,n}$ chosen so that $\sqrt{\frac{n}{p}}\Urandomtilde_{t,\ntp}\in\rnp$ is $WAM(\alpha_Q,\beta_Q)$ for some $\alpha_Q, \beta_Q\in(0,1/2)$. \\
% $\accolade{\BrandomtildeSet_t^+}_{t=0}^\infty$ is a sequence of random PSSs such that the matrices $\Brandomtilde_t^+\in\R^{p\times m}$ with unit columns in $\BrandomtildeSet_t^+\subset\mathbb{S}^{p-1}$ are independent.    \\
% $ $\\
\textbf{[0] Initialization}\\
\hspace*{10mm}Choose $\x_0\in\bcX\subseteq\rn$, $\delta_0>0$, $\ef>0$, $\gamma>2$, $0<\tau<\left(\frac{\gamma-2}{\gamma+2}\right)^{1/2}$, $j_{\max}\in\N$, $\dmax=\tau^{-j_{\max}}\delta_0$ \\ 
\hspace*{10mm}$\ell_0=0$ and $t_0=0$. Generate $\tilde{\Ub}_{0,\ntp}$, where $\Urandomtilde_0:=\left[\Urandomtilde_{0,\ntp}\ \ \Urandomtilde_{0,\ntnmp}\right]\in\mathbb{O}(n)$ is Haar-distributed, \\
\hspace*{10mm}with $p\in\intbracket{1,n}$ chosen so that $\sqrt{\frac{n}{p}}\Urandomtilde_{0,\ntp}\in\rnp$ is $WAM(\alpha_Q,\beta_Q)$ for some $\alpha_Q, \beta_Q\in(0,1/2)$.\\
\hspace*{10mm}Construct a PSS $\tilde{\mathbb{D}}^p_0\subset\spOne$ using {\bl any matrix} $\Drandomtildep_0\in\R^{p\times m}$, with $m\in {\bl \intbracket{p+1,2p}}$, {\bl generated}\\ 
\hspace*{10mm}{\bl independently of $\Urandomtilde_{0}$, whose columns form a PSS of~$\rp$}.\\
\hspace*{10mm}Set the iteration counter $k \gets 0$.\\
% {\bl\textbf{[1] Matrix index update}\\
% \hspace*{10mm}If $\ell_{k} \geq  \max_{i\leq k-1} \ell_i$ or $k=0$, set $t_{k}\gets \ell_{k}$.   \\
% \hspace*{10mm}Otherwise $t_{k}\gets 1+\max_{i\leq k-1} t_i$.\\}
\textbf{[{1}] Matrix and PSS for subspace polling}\\
% \hspace*{10mm}If $k\geq 1$ and $t_k=t_{i_0}$ for some $i_0\in \intbracket{0,k-1}$   \\
% \hspace*{16mm}Generate $\tilde{\Ub}_{t_k,\ntp}$ using an independent copy $\Urandomtilde_{t_k,\ntp}$ of $\Urandomtilde_{t_{i_0},\ntp}$ such that  \\
% \hspace*{16mm}$\{\Urandomtilde_{t_i,\ntp}\}_{i=0}^{k-1}$ and $\Urandomtilde_{t_k,\ntp}$ are i.i.d. \\
% \hspace*{16mm}Construct a PSS $\tilde{\mathbb{D}}^p_{t_k}\subset\spOne$ using the columns of an independent copy $\Drandomtildep_{t_k}$ of $\Drandomtildep_{t_{i_0}}$ \\
% \hspace*{16mm}such that $\{\Drandomtildep_{t_i}\}_{i=0}^{k-1}$ and $\Drandomtildep_{t_k}$ are i.i.d. \\
% \hspace*{10mm}If $k\geq 1$ and $t_k\not\in\accolade{t_0,\dots,t_{k-1}}$ \\
\hspace*{10mm}Generate $\tilde{\Ub}_{t_k,\ntp}$ as a realization of $\Urandomtilde_{t_k,\ntp}$, with $\{\Urandomtilde_{t_i,\ntp}\}_{i=0}^{k}$ {\bl obtained from the i.i.d. $\{\Urandomtilde_{t_i}\}_{i=0}^{k}$}.\\
\hspace*{10mm}Construct a PSS $\tilde{\mathbb{D}}^p_{t_k}\subset\spOne$ using the columns of $\Drandomtildep_{t_k}$, where the independent $\{\Drandomtildep_{t_i}\}_{i=0}^{k}$ {\bl are generated}\\ 
\hspace*{10mm}{\bl independently of $\{\Urandomtilde_{t_i}\}_{i=0}^{k}$}.  \\
% \hspace*{16mm}such that $\accolade{\Urandom_{t_i,\ntp}}_{i=0}^k$ and $\{\Urandomtilde_{t_i,\ntp}\}_{i=0}^k$ are independent.   \\
% \hspace*{16mm}Select $\BrandomSet_{t_k}^+$ as the columns of an independent copy $\Brandom_{t_k}^+$ of $\Brandomtilde_{t_k}^+$, where $\Brandom_{t_k}^+$ is independent of  \\
% % \hspace*{16mm}$\Brandom_{t_i}^+,\ i\leq k-1$, and $\Brandomtilde_t^+, t\geq 0$. \\
% \hspace*{16mm}$\Brandom_{t_i}^+,\ i\leq k-1$, and $\accolade{\Brandom_{t_i}^+}_{i=0}^k$ and $\{\Brandomtilde_{t_i}^+\}_{i=0}^k$ are independent. \\
% \hspace*{10mm}Otherwise $\Urandom_{t_k,\ntp}:=\Urandomtilde_{t_k,\ntp}$ and $\BrandomSet_{t_k}^+:=\BrandomtildeSet_{t_k}^+$. \\
\hspace*{10mm}Define $\Ukp:=\tilde{\Ub}_{t_k,\ntp}$ and $\pollD_k:=\tilde{\mathbb{D}}^p_{t_k}$.\\
\textbf{[{2}] Poll}\\
%\hspace*{10mm}Select a PSS $\pollD_k\subset\rp$ such that $\pollDrandom_k$ is independent of $\pollDrandom_0,\dots,\pollDrandom_{k-1}$ when $k\geq 1$. \\
\hspace*{10mm}Generate a set of poll points $\mathcal{P}_k:=\{\xk+\dk\ub: \ub\in\mathbb{U}_k^n \}$; where $\mathbb{U}_k^n:=\{\Ukp\di^p:\di^p\in \pollD_k\}$, {\bbl and}\\
\hspace*{10mm}{\bbbl $\pollD_k$ is generated independently of $\Ukp$}.\\
%\hspace*{8mm}{}\\ 
\hspace*{10mm}Obtain estimates\footnotemark $\fuk$ and $\fok$ of $f$, respectively, at ${\xk+\dk\ub}\in\mathcal{P}_k$ and $\xk$.\\
%\hspace*{10mm}function evaluations.\\
%\hspace*{10mm}\textbf{{Success}}\\
\hspace*{10mm}If $\Pkfeas:=\Pk\cap\bcX\neq\emptyset$ and $\fuk -\fok \leq -\gamma \ef\dk^2$ for some ${\xk+\dk\ub}\in\Pkfeas$ (\textbf{successful})\\ 
\hspace*{16mm}set $\xkun\gets \xk+\dk\ub$,  $\dkun\gets \min\{\tau^{-1}\dk,\dmax\}$, and $\ell_{k+1}\gets \ell_k-1$. \\
%\hspace*{10mm}\textbf{{Failure}}\\
\hspace*{10mm}Otherwise (\textbf{unsuccessful}) set $\xkun\gets \xk$, $\delta_{k+1}\gets\tau\dk$, and $\ell_{k+1}\gets \ell_k+1$. \\
\textbf{[3] Matrix index update}\\
\hspace*{10mm}If $\ell_{k+1} \geq  \max_{i\leq k} \ell_i$, set $t_{k+1}\gets \ell_{k+1}$.   \\
\hspace*{10mm}Otherwise $t_{k+1}\gets 1+\max_{i\leq k} t_i$.\\
\textbf{[4] Termination}\\
\hspace*{10mm}If no termination criterion is met, set $k\gets k+1$ and go to \textbf{[1]}.\\
%\hspace*{10mm}\\
\hspace*{10mm}Otherwise stop.
\end{algorithm}
% \footnotetext[1]{This can be achieved by selecting $\Urandom_{t_k,\ntp}:=\Urandom_{t_k,\ntp}'$ from a Haar-distributed matrix $\Urandom_{t_k}':=\left[\Urandom_{t_k,\ntp}'\ \ \Urandom_{t_k,\ntnmp}'\right]\in\mathbb{O}(n)$, where the sequence $\accolade{\Urandom_t'}_{t=0}^\infty$ of i.i.d. matrices is such that $\{\Urandomtilde_t\}_{t=0}^\infty$ and %$\accolade{\Urandom_{t_i}, i\leq k\ \mbox{and}\ \Urandomtilde_t, t\geq 0}$ 
% $\accolade{\Urandom_t'}_{t=0}^\infty$
% are two independent families of random matrices. A similar strategy applies for the selection of $\BrandomSet_{t_k}^+$.}
\footnotetext{In practice, there is no need for computing estimates  whenever $\Pkfeas:=\Pk\cap\bcX=\emptyset$. Here for theoretical purposes only,~$\fok$ and~$\fuk$ are always computed in order for~$\Fok$ and~$\Fuk$ to exist for the definition of the $\sigma$-algebras~$\falgebra$, for all $k\in\N$.}
%\end{minipage}

According to Proposition~\ref{decreaseProp1}, if $\fuk-\fok\leq -\gamma\ef\dk^2$ for some feasible $\xk+\dk\ub\in \mathcal{P}_k\cap \bcX$, then the iteration is successful. In that case, the incumbent solution and the stepsize parameter are respectively updated according to $\xkun=\xk+\dk\ub$ and $\dkun=\min\accolade{\tau^{-1}\dk,\dmax}$, for some $\dmax$ defined as $\dmax:=\tau^{-j_{\max}}\delta_0$, with $j_{\max}\in \N$ for the needs of the theoretical analysis presented in Section~\ref{sec5}, where inspired by~\cite[Section~2.1]{dzahini2020expected}, $0<\tau<\left(\frac{\gamma-2}{\gamma+2}\right)^{1/2}$. On the other hand, when no trial point is feasible or if the condition $\fuk-\fok\leq -\gamma\ef\dk^2$ does not hold for all points in $\mathcal{P}_k$, then the iteration is unsuccessful, in which case~$\xk$ is not updated while $\dkun=\tau\dk$. Following~\cite{BhQPMADS2015}, the index $t_k$ is updated at the end of each iteration by introducing a sequence $\accoladekinN{\ell_k}$ initialized with $\ell_0=0$, with~$\ell_k$ updated according to $\ell_{k+1}=\ell_k-1$ on successful iterations and $\ell_{k+1}=\ell_k+1$ on unsuccessful ones, so that $\ell_k\to \infty$ when $\dk\to 0$. Then, at the end of each iteration, regardless of the iteration type, $t_{k+1}$ is equal to $\ell_{k+1}$ if $\ell_{k+1} \geq \max_{i\leq k}\ell_i$ and is equal to $1+\max_{i\leq k} t_i$ otherwise.

$ $\\
\begin{table}[ht!]
\caption{Illustration of possible values of $t_k$ and $\ell_k$}
\label{tablktk}
\centering
\resizebox{\textwidth}{!}{
\begin{tabular}{@{}lccccccccccccccc@{}}
	\cmidrule(l){2-16}
	& \multicolumn{15}{c}{Iteration  number $k\quad / \quad$ Outcome ($S$ for success and $F$ for failure)}                                               \\ \cmidrule(l){2-16} 
	& $\ 0/S$ & $\ 1/S$ & $\ 2/S$ & $\ 3/F$ & $\ 4/F$ & $\ 5/S$ & $\ 6/F$ & $\ 7/F$ & $\ 8/F$ & $\ 9/F$ & $10/F$ & $11/S$ & $12/F$ & $13/S$ & $14/S$ \\ \midrule
	$\ell_{k+1}  $ & $-1$   & $-2$   & $-3$   & $-2$   & $-1$   & $-2$   & $-1$   & $0$    & $1$    & $2$    & $3$     & $2$     & $3$     & $2$     & $1$     \\
	$t_{k+1}  $    & $1$    & $2$    & $3$    & $4$    & $5$    & $6$    & $7$    & $0$    & $1$    & $2$    & $3$     & $8$     & $3$     & $9$     & $10$    \\ \bottomrule
\end{tabular}
}
\end{table}

\begin{table}[ht!]
\centering
\resizebox{\textwidth}{!}{
\begin{tabular}{@{}lccccccccccccccc@{}}
	\cmidrule(l){2-16}
	& \multicolumn{15}{c}{Iteration  number $k\quad / \quad$ Outcome ($S$ for success and $F$ for failure)}                                \\ \cmidrule(l){2-16} 
	& $15/S$ & $16/F$ & $17/F$ & $18/S$ & $19/F$ & $20/F$ & $21/S$ & $22/F$ & $23/F$ & $24/S$ & $25/F$ & $26/F$ & $27/S$ & $28/F$ & $29/F$ \\ \midrule
	$\ell_{k+1} $ & $0$    & $1$    & $2$    & $1$    & $2$    & $3$    & $2$    & $3$    & $4$    & $3$    & $4$    & $5$    & $4$    & $5$    & $6$    \\
	$t_{k+1} $    & $11$   & $12$   & $13$   & $14$   & $15$   & $3$    & $16$   & $3$    & $4$    & $17$   & $4$    & $5$    & $18$   & $5$    & $6$    \\ \bottomrule
\end{tabular}
}
\end{table}
%---------------------------------------------------

\begin{remark}\label{Referee1Remark}
{\bl We note that each time $\ell_{k+1}$ is greater than or equal to all previous values $\ell_0, \ldots,\ell_k$, it updates $t_{k+1}=\ell_{k+1}$, creating a new ``record'' value in $t_k$. But such a situation can only happen during an unsuccessful iteration (where~$\dk$ is decreased) since only those iterations increase~$\ell_k$. 
%Because~$\ell_k$ increases by~$1$ each time an unsuccessful iteration occurs (and does not decrease afterward unless a successful iteration occurs), this means that new records of $\ell_k$ correspond to new values of~$t_k$. Consequently, 
As a consequence of the updating strategy, $t_{k+1}=\ell_{k+1}=0,1,2,\dots$ at the respective times that $\ell_k$ is equal to the max or sets a new max, hence there exists a subsequence~$K$ such that $\accolade{t_{k}}_{k\in K}$ covers the entire set $\accolade{0,1,2,\dots}$ as $\dk\to 0$.
%$K:=\accolade{k\in \N:t_{k+1}=\ell_{k+1}=\ell_{k}+1}$.
}
\end{remark}

{\bbbl The next proposition concerns only unsuccessful iterations.}
\begin{proposition}\label{densityProp}
%Let $K$ be a random subset of unsuccessful iterations. 
{\bbbl Let $K\subseteq K_{\infty}$ be such that $\accolade{t_{k}}_{k\in K}=\N$ and consider the set $\UsetDrandom_k^n:=\{\Ukprandom\diRandom:\diRandom\in \pollDrandom_k\}\subset\snOne$ of poll directions generated by Algorithm~\ref{algoStoDARS} at iteration~$k$, where $\pollDrandom_k$ is generated independently of $\Ukprandom$. Then the event $\mathscr{E}_\mathbb{U}:=\accolade{\overline{\bigcup_{k\in K}\UsetDrandom_k^n}=\snOne}$ is almost sure\footnote{{\bbbl Here, $\overline{\mathscr{A}}$ represents the closure of set $\mathscr{A}$, that is, the smallest closed set that contains $\mathscr{A}$.}}, that is, $\bigcup_{k\in K}\UsetDrandom_k^n$ is dense on the unit sphere $\snOne$ with probability one.
}
%The sequence $\accolade{\ukRandom}_{k\in K}$, with $K\subseteq K_{\infty}$, of poll directions generated by Algorithm~\ref{algoStoDARS} is dense on the unit sphere $\snOne$ with probability one, where $\ukRandom:=\Ukprandom\dikRandom$ {\bbl and $\dikRandom$ is any arbitrary direction selected from $\pollDrandom_k$ independently of $\Ukprandom$.}
\end{proposition}
%\begin{proof} {\bbl See Appendix~\ref{prop33}.}
%\end{proof}

\begin{proof}
{\bll By construction in Algorithm~\ref{algoStoDARS}, the random objects used at iteration~$k$ are obtained from a pre-generated sequence indexed by~$t\in\N$. More precisely, $\Ukprandom:=\Urandomtilde_{t_k,\ntp}$ and $\pollDrandom_k:=\pollDrandomtilde_{t_k}$. The pairs $\left(\Urandomtilde_{t,\ntp},\pollDrandomtilde_t\right)$, $t\in\N$, are mutually independent, and for every fixed $t\in\N$, the set $\pollDrandomtilde_t$ is generated independently of $\Urandomtilde_{t,\ntp}$. Fix an index $j_0\in\intbracket{1,m}$, e.g., $j_0=1$. For each $t\in\N$, let $\diRandomtilde_{t,j_0}^p\in \pollDrandomtilde_t\subset \spOne$ denote the $j_0$th column of $\Drandomtildep_t$ (the matrix with columns $\pollDrandomtilde_t$), and define $\urandomtilde_{t,j_0}:=\Urandomtilde_{t,\ntp}\diRandomtilde_{t,j_0}^p\in \snOne$. Since $\Urandomtilde_{t,\ntp}$ is obtained from the first~$p$ columns of a Haar-distributed orthogonal matrix and since $\diRandomtilde_{t,j_0}^p$ is a unit vector independent of $\Urandomtilde_{t,\ntp}$, it follows from Proposition~\ref{unitHaarProp} that $\urandomtilde_{t,j_0}\sim \mathcal{U}(\snOne)$ for all~$t\in\N$. Moreover, because the pairs $\left(\Urandomtilde_{t,\ntp},\pollDrandomtilde_t\right)$, $t\in\N$, are mutually independent, the random vectors $\urandomtilde_{t,j_0}$, $t\in\N$, are independent. Hence, $\accolade{\urandomtilde_{t,j_0}}_{t=0}^\infty$ is a sequence of i.i.d. random vectors uniformly distributed on $\snOne$, and therefore is almost surely dense on $\snOne$.} 

As emphasized in~\cite{BhQPMADS2015}, the update rule (inspired by~\cite[Eqs.~(29) and~(30)]{BhQPMADS2015}) of $t_k$ and $\ell_k$ in Algorithm~\ref{algoStoDARS}, illustrated by Table~\ref{tablktk}, guarantees the existence of a random subset $K$ of unsuccessful iterations {\bl (see also Remark~\ref{Referee1Remark})} such that $\accolade{t_{k}}_{k\in K}$ covers the entire set $\accolade{0,1,2,\dots}$ as $\dk\to 0$. {\bll Thus, for every $t\in\N$, there exists at least one $k\in K$ such that $t_k=t$. For such a~$k$, the construction in Algorithm~\ref{algoStoDARS} gives $\Ukprandom=\Urandomtilde_{t_k,\ntp}=\Urandomtilde_{t,\ntp}$ and $\pollDrandom_k=\pollDrandomtilde_{t_k}=\pollDrandomtilde_t$. If $\dikjoRandom\in \pollDrandom_k$ denotes the $j_0$th column of $\Drandom_k$ (the matrix with columns~$\pollDrandom_k$), then $\dikjoRandom=\diRandomtilde_{t_k,j_0}^p=\diRandomtilde_{t,j_0}^p$. Consequently,
$\ukjoRandom:=\Ukprandom\dikjoRandom=\Urandomtilde_{t_k,\ntp}\diRandomtilde_{t_k,j_0}^p=:\urandomtilde_{t_k,j_0}$. Since $\accolade{t_k}_{k\in K}=\N$, then $\accolade{\ukjoRandom:k\in K}=\accolade{\urandomtilde_{t,j_0}:t\in\N}$, hence $\prob{\overline{\accolade{\ukjoRandom:k\in K}}=\snOne}=1$.} We finally note that for every $k\in K$, $\ukjoRandom\in \UsetDrandom_k^n$ and the proof is completed by observing that $\accolade{\ukjoRandom:k\in K}\subseteq \bigcup_{k\in K}\UsetDrandom_k^n$.
\end{proof}

\begin{remark}
{\bll 
We note that the essential point of the proof of Proposition~\ref{densityProp} uses the independence and uniform distribution of $\accolade{\urandomtilde_{t,j_0}}_{t=0}^\infty$ resulting from the pre-generated sequence $\accolade{\left(\Urandomtilde_{t,\ntp},\pollDrandomtilde_t\right)}_{t\in \N}$, before any conditioning on unsuccessful iterations. Thus, it does not claim that the random directions remain independent or uniformly distributed after conditioning on the event that the corresponding iterations are unsuccessful, which would generally be false since conditioning on failure or success may bias the distribution of the selected directions and destroy their independence. This is precisely the role of the $t_k$-mechanism. Rather than relying on distributional properties after conditioning on unsuccessful iterations, the update rule for~$t_k$ ensures that, along $K\subseteq K_{\infty}$, we have the coverage property $\accolade{t_k}_{k\in K}=\N$. Consequently, the poll directions generated on these unsuccessful iterations contain the pre-generated i.i.d. uniformly distributed sequence $\accolade{\urandomtilde_{t,j_0}}_{t=0}^\infty$. The density conclusion therefore follows from the almost sure density of this underlying pre-generated sequence together with the aforementioned coverage property. In particular, density is established through the pathwise coverage of the original pre-generated sequence along unsuccessful iterations, rather than through distributional properties that need not be preserved after conditioning on success or failure.
}
\end{remark}

\subsection{Stochastic process generated by StoDARS}
The deterministic estimates $\fok$ and $\fuk$ are constructed at each iteration of Algorithm~\ref{algoStoDARS} by using evaluations of the stochastically noisy function $f_{\ubar{\theta}}$ and the randomly chosen entries of the matrices $\Ukprandom$ and $\Drandom_k$. Consequently, such function estimates can be considered as realizations of random estimates $\Fok$ and $\Fuk$ whose behavior necessarily influences each iteration of the method. The latter therefore results in a stochastic process. The goal of the present work is to show that under certain assumptions on the sequence $\accoladekinN{\Fok,\Fuk}$ and the matrices used for subspace polling, the resulting process has some desirable convergence properties, conditioned on the past history of the algorithm. In particular, inspired by~\cite{audet2019stomads,blanchet2016convergence,chen2018stochastic,dzahini2020expected,DzaWildSub2022,paquette2018stochastic}, these random estimates are required to be probabilistically sufficiently accurate, conditioned on the past. On the other hand, although the matrices used for subspace polling are clearly defined in Algorithm~\ref{algoStoDARS}, the following is explicitly assumed in the remainder of the manuscript for the sake of clarity.

\begin{assumption}\label{QkHaarAssumpt}{\bl (Motivated by Theorem~\ref{JLforHaarTheor})}
For all $k\geq 0$, the independent random matrices  $\Drandom_k\in\R^{p\times m}$, with $m\in {\bl \intbracket{p+1,2p}}$, and $\Ukprandom\in\rnp$ used for subspace polling in Algorithm~\ref{algoStoDARS} are such that $\pollDrandom_k\subset\spOne$, $\Urandom_k:=\left[\Urandom_{k,\ntp}\ \, \Urandom_{k,\ntnmp}\right]\in \mathbb{O}(n)$ is Haar-distributed, and $\QkRandom:=\sqrt{\frac{n}{p}}\Ukprandom$ is $WAM(\alpha_Q,\beta_Q)$ for some $\alpha_Q, \beta_Q\in(0,1/2)$.
\end{assumption}

Next, $(\QkRandom)_{ij}$, with $(i,j)\in\mathbb{I}_{\ntp}:=\intbracket{1,n}\times\intbracket{1,p}$, denote the entries of $\QkRandom$. Likewise, $(\Drandom_k)_{i'j'}$, with $(i',j')\in\mathbb{I}_{p,m}$, denote the entries of $\pollDrandom_k$. The conditioning on the past is formalized by using the filtrations $\accoladekinN{\falgebra}$ and $\accoladekinN{\fqdalgebra}$ defined by
\begin{eqnarray*}
\falgebra&:=&\sigma\left({\ubar{f}}_{\ell}({\bm{0}}), {\ubar{f}}_{\ell}(\urandom),(\Qrandom_\ell)_{ij}, (\Drandom_\ell)_{i'j'}\ \mbox{with}\ \ell\in\intbracket{0,k-1}, (i,j)\in\mathbb{I}_{\ntp}\ \mbox{and}\ (i',j')\in\mathbb{I}_{p,m}\right)\\
\fqdalgebra&:=&\sigma\left({\ubar{f}}_{\ell}({\bm{0}}), {\ubar{f}}_{\ell}(\urandom),(\Qrandom_\ell)_{ij}, (\Drandom_\ell)_{i'j'},(\QkRandom)_{ij}, (\Drandom_k)_{i'j'} \ \mbox{with}\ \ell\in\intbracket{0,k-1}, (i,j)\in\mathbb{I}_{\ntp}\ \mbox{and}\ (i',j')\in\mathbb{I}_{p,m}\right),
\end{eqnarray*}
where for {\it completeness}~\cite{billingsley1995ThirdEdition}, $\F_{-1}:=\sigma(\x_0)$. The idea behind the above definitions is explained in~\cite[Remark~2.1]{DzaWildSub2022}. By construction, $\QkRandom$ and $\Drandom_k$ are $\fqdalgebra$-measurable so that $\E{\QkRandom|\fqdalgebra}=\QkRandom$ almost surely and $\E{\Drandom_k|\fqdalgebra}=\Drandom_k$ almost surely. The sufficient accuracy of the random estimates is formalized as follows.

\begin{definition}
A sequence $\accolade{\Fok,\Fuk}$ of random estimates is called $\beta_f$-probabilistically $\ef$-accurate with respect to the corresponding sequence $\accolade{\Xk,\Dk,\QkRandom,\Drandom_k}$ if the events
\begin{eqnarray*}
J_k&:=&\left\lbrace \mbox{Given}\ \QkRandom, \Drandom_k\ \mbox{and}\ \Xk, \Fok\ \mbox{and}\ \Fuk\ \mbox{are}\ \ef\mbox{-accurate}\ \mbox{estimates of}\ f(\Xk)\ \mbox{and} \right.\\
& & \qquad\qquad\qquad\qquad\qquad\qquad\qquad\qquad\qquad\qquad\left. f(\Xk+\Dk\urandom),\ \mbox{respectively, for}\ \Dk \right\rbrace
\end{eqnarray*}
satisfy the condition $\prob{J_k|\fqdalgebra}=\E{\ijk|\fqdalgebra}\geq \beta_f$ almost surely. The estimates $\Fok$ and $\Fuk$ are called ``good'' if $\ijk=1$. Otherwise, they are called ``bad.''
\end{definition}

\section{Zeroth-order convergence}\label{sec4}
In a deterministic directional direct-search framework, global convergence properties strongly rely on the fact that the objective function~$f$ never increases between successive iterations and that the stepsize $\dk$ tends to zero as $k\to\infty$. As highlighted in~\cite{dzahini2020expected}, the main difficulty in the analysis of stochastic direct-search methods lies in the fact that the monotonicity described above is no longer always guaranteed. The analysis of StoDARS therefore relies on the assumption that function estimates' accuracy improves in coordination with the perceived progress of the method. It uses some properties of a specific supermartingale~$\accoladekinN{\Phik}$ whose increments not only have a tendency to decrease but also depend on the change in the values of~$f$ between successive iterations. The goal of the present section is to derive the  {\it zeroth-order convergence} result; that is, the sequence~$\accoladekinN{\Dk}$ converges to zero almost surely, which is crucial for deriving all {\it almost sure} convergence results afterward. Inspired by~\cite[Assumption~2]{dzahini2020expected} and~\cite[Assumption~2.4]{paquette2018stochastic}, the following assumption similar to~\cite[Assumption~2]{audet2019stomads} is made.

%Let $\gamma>2$ and $\nu,\tau\in (0,1)$ satisfying $\frac{\nu}{1-\nu}\geq \frac{2(\tau^{-2}-1)}{\gamma-2}$.  $\beta_f\in (0,1)$ satisfying  $\frac{\beta_f}{1-\beta_f}\geq \frac{4\nu}{(1-\nu)(1-\tau^2)}$.

\begin{assumption}\label{efkfAssump}
For fixed $\ef>0$, {\bl $\gamma>2$, $\nu,\tau\in (0,1)$ satisfying $\frac{\nu}{1-\nu}\geq \frac{2(\tau^{-2}-1)}{\gamma-2}$} and $\beta_f\in ({\bl 1/2},1)$ {\bl satisfying  $\frac{\beta_f}{1-\beta_f}\geq \frac{4\nu}{(1-\nu)(1-\tau^2)}$}, the following hold for the random quantities generated by Algorithm~\ref{algoStoDARS}:
\begin{itemize}
\item[(i)] The sequence $\accolade{\Fok,\Fuk}$ is $\beta_f$-probabilistically $\ef$-accurate.
\item[(ii)] The sequence $\accolade{\Fok,\Fuk}$ satisfies the following $\kappa_f$-variance conditions, each of which holds almost surely:
\[\E{\abs{\Fok-f(\Xk)}^2|\fqdalgebra}\leq \ef^2(1-\beta_f)\Dk^4\quad\mbox{and}\quad \E{\abs{\Fuk-f(\Xk+\Dk\urandom)}^2|\fqdalgebra}\leq \ef^2(1-\beta_f)\Dk^4.\]
\end{itemize}
\end{assumption}

Here,~{\it (i)} assumes that the function estimates are increasingly better as the stepsize gets smaller, with some sufficiently large, but fixed, probability. This means that they might be overall biased and arbitrarily inaccurate with small probability, thus allowing for outliers when estimating unknown~$f$ function values. On the other hand, the variance of these estimates is adaptively controlled through~{\it (ii)} and decreases~$\Dk$. In particular, the almost sure convergence of~$\accoladekinN{\Dk}$ to zero drives the variance to zero. {\bl Details on how random estimates satisfying Assumption~\ref{efkfAssump} can be made available using, e.g., the standard Chebyshev inequality, are provided in~\cite[Section~5]{chen2018stochastic}, \cite[Section~2.3]{paquette2018stochastic}, \cite[Section~3.5]{blanchet2016convergence}, \cite[Section~5.1]{dzahini2020constrained}, ~\cite[Section~2.2]{dzahini2020expected} and~\cite[Section~2.3]{audet2019stomads}. We note that in particular, the sample size (i.e., the number of noisy function values averaged to compute the estimates) is $\mathcal{O}\left( \Dk^{-4}\right)$.}
%Details on how random estimates satisfying Assumption~\ref{efkfAssump} can be made available are provided in~\cite[Section~2.2]{dzahini2020expected} and~\cite[Section~2.3]{audet2019stomads}. 
%and are consequently not provided here again. 

The following result from~\cite{dzahini2020expected} is inspired by~\cite[Lemma~2.5]{paquette2018stochastic} and shows the relationship between the above variance assumption and the probability of obtaining bad estimates. It is crucial for the proof of Theorem~\ref{zerothOrdTheor}.

\begin{lemma}\label{SDDSlemma1}\cite[Lemma~1]{dzahini2020expected}
Under Assumption~\ref{efkfAssump}, the following hold for all $k\geq 0$ for the random process generated by Algorithm~\ref{algoStoDARS}:
\begin{eqnarray*}
\E{\ijkc\abs{\Fok-f(\Xk)}|\fqdalgebra}&\leq& \ef(1-\beta_f)\Dk^2\quad\mbox{almost surely, and}\\
\E{\ijkc\abs{\Fuk-f(\Xk+\Dk\urandom)}|\fqdalgebra}&\leq& \ef(1-\beta_f)\Dk^2\quad\mbox{almost surely.}
\end{eqnarray*} 
\end{lemma}

\begin{proof}
The proof is almost identical to that of~\cite[Lemma~1]{dzahini2020expected}.
\end{proof}
The proof of the main result of this section also requires the following standard assumption.
\begin{assumption}\label{boundedFassumpt}{\bl (Inspired by~\cite[Assumption~10.5]{CoScVibook},~\cite[Theorems~3.4 and Section~3.8]{KoLeTo03a})}
The function $f$ is bounded from below; that is, $-\infty<f_{\min}\leq f(\x)$ for some $f_{\min}\in\R$ and all $\x\in{\bl \mathcal{X}}$.
\end{assumption}

\begin{theorem}\label{zerothOrdTheor}
Let Assumptions~\ref{QkHaarAssumpt},~\ref{efkfAssump} and~\ref{boundedFassumpt} hold. 
%Let $\gamma>2$ and $\nu,\tau\in (0,1)$ satisfying $\frac{\nu}{1-\nu}\geq \frac{2(\tau^{-2}-1)}{\gamma-2}$. Assume that Assumption~\ref{efkfAssump} holds with $\beta_f\in (0,1)$ satisfying  $\frac{\beta_f}{1-\beta_f}\geq \frac{4\nu}{(1-\nu)(1-\tau^2)}$. 
Define $\varrho=\frac{1}{2}\beta_f(1-\nu)(1-\tau^2)$. Then, the random function $\Phik:=\frac{\nu}{\ef}\left(f(\Xk)-f_{\min}\right)+(1-\nu)\Dk^2$ satisfies 
\begin{equation}\label{PhikDynamics}
\E{\Phikun-\Phik|\fqdalgebra}\leq -\varrho\Dk^2 \quad\mbox{almost surely, for all}\ k\in \N, 
\end{equation}
which implies that 
\begin{equation}\label{finiteSeries}
\sum_{k=0}^{\infty}\Dk^2<\infty\quad\mbox{almost surely, whence}\qquad \underset{k\to\infty}{\lim}\Dk=0\quad\mbox{almost surely.}
\end{equation}
\end{theorem}

\begin{proof}
{\bbl See Appendix~\ref{th41}.}
\end{proof}

\section{Expected complexity analysis in the unconstrained case}\label{sec5}
By using an existing supermartingale-based framework proposed by Blanchet et al.~\cite{blanchet2016convergence}, this section derives in Theorem~\ref{complexResult} an upper bound on the expected number of iterations required by Algorithm~\ref{algoStoDARS} to drive~$\normii{\nabla f(\Xk)}$ below a given threshold $\epsilon\in (0,1)$, when applied to Problem~\eqref{probl1} in the particular case~$\bcX:=\rn$. Following~\cite{dzahini2020expected,DzaWildSub2022,paquette2018stochastic}, Section~\ref{sec51} introduces a general renewal-reward martingale process and its associated {\it stopping time} along with related assumptions. Then Section~\ref{sec52} demonstrates that these assumptions are satisfied by the stochastic process generated by Algorithm~\ref{algoStoDARS}, leading to the desired complexity result.

\subsection{General renewal-reward discrete time process and its stopping time}\label{sec51}

Next is recalled a formal definition of the stopping time associated with a discrete time stochastic process.
\begin{definition}\cite[Definition~1]{blanchet2016convergence}
$\Tepsilon$ is called a \emph{stopping time} with respect to the discrete time process $\accolade{\X_k}_{k\in\N}$ if, for all $k\in\N$, $\accolade{\Tepsilon=k}\in \sigma\left(\X_1,\dots,\X_k \right)$. 
\end{definition}

Consider the discrete time process $\accolade{\left(\Phik,\Dk\right)}_{k\in\N}$, with $\Phik, \Dk\in [0,\infty)$. On the same probability space as $\accolade{\left(\Phik,\Dk\right)}_{k\in\N}$, let us introduce the biased random walk process $\accolade{\Wtildek}_{k\in\N}$ satisfying
\begin{equation}\label{dynamicsWk}
\prob{\Wtildekun=1|\mathcal{H}_k}=q\qquad \mbox{ and }\qquad \prob{\Wtildekun=-1|\mathcal{H}_k}=1-q,
\end{equation}
where $\mathcal{H}_k=\sigma\left( \left(\PhiZero,\Dzero,\WtildeZero\right),\dots,\left(\Phik,\Dk,\Wtildek\right) \right)$ with ${\WtildeZero=1}$, and  $q\in (1/2, 1)$.

Let $\accolade{\Tepsilon}_{\epsilon>0}$ be a family of stopping times with respect to $\accoladekinN{\mathcal{H}_k}$, parameterized by $\epsilon$. In order to derive in~\cite{blanchet2016convergence,dzahini2020expected,DzaWildSub2022,paquette2018stochastic} an upper bound on $\E{\Tepsilon}$,  the following are assumed.

\begin{assumption}{\bl \cite{blanchet2016convergence}}\label{assumptionRenewal}
The discrete time process $\accoladekinN{\left(\Phik,\Dk,\Wtildek\right)}$ satisfies the following:
\begin{itemize}
\item[(i)] There exist $\lambda>0$ and $\dmax=\delta_0 e^{\lambda j_{\max}}$ such that $\Dk\leq \dmax$ for all $k$, where $j_{\max}\in\Z$.
\item[(ii)] For some $j_{\epsilon}\in\Z\cap (-\infty, 0]$, there exists $\depsilon=\delta_0 e^{\lambda j_{\epsilon}}$ such that for all $k$
\begin{equation}\label{dynamicsRenewal} 
\itksup\Dkun\geq \itksup\min\accolade{\Dk e^{\lambda\Wtildekun}, \depsilon}, 
\end{equation}
where $\Wtildekun$ satisfies~\eqref{dynamicsWk}. 
\item[(iii)] There exist a nondecreasing function $h:(0,\infty)\to (0,\infty)$ and a constant $\varrho>0$ such that for all $k$
\[\E{\Phikun-\Phik|\mathcal{H}_k}\itksup\leq-\varrho h\left(\Dk\right)\itksup\quad\mbox{almost surely.}\]
\end{itemize}
\end{assumption}
As emphasized in~\cite{blanchet2016convergence}, the event $\accolade{\Dk\geq\depsilon}$ occurs sufficiently frequently on average so that $\E{\Phikun-\Phik}$ can often be bounded by some negative fixed value (depending on~$\eps$), thus allowing an upper bound on the expected stopping time $\E{\Tepsilon}$, as stated by the next result.

\begin{theorem}\cite[Theorem~2]{blanchet2016convergence}\label{theorComplexity1}
Under Assumption~\ref{assumptionRenewal}, 
$\E{\Tepsilon}\leq \frac{q}{2q-1}\cdot\frac{\PhiZero}{\varrho h\left(\depsilon\right)}+1.$
\end{theorem}

\subsection{Expected complexity result}\label{sec52}
Consider the discrete time process $\accoladekinN{\left(\Phik,\Dk,\Wtildek\right)}$ defined by the random function~$\Phik$ introduced in Theorem~\ref{zerothOrdTheor}, the stepsize~$\Dk$, and the random variables $\Wtildek:=2\left(\iak\ijk-\frac{1}{2}\right)$, where 
\begin{equation}\label{AkEventDef}
A_k:=\accolade{\alpha_Q\kappa(\pollDrandom_k)\normii{\nabla f(\Xk)}\leq -\sqrt{\frac{n}{p}}\nabla f(\Xk)^{\top}\left(\Ukprandom{\ubar{\hat{\di}}^p_k}\right)}.
\end{equation}
For an arbitrary fixed $\epsilon\in (0,1)$, consider the random number $\Tepsilon:=\inf\accolade{k\in\N:\normii{\nabla f(\Xk)}\leq \eps}$ of iterations required by Algorithm~\ref{algoStoDARS} to drive the norm of the gradient of~$f$ below~$\epsilon$.
% \begin{equation}\label{randomTime}
% \Tepsilon:=\inf\accolade{k\in\N:\normii{\nabla f(\Xk)}\leq \eps}.
% \end{equation}
As in~\cite{blanchet2016convergence,dzahini2020expected,DzaWildSub2022,paquette2018stochastic}, the occurrence of the event $\accolade{\Tepsilon=k}$ is determined by observing $\left(\PhiZero,\Dzero,\WtildeZero\right),\dots,\left(\Phikmun,\Dkmun,\Wtildekmun\right)$, which means that~$\Tepsilon$ is a stopping time for the stochastic process generated by Algorithm~\ref{algoStoDARS}, and hence with respect to the filtration $\accoladekinN{\mathcal{H}_k}:=\accoladekinN{\falgebra}$. Bounding $\E{\Tepsilon}$ by applying Theorem~\ref{theorComplexity1} requires showing next that Assumption~\ref{assumptionRenewal} is satisfied for Algorithm~\ref{algoStoDARS}.

Regarding Assumption~\ref{assumptionRenewal} {\it (i)} and~{\it (iii)}, the arguments are basically those in~\cite{dzahini2020expected} and~\cite{DzaWildSub2022}, respectively. More precisely,~{\it (i)} is trivially satisfied for $\lambda:=-\log(\tau)$ while, on the other hand, it follows from~\eqref{PhikDynamics} and the inclusion $\falgebra\subset\fqdalgebra$ that almost surely
\begin{equation}\label{tower1}
\E{\Phikun-\Phik|\falgebra}=\E{\E{\Phikun-\Phik|\fqdalgebra}|\falgebra}\leq \E{-\varrho\Dk^2|\falgebra}=-\varrho\Dk^2.
\end{equation}
This shows that~{\it (iii)} is satisfied with the function $h:\scalebox{0.75}{$\vartheta$}\mapsto\scalebox{0.75}{$\vartheta$}^2$ since~\eqref{tower1} holds for any realization of the random variable $\E{\Phikun-\Phik|\falgebra}$, and hence on the event $\accolade{\Tepsilon>k}$ in particular.

Showing in Lemma~\ref{dynamicsLemma} how Assumption~\ref{assumptionRenewal}-{\it (ii)} is satisfied for Algorithm~\ref{algoStoDARS} requires the following intermediate results of Lemmas~\ref{AkLemma1} and~\ref{smallLemmaHaar}. In particular the result stated next, which is a trivial corollary of Proposition~\ref{propKappaD2}, shows that at least one of the random poll directions $\accolade{\Ukprandom\diRandom:\diRandom\in\pollDrandom_k}$ used by Algorithm~\ref{algoStoDARS} makes an acute angle with the negative gradient, with a sufficiently high, but fixed, probability, conditioned to the past. It thereby provides a stochastic variant of a result in~\cite{KoLeTo03a,Vicente2013} representing the cornerstone of the complexity analysis of deterministic direct-search, which corresponds to~\eqref{cosineDescentDkn} when $\vi=-\nabla f(\xk)$.

\begin{lemma}\label{AkLemma1}
Let Assumption~\ref{QkHaarAssumpt} hold, and consider ${\ubar{\hat{\di}}^p_k}:=\hat{\di}^p_k({\footnotesize{\X}},\qrandom,\drandom)\in\argmax\accolade{-\nabla f(\Xk)^{\top}\left(\QkRandom\diRandom\right):\diRandom\in\pollDrandom_k}$. Then the event $A_k$ defined in~\eqref{AkEventDef} satisfies $\prob{A_k|\falgebra}=\E{\iak|\falgebra}\geq 1-\beta_Q$ almost surely.
\end{lemma}
\begin{proof}
By construction, $\nabla f(\Xk)$ is $\falgebra$-measurable while $\QkRandom$ and $\pollDrandom_k$ are independent of $\falgebra$. The conditions of Proposition~\ref{propKappaD2} are therefore satisfied, and the result straightforwardly follows from~\eqref{kdprob2prime}.
\end{proof}

{\bl The following assumption on the cosine measure is made, not only for theoretical purposes, but also in order to prevent the PSSs used by Algorithm~\ref{algoStoDARS} from deteriorating (that is, ``becoming close to loosing the positive spanning property''~\cite{Vicente2013}). We note that it is neither strong nor unreasonable since, e.g., $\kappa(\pollD_k)=\frac{1}{\sqrt{p}}$ when $\pollD_k:=[\bm{I}_p -\bm{I}_p]$~(\cite[Page~408]{KoLeTo03a} and~\cite[Page~22]{CoScVibook}).
}
\begin{assumption}\label{kapminAssump}{\bl (\cite[Assumption~(3.12)]{KoLeTo03a}  and~\cite[Assumption~1]{Vicente2013})}
%To SW: above it's really Assumption~"(3.12)"; it's actually an equation, not a standalone assumption
There exists $\kappa_{\min}>0$ such that the realizations of the PSSs used by Algorithm~\ref{algoStoDARS} for subspace polling satisfy $\kappa(\pollD_k)\geq \kappa_{\min}$.
\end{assumption}
\begin{assumption}\label{gradLipAssump}{\bl (Inspired by~\cite[Theorem~2.8]{CoScVibook}, \cite[Theorem~3.3]{KoLeTo03a}, and~\cite[Theorem~1]{Vicente2013}).}
The gradient $\nabla f$ of the objective function is $L_g$-Lipschitz continuous {\bl in an open set containing all the points in the poll sets~$\mathcal{P}_k$}.
%, with $L_g>\kappa_{\min}$.
\end{assumption}

In the following result inspired by~\cite[Lemma~2]{dzahini2020expected}, $\depsilon$ is a constant defined by 
\begin{equation}\label{depsEq}
\depsilon:=\frac{\epsilon}{\xi}\qquad\mbox{with}\qquad\xi>\kappa_{\min}^{-1}\alpha_Q^{-1}\left(L_g\sqrt{\frac{n}{p}}+(\gamma+2)\ef\right)=:\hat{\xi}.
\end{equation}
Thus, as in~\cite{dzahini2020expected} and~\cite{blanchet2016convergence}, it can be assumed without any loss of generality that $\depsilon=\tau^{-i}\delta_0$ for some integer $i\leq 0$, so that $\Dk=\tau^{-i_k}\depsilon$ for some integer $i_k\leq 0$.
\begin{lemma}\label{smallLemmaHaar}
Let Assumptions~\ref{QkHaarAssumpt},~\ref{kapminAssump} and~\ref{gradLipAssump} hold, and assume that $\dk\leq \depsilon$. Let $\fok$ and $\fuk$ be $\ef$-accurate estimates of $f(\xk)$ and  $f(\xk+\dk\ub)$, respectively. Assume that the event $A_k$ occurs. If $\normii{\nabla f(\xk)}>\epsilon$, then there exists $\ub\in \mathbb{U}_k^n:=\{\Ukp\di^p:\di^p\in \pollD_k\}$ such that $\fuk-\fok\leq -\gamma\ef\dk^2$; that is, the iteration $k$ is successful. 
\end{lemma}

\begin{proof}
{\bbl See Appendix~\ref{lem52}.}
\end{proof}
%Assumptions~\ref{QkHaarAssumpt},~\ref{efkfAssump} and~\ref{boundedFassumpt}
The result stated next shows that the dynamics~\eqref{dynamicsRenewal} of Assumption~\ref{assumptionRenewal}-$(ii)$ holds.
\begin{lemma}\label{dynamicsLemma}
{\bl Let Assumption~\ref{QkHaarAssumpt} hold with $0<\beta_Q<1-\frac{1}{2\beta_f}$ and let Assumptions~\ref{QkHaarAssumpt},~\ref{efkfAssump},~\ref{boundedFassumpt},~\ref{kapminAssump}, and~\ref{gradLipAssump} hold.}
%and all assumptions made in Theorem~\ref{zerothOrdTheor} hold. 
Then Assumption~\ref{assumptionRenewal}-$(ii)$ is satisfied for $\Wtildek:=2\left(\iak\ijk-\frac{1}{2}\right)$, with $\lambda=-\log(\tau)$ and some fixed $q\in (\tilde{\beta}, 1)$, where $\tilde{\beta}=\beta_f(1-\beta_Q)$.
\end{lemma}

\begin{proof}
The result is proved by suitably adapting the proof of~\cite[Lemma~7]{blanchet2016convergence}, as was also done in~\cite[Lemma~3]{dzahini2020expected}, \cite[Lemma~5.3]{DzaWildSub2022} and~\cite[Lemma~4.8]{paquette2018stochastic}. First, observe that~\eqref{dynamicsRenewal} trivially holds when $\itksup=0$. It is shown next that when $\itksup=1$ (which is assumed in the remainder of the proof), then 
\[\Dkun\geq \min\accolade{\depsilon, \min\accolade{\dmax,\tau^{-1}\Dk}\iak\ijk+\tau\Dk\left(1-\iak\ijk\right)}.\] 
Recall that $\dk=\tau^{i_k}\depsilon$ for some integer $i_k$ so that if $\dk>\depsilon$, then $\dk\geq \tau^{-1}\depsilon$, whence $\dkun\geq \tau\dk\geq \depsilon$. Next, assume that $\dk\leq \depsilon$, and observe that $\normii{\nabla f(\xk)}>\eps$ since $\Tepsilon>k$. If $\iak\ijk=1$, it follows from Lemma~\ref{smallLemmaHaar} that the iteration is successful. Consequently, $\dkun=\min\accolade{\tau^{-1}\dk,\dmax}$. But if $\iak\ijk=0$, then it always holds that $\dkun\geq \tau\dk$. The proof is completed by observing that $\prob{\iak\ijk=1|\falgebra}=q$ for some $q\geq\tilde{\beta}$.
\end{proof}

The main complexity result of this section is derived next and shows that the expected number of iterations required by StoDARS to drive the norm of the gradient of~$f$ below~$\epsilon$ is $O\left(\epsilon^{-2}/(2\tilde{\beta}-1)\right)$.
\begin{theorem}\label{complexResult}
Let all assumptions made in Lemma~\ref{dynamicsLemma} hold. Then $\E{\Tepsilon}\leq \frac{\tilde{\beta}}{2\tilde{\beta}-1}\cdot\frac{\PhiZero{\tilde{\xi}}^2}{\varrho \eps^2}+1$, for some $\tilde{\xi}\geq\hat{\xi}$, where~{\bl $\varrho=\frac{1}{2}\beta_f(1-\nu)(1-\tau^2)$}, 
%is the constant of Theorem~\ref{zerothOrdTheor}, 
$\tilde{\beta}$ is the same parameter of Lemma~\ref{dynamicsLemma}, and $\hat{\xi}$ is given by~\eqref{depsEq}.
\end{theorem}

\begin{proof}
The proof, which trivially follows from Theorem~\ref{theorComplexity1}, is identical to that of~\cite[Theorem~5.4]{DzaWildSub2022}.
\end{proof}

\section{Convergence to Clarke stationary points}\label{sec6}

By exploiting the ability of StoDARS to generate a dense set of poll directions, the main goal of this section is to show that, with probability one, there exists a Clarke-{\it stationary} accumulation point~$\hat{\X}$ of the sequence $\accoladekinN{\Xk}$ of random iterates, that is, a point at which the Clarke generalized derivative of~$f$ is nonnegative in any direction that is hypertangent to~$\bcX$ at~$\hat{\X}$ with probability one. Recall that~$K_{\infty}$ denotes the random set of all unsuccessful iterations. In the remainder of the manuscript, the function~$f$ is no longer assumed to be continuously differentiable. Instead,~$f$ is assumed to be locally Lipschitz continuous throughout this section. We note that treating constraints suggests the use of a generalization of the Clarke derivative introduced by Jahn~\cite{Jahn2007}. The so-called Clarke--Jahn generalized derivative at $\hat{\x}\in \bcX$ in a direction $\vb\in \rn$ is defined by
\begin{equation*}%\label{ClarkeJahn}
f^{\circ}(\hat{\x}; \vb):=\underset{h\searrow 0,\, \x+h\vb\in \bcX}{\underset{\x\to \hat{\x},\, \x\in \bcX}{\limsup}} \frac{f(\x+h\vb)-f(\x)}{h}.
\end{equation*}
Stochastic variants of the well-known notions of {\it refining subsequence} and {\it refining directions} (see, e.g.,~\cite{AbAu06,AuDe2006,AuHa2017} and references therein) are introduced in the following definition.
\begin{definition}\label{refPointDirDef}
The sequence $\accoladekinK{\Xk}$ of iterates is called a \emph{refining subsequence} if the event $\mathscr{E}_\delta:=\accolade{\lim_{k\in K}\Dk = 0}$ is almost sure. {\bbbl Recall $\mathscr{E}_\mathbb{U}$ and $\UsetDrandom_k^n$ of Proposition~\ref{densityProp} and} let  $\hat{\urandom}:=\underset{k\in L}{\lim}\frac{\ukRandom}{\normii{\ukRandom}}$ be an almost sure limit\footnote{The limit $\hat{\urandom}$ exists thanks to the compactness of the closed unit ball.} {\bbbl on $\mathscr{E}_\delta\cap\mathscr{E}_\mathbb{U}$,} for some $L\subseteq K\subseteq K_{\infty}$ and poll directions $\ukRandom {\bbbl \, \in \UsetDrandom_k^n}$. If the refining subsequence\footnote{Implicitly assuming $\accoladekinK{\Xk}$ to be a refining subsequence means that $\prob{\mathscr{E}_\delta}=1$, and hence $\prob{\mathscr{E}_\delta\, {\bbbl \cap\,\mathscr{E}_\mathbb{U}}\cap \accolade{\lim_{k\in L}\ukRandom/\normii{\ukRandom}=\hat{\urandom}}}=1$, so that the inclusions $\Xk+\Dk\ukRandom\in \bcX$ hold for infinitely many $k\in L$, with probability one.} $\accoladekinK{\Xk}$ converges almost surely to a \emph{refined point} $\hat{\X}$ with realizations $\hat{\X}(\omega)\in\bcX$, with $\omega\in \mathscr{E}_\delta$, and if $\Xk(\omega)+\Dk(\omega)\ukRandom(\omega)\in \bcX$ for infinitely many $k\in L(\omega)$, for all $\omega\in \mathscr{E}_\delta\, {\bbbl \cap\,\mathscr{E}_\mathbb{U}} \cap \accolade{\lim_{k\in L}\ukRandom/\normii{\ukRandom}=\hat{\urandom}}$,  then $\hat{\urandom}$ is called a \emph{refining direction} for~$\hat{\X}$ with probability one.
\end{definition}

An important cone plays a central role in this analysis. Indeed, when analyzing StoDARS, one is concerned with the requirement that $\xk+\dk\uk\in\bcX$ for infinitely many~$k$ sufficiently large. Such a requirement is met, for example, when the sequence $\accolade{\uk}$ admits an accumulation point~$\hat{\ub}$ that belongs to the  {\it hypertangent cone} defined next.

\begin{definition}\cite[Definition~6.7]{AuHa2017}\label{hypTangentDef}
A vector $\ub\in\rn$ is called a \emph{hypertangent vector} to the set $\bcX$ at $\x\in\bcX$ if there exists $\varepsilon>0$ such that
\[\yb+h\vb\in \bcX\quad\mbox{for all}\ \yb\in\bcX\cap \mathcal{B}(\x,\varepsilon),\quad \vb\in \mathcal{B}(\ub,\varepsilon),\quad\mbox{and}\ 0<h<\varepsilon. \] 
The set of hypertangent vectors to $\bcX$ at $\x$ is called the \emph{hypertangent cone} to $\bcX$ at $\x$ and denoted by $T^H_{\bcX}(\x)$.
\end{definition}

We derive below a lemma that is key for obtaining the main result of Theorem~\ref{ClarkeTheor}. It will also play an important role in the proof of Theorem~\ref{secondOrderClarkeTheor}, that is, the main result of Section~\ref{sec7}.
\begin{lemma}\label{limsupWeakTail}
Let {\bl Assumptions~\ref{QkHaarAssumpt},~\ref{efkfAssump} and~\ref{boundedFassumpt} hold.}
%all the assumptions made in Theorem~\ref{zerothOrdTheor} hold.
Let $\hat{\X}$ be a refined point with realizations $\hat{\X}(\omega)\in \bcX$, $\omega\in \mathscr{E}_\delta$, for the refining subsequence $\accoladekinK{\Xk}$. Let $\hat{\urandom}$ be the almost sure limit of $\accoladekinL{\ukRandom}$, for some $L\subseteq K\subseteq K_{\infty}$, such that $\hat{\urandom}(\omega)\in T_{\bcX}^H(\hat{\X}(\omega))$ for all $\omega\in \mathscr{E}_{\delta,\ub, {\bbbl \mathbb{U}}}:=\mathscr{E}_\delta\, {\bbbl \cap\,\mathscr{E}_\mathbb{U}}\cap \accolade{\lim_{k\in L}\ukRandom=\hat{\urandom}}$, where $\ukRandom\in\, {\bbbl \UrandomSet_k^n} \subset\snOne$.
%$\ukRandom:=\Ukprandom\dikRandom\in\snOne$ with $\dikRandom\in\pollDrandom_k$. 
Then with probability one, $\hat{\urandom}$ is a refining direction for~$\hat{\X}$, and
\begin{equation}\label{limsupWeakTresult}
\underset{k\in L}{\limsup} \frac{f\left(\Xk+\Dk\ukRandom\right)-f(\Xk)}{\Dk}\geq 0\qquad\mbox{almost surely.}
\end{equation}
Likewise, any infinite random subset $L^0\subseteq L$ satisfies
\begin{equation}\label{limsupWeakTresultPrime}
\underset{k\in L^0}{\limsup} \frac{f\left(\Xk+\Dk\ukRandom\right)-f(\Xk)}{\Dk}\geq 0\qquad\mbox{almost surely.}
\end{equation}
\end{lemma}

\begin{remark}\label{refiningRemark}
Before proving Lemma~\ref{limsupWeakTail}, we mention that given a refined point $\hat{\X}$ and a sequence $\accoladekinK{\ukRandom}\subset\snOne$ of poll directions, it is not necessarily guaranteed that there exists a subsequence $\accoladekinL{\ukRandom}$ converging almost surely to some $\hat{\urandom}$ satisfying $\prob{\hat{\urandom}\in T_{\bcX}^H(\hat{\X})}=1$ as assumed above. Thus, the latter assumption might be too strong or unrealistic unless one is able to prove that $\accoladekinK{\ukRandom}$ is such that any arbitrary direction $\hat{\urandom}$ satisfying $\prob{\hat{\urandom}\in T_{\bcX}^H(\hat{\X})}=1$ is the almost sure limit of a subsequence $\accoladekinL{\ukRandom}$, which is possible, for example, in the context of Algorithm~\ref{algoStoDARS}, provided $\accoladekinK{\ukRandom}\subset\snOne$ is dense on the unit sphere $\snOne$. Fortunately, this is the case for StoDARS, as {\bbbl guaranteed by} Proposition~\ref{densityProp}.
\end{remark}

\begin{proof}
{\bbl See Appendix~\ref{lem61}.}
\end{proof}

%---------------------------------------------------

%$ $\\
%$ $\\
%$ $\\
%$ $\\
%Assume that the refining subsequence $\accolade{\Xk}_{k\in K}$ is convergent almost surely to a refined point $\hat{\X}$, and let $\urandom\in T^H_{\bcX}(\hat{\X})$ be an hypertangent direction. By density, there exists $L\subseteq K$ such that $\urandom=\lim_{k\in L}\ukRandom$ almost surely. Since $\lim_{k\in K}\Xk=\hat{\X}$ almost surely, then $\lim_{k\in L}\Xk=\hat{\X}$ almost surely as desired.
%\clearpage
The following result gives a lower bound on the Clarke--Jahn generalized derivative and will be useful for the proof of Theorem~\ref{ClarkeTheor} below.
\begin{lemma}\label{lowerBoundClarke}~\cite[Lemma~8.2]{AuHa2017}
Suppose $f$ is Lipschitz near $\hat{\x}\in \bcX$ and $\hat{\ub}\in T^H_{\bcX}(\hat{\x})$ is an hypertangent direction. Then 
\begin{equation*}%\label{JahnClarkeDef}
f^{\circ}(\hat{\x}; \hat{\ub})\geq \underset{h\searrow 0,\, \ub\to \hat{\ub},\, \x+h\ub\in \bcX}{\underset{\x\to \hat{\x},\, \x\in \bcX}{\limsup}} \frac{f(\x+h\ub)-f(\x)}{h}.
\end{equation*}
\end{lemma}

\begin{theorem}\label{ClarkeTheor}
Let all assumptions made in Lemma~\ref{limsupWeakTail} hold. Let $f$ be locally Lipschitz continuous. Let $\hat{\X}$ (with realizations $\hat{\X}(\omega)\in \bcX$, $\omega\in \mathscr{E}_{\delta}$) be the almost sure limit of a convergent refining subsequence $\accoladekinK{\Xk}$. Then for any $\hat{\urandom}\in\snOne$, with realizations $\hat{\urandom}(\omega)\in T^H_{\bcX}(\hat{\X}(\omega))$, $\omega\in \mathscr{E}_{\delta}\, {\bbbl \cap\,\mathscr{E}_\mathbb{U}}$, it holds that $f^{\circ}(\hat{\X}; \hat{\urandom})\geq 0$ with probability one. In other words, there exists an almost sure event $\mathscr{E}_{\delta,\ub,f, {\bbbl \mathbb{U}}}$ such that for all $\omega\in \mathscr{E}_{\delta,\ub,f, {\bbbl \mathbb{U}}}$, $f^{\circ}(\hat{\X}(\omega); \hat{\urandom}(\omega))\geq 0$.
\end{theorem}

\begin{proof}
Let $\hat{\urandom}\in\snOne$ be a random vector with realizations $\hat{\urandom}(\omega)\in T^H_{\bcX}(\hat{\X}(\omega))$, $\omega\in \mathscr{E}_{\delta}\, {\bbbl \cap\,\mathscr{E}_\mathbb{U}}$. It follows from Proposition~\ref{densityProp} that there exists $L\subseteq K\subseteq K_{\infty}$ such that $\lim_{k\in L}\ukRandom=\hat{\urandom}$ almost surely, which implies that $\mathscr{E}_{\delta,\ub, {\bbbl \mathbb{U}}}$ (defined as in Lemma~\ref{limsupWeakTail}) is almost sure. Thus, the fact that $\mathscr{E}_{\delta,\ub, {\bbbl \mathbb{U}}}\subseteq \accolade{\hat{\urandom}\in T_{\bcX}^H(\hat{\X})}$ ensures that $\hat{\urandom}$ is a refining direction for~$\hat{\X}$ with probability one. It therefore follows from Lemma~\ref{limsupWeakTail} that~\eqref{limsupWeakTresult} holds.

The remainder of the proof is inspired by that of~\cite[Theorem~8.3]{AuHa2017}, and the result is proved conditioned on the almost sure event $\mathscr{E}_{\delta,\ub,f, {\bbbl \mathbb{U}}}:=\mathscr{E}_{\delta,\ub, {\bbbl \mathbb{U}}}\cap \accolade{\underset{k\in L}{\limsup} \frac{f\left(\Xk+\Dk\ukRandom\right)-f(\Xk)}{\Dk}\geq 0}$. Let $\omega\in \mathscr{E}_{\delta,\ub,f, {\bbbl \mathbb{U}}}$, and  denote in general by $\ubar{z}(\omega)=:z$ the corresponding realization of a given random quantity $\ubar{z}$, and let $\mathcal{L}:=L(\omega)$. Since $\xk\to\hat{\x}$, $\uk\to\hat{\ub}$ and $\dk\to 0$ (because we conditioned on $\mathscr{E}_{\delta,\ub,f, {\bbbl \mathbb{U}}}$), it follows from Lemma~\ref{lowerBoundClarke} that
\[f^{\circ}(\hat{\X}(\omega); \hat{\urandom}(\omega))=f^{\circ}(\hat{\x}; \hat{\ub})\geq\underset{k\in \mathcal{L}}{\limsup}\frac{f(\xk+\dk\uk)-f(\xk)}{\dk}\geq 0,\]
which completes the proof.
\end{proof}

\section{Convergence to second-order stationary points}\label{sec7}

The analysis of second-order behavior of MADS algorithms for general constrained optimization~\cite{AbAu06} is extended in this section to the stochastic case for StoDARS algorithms by exploiting the ability of these methods to generate a dense set of poll directions. More precisely, inspired by~\cite{AbAu06}, reasonable conditions are proposed under which, with probability one, there exists an accumulation point for the sequence of random iterates generated by StoDARS satisfying a second-order necessary optimality condition for general set-constrained optimization problems. 

In order to establish the main result of Theorem~\ref{secondOrderClarkeTheor}, nonsmooth notions of second derivatives are introduced. In particular, we consider a so-called generalized Hessian, a second-order-like extension of the Rademacher's theorem-based definition of the Clarke subdifferential~\cite[Section~1.2]{Clar83a}, developed in~\cite{CoCo90a} and~\cite{HUStNg84a}. 
%a so-called generalized Hessian, is considered. 
In the present analysis the definition below due to Hiriart--Urruty, Strodiot, and Nguyen~\cite{HUStNg84a} is followed because of its consistency with the Clarke calculus for first-order derivatives, as highlighted in~\cite[Section~4.1]{AbAu06}. First, the notion of $C^{1,1}$ functions, which is central to the analysis, is introduced.

\begin{definition}\label{C11Def}\cite{AbAu06}
A function $g$ is $C^{1,1}$ around $\x$ if there exists an open set $\mathcal{S}$ containing $\x$ such that $g$ is continuously differentiable with Lipschitz derivatives for every point in $\mathcal{S}$.
\end{definition}

\begin{definition}\cite[Definition~4.1]{AbAu06}
Let the function $\zeta:\rn\to\R$ be $C^{1,1}$ around $\x\in\bcX\subseteq\rn$. The \emph{generalized Hessian} of $\zeta$ at $\x$, denoted by $\partial^2\zeta(\x)$, is the set of matrices defined as the convex hull of the set 
\[\accolade{\Amatrix\in\rnn:\exists\xk\to\x\ \mbox{with}\ \zeta\ \mbox{twice differentiable at}\ \xk\ \mbox{and}\ \nabla^2\zeta(\xk)\to\Amatrix}.\] 
\end{definition}

A second-order necessary optimality condition for constrained problems traditionally expressed in terms of the Lagrangian was extended in~\cite[Section~4.2]{AbAu06} for set-constrained problems, by establishing Clarke-based second-order necessary conditions for set-constrained optimality. It is expressed through Theorem~\ref{secOrdNecTheor} in terms of {\it feasible directions} formally defined below.

\begin{definition}\cite[Definition~4.3]{AbAu06}\label{feasDirDef}
The direction $\ub\in\rn$ is said to be \emph{feasible} for $\bcX\subset\rn$ at $\x\in \bcX$ if there exists $\varepsilon>0$ for which $\x+h\ub\in\bcX$ for all $0\leq h<\varepsilon$. The set of feasible directions for $\bcX$ at $\x\in\bcX$ is a cone denoted by~$T^F_{\bcX}(\x)$.
\end{definition}

\begin{theorem}\label{secOrdNecTheor}\cite[Theorem~4.4]{AbAu06}~(Second-order necessary condition for set-constrained optimality). If $f$ is $C^{1,1}$ around a local solution $\x^\star\in \bcX$ of Problem~\eqref{probl1}, then any feasible direction $\ub\in T^F_{\bcX}(\x^\star)$ for which $\ub^{\top}\nabla f(\x^\star)=0$ satisfies $\ub^{\top}\Amatrix\ub\geq 0$ for some $\Amatrix\in \partial^2 f(\x^\star)$.
\end{theorem}

In addition to the hypertangent cone and the cone of feasible directions, two other cones play a major role in this analysis and are introduced next.
\begin{definition}\cite[Definition~3.3]{AbAu06}\label{ClarkTangDef}
A vector $\ub\in\rn$ is called a \emph{Clarke tangent vector} to $\bcX\subset\rn$ at the point $\x\in\emph{cl}(\bcX)$ if, for every sequence $\accolade{\yb_k}$ of elements of $\bcX$ converging to~$\x$ and for every sequence $\accolade{h_k}$ of positive real numbers converging to zero, there exists a sequence $\accolade{\vi_k}$ of vectors converging to $\vi$ such that $\yb_k+h_k\vi_k\in \bcX$. The set of all Clarke tangent vectors to $\bcX$ at $\x$ is called the \emph{Clarke tangent cone} to $\bcX$ at $\x$, denoted by $T^{Cl}_{\bcX}(\x)$.
\end{definition}

\begin{definition}\cite[Definition~3.4]{AbAu06}\label{ContingDef}
A vector $\ub\in\rn$ is called a \emph{tangent vector} to $\bcX\subset\rn$ at the point $\x\in\emph{cl}(\bcX)$ if there exists a sequence $\accolade{\yb_k}$ of elements of $\bcX$ converging to~$\x$ and a sequence $\accolade{\lambda_k}$ of positive real numbers for which $\ub=\lim_k\lambda_k(\yb_k-\x)$. The set of all tangent vectors to $\bcX$ at $\x$ is called the \emph{contingent cone} or the \emph{sequential Bouligand tangent cone} to $\bcX$ at $\x$ and denoted by $T^{Co}_{\bcX}(\x)$.
\end{definition}

Before stating the main result of this section, some useful properties of the set $\partial^2\zeta(\x)$ and the set-valued mapping $\x\rightrightarrows\partial^2\zeta(\x)$ are recalled, along with a notion of {\it regularity} related to~$\bcX$.

\begin{proposition}\label{genHessSetVal}(\cite{HUStNg84a} and \cite[Section~4.1]{AbAu06})
Let the function $\zeta:\rn\to\R$ be $C^{1,1}$ around $\x\in\bcX\subseteq\rn$. Then the following hold:
\begin{itemize}
\item[(i)] $\partial^2\zeta(\x)$ is a \emph{nonempty, convex}, and \emph{compact} set of symmetric matrices.
\item[(ii)] $\partial^2\zeta$ is a \emph{locally bounded} set-valued mapping; that is, there exist $\varepsilon>0$ and $\eta\in\R_{+}$ such that 
\[\sup\accolade{\norme{\Amatrix}:\Amatrix\in\partial^2\zeta(\yb), \yb\in\mathcal{B}(\x,\varepsilon)}\leq\eta.\]
\item[(iii)] $\partial^2\zeta$ is \emph{upper-semicontinuous} (or \emph{closed}) set-valued mapping; that is,
\[\mbox{If}\ \xk\to\x\ \mbox{and}\ \Amatrix_k\to\Amatrix\ \mbox{with}\ \Amatrix_k\in\partial^2\zeta(\xk)\ \mbox{for all}\ k,\ \mbox{then}\ \Amatrix\in \partial^2\zeta(\x).\]
\end{itemize}
If $\zeta$ is $C^{1,1}$ in an open set $\mathcal{U}\subset\rn$ containing a line segment $[\bm{\gamma}_1,\bm{\gamma}_2]$, then
\begin{itemize}
\item[(iv)] there exists $\x\in (\bm{\gamma}_1,\bm{\gamma}_2)$ and a matrix $\Amatrix_{\x}\in \partial^2\zeta(\x)$ such that \[\zeta(\bm{\gamma}_2)=\zeta(\bm{\gamma}_1)+(\bm{\gamma}_2-\bm{\gamma}_1)^{\top}\nabla\zeta(\bm{\gamma}_1)+\frac{1}{2}(\bm{\gamma}_2-\bm{\gamma}_1)^{\top}\Amatrix_{\x}(\bm{\gamma}_2-\bm{\gamma}_1).\] 
\end{itemize}
\end{proposition}
\begin{definition}\cite[Definition~3.5]{AbAu06}\label{refularDef}
The set $\bcX$ is said to be \emph{regular} at $\x$ if $T^{Cl}_{\bcX}(\x)=T^{Co}_{\bcX}(\x)$.
\end{definition}

\begin{theorem}\label{secondOrderClarkeTheor}
%{Let Assumption~\ref{negInv} hold.} 
Let all assumptions made in Lemma~\ref{limsupWeakTail} hold. Let $f$ be $C^{1,1}$ around any realization $\hat{\X}(\omega)$, $\omega\in \mathscr{E}_{\delta}$ of the almost sure limit $\hat{\X}$ of a convergent refining subsequence. Assume that for all $\omega\in \mathscr{E}_{\delta}$, $T^H_{\bcX}(\hat{\X}(\omega))\neq \emptyset$ and that $\bcX$ is regular at $\hat{\X}(\omega)$. 
Then $\hat{\X}$ satisfies the second-order necessary condition for set-constrained optimality, with probability one.
\end{theorem}
%\mathscr{E}_{\delta,\ub, {\bbbl \mathbb{U}}}:=\mathscr{E}_\delta\, {\bbbl \cap\,\mathscr{E}_\mathbb{U}}\cap \accolade{\lim_{k\in L}\ukRandom=\hat{\urandom}}
\begin{proof}
Consider $\hat{\urandom}\in\snOne$ with realizations $\hat{\urandom}(\omega)$ such that $\hat{\urandom}(\omega)\in T^F_{\bcX}(\hat{\X}(\omega))$ and $\hat{\urandom}(\omega)^{\top}\nabla f(\hat{\X}(\omega))=0$, for all $\omega\in \mathscr{E}_{\delta}\, {\bbbl \cap\,\mathscr{E}_\mathbb{U}}$.
%$\prob{\hat{\urandom}^{\top}\nabla f(\hat{\X})=0}=1$. 
It follows from Proposition~\ref{densityProp} that there exists {\bbbl a subsequence $\accolade{\ukRandom}_{k\in L}$, for some $\ukRandom\in \UrandomSet_k^n$ and} $L\subseteq K\subseteq K_{\infty}$, such that $\lim_{k\in L}\ukRandom=\hat{\urandom}$ almost surely, which implies that $\mathscr{E}_{\delta,\ub, {\bbbl \mathbb{U}}}:=\mathscr{E}_\delta\, {\bbbl \cap\,\mathscr{E}_\mathbb{U}}\cap \accolade{\lim_{k\in L}\ukRandom=\hat{\urandom}}$  is almost sure.  We first observe that the conclusions~\eqref{limsupWeakTresult} and~\eqref{limsupWeakTresultPrime} of Lemma~\ref{limsupWeakTail} still hold with $T^H_{\bcX}$ replaced with $T^F_{\bcX}$ under the above regularity assumption and when $T^H_{\bcX}(\hat{\X}(\omega))\neq \emptyset$ for all $\omega\in \mathscr{E}_{\delta}\, {\bbbl \cap\,\mathscr{E}_\mathbb{U}}$. Indeed, the latter assumptions guarantee that $T^{Cl}_{\bcX}(\hat{\X}(\omega))= \mbox{cl}\left(T^H_{\bcX}(\hat{\X}(\omega))\right)=T^{Co}_{\bcX}(\hat{\X}(\omega))=\mbox{cl}\left(T^F_{\bcX}(\hat{\X}(\omega))\right)$ for any $\omega\in \mathscr{E}_{\delta}\, {\bbbl \cap\,\mathscr{E}_\mathbb{U}}$, so that $\hat{\urandom}(\omega)\in\mbox{cl}\left(T^F_{\bcX}(\hat{\X}(\omega))\right)= T^{Cl}_{\bcX}(\hat{\X}(\omega))$; that is, $\hat{\urandom}(\omega)$ is a Clarke tangent vector. Thus, the fact that $\mathscr{E}_{\delta,\ub, {\bbbl \mathbb{U}}}\subseteq \accolade{\hat{\urandom}\in T_{\bcX}^{Cl}(\hat{\X})}$ ensures (by definition of the Clarke tangent cone) that $\Xk(\omega)+\Dk(\omega)\ukRandom(\omega)\in \bcX$ for infinitely many $k\in L(\omega)$ sufficiently large and all $\omega\in \mathscr{E}_{\delta,\ub, {\bbbl \mathbb{U}}}$; that is, $\hat{\urandom}$ is a refining direction for~$\hat{\X}$ with probability one. All other reasoning in the proof of Lemma~\ref{limsupWeakTail} therefore applies, leading to~\eqref{limsupWeakTresult} and~\eqref{limsupWeakTresultPrime}.

%implying that the event $\mathscr{E}_{\delta,\ub,f,\nabla}:=\mathscr{E}_{\delta,\ub}\cap\accolade{\hat{\urandom}^{\top}\nabla f(\hat{\X})=0}\cap \accolade{\underset{k\in L}{\limsup} \frac{f\left(\Xk+\Dk\ukRandom\right)-f(\Xk)}{\Dk}\geq 0}$ is almost sure.
The remainder of the proof is conditioned on $\mathscr{E}_{\delta,\ub, {\bbbl \mathbb{U}}}$ and inspired by the proof of~\cite[Theorem~4.8]{AbAu06}. For an arbitrary $\omega\in \mathscr{E}_{\delta,\ub, {\bbbl \mathbb{U}}}$, denote in general by $z:=\ubar{z}(\omega)$ the corresponding realization of a given random quantity $\ubar{z}$, and let $\mathcal{L}:=L(\omega)$.
Since $\mathscr{E}_{\delta,\ub, {\bbbl \mathbb{U}}}\subseteq \accolade{\underset{k\in L^0}{\limsup} \frac{f\left(\Xk+\Dk\ukRandom\right)-f(\Xk)}{\Dk}\geq 0}$ with $L^0\subseteq L$ as explained above, then the conditioning on $\mathscr{E}_{\delta,\ub, {\bbbl \mathbb{U}}}$ implies in particular that $\underset{k\in \mathcal{L}^0}{\limsup}\cfrac{f(\xk+\dk\uk)-f(\xk)}{\dk}\geq 0$.
Recall that $\hat{\ub}^{\top}\nabla f(\hat{\x})=0$, $\lim_{k\in \mathcal{L}}\dk=0$ and $\lim_{k\in \mathcal{L}}\xk=\hat{\x}$ (since we conditioned on $\mathscr{E}_{\delta,\ub, {\bbbl \mathbb{U}}}$), and suppose by way of contradiction that $\hat{\ub}^{\top}\hat{\mathbf{A}}\hat{\ub}<0$ for all $\hat{\mathbf{A}}\in\partial^2 f(\hat{\x})$.
%Let $\hat{\ub}\in\rn$ be any nonzero feasible direction that satisfies $\hat{\ub}^{\top}\nabla f(\hat{\x})=0$, and suppose, by way of contradiction, that $\hat{\ub}^{\top}\hat{\mathbf{A}}\hat{\ub}<0$ for all $\hat{\mathbf{A}}\in\partial^2 f(\hat{\x})$. { (I skipped the ``$\varepsilon$-argument''.)}
%Let $K{\subseteq K_{\infty}}$ denote {a set} of indices of unsuccessful iterations { satisfying $\lim_{k\in K}\Delta_k^m=0$ and $\lim_{k\in K}\xk=\hat{\x}$}. Regularity of $\Omega$, together with the assumption that $T^H_\Omega(\hat{\x})\neq \emptyset$, guarantees that ${ T^{Cl}_\Omega(\hat{\x})=\ } \mbox{cl}\left(T^H_\Omega(\hat{\x})\right)=T^{Co}_\Omega(\hat{\x})=\mbox{cl}\left(T^F_\Omega(\hat{\x})\right)$. Therefore, the denseness of the refining directions in $T^H_\Omega(\hat{\x})$ ensures the existence of $\accolade{\wb_k}_{k\in {L}}$ { with $L\subseteq K$,} converging to $\hat{\ub}$ with $\wb_k:=\frac{\di_k}{\norme{\di_k}}, \di_k\in\pollD_k$, for each { $k\in L$}. 
%Applying Taylor series yields
It follows from Proposition~\ref{genHessSetVal}-{\it (iv)} that
\begin{eqnarray}
f(\xk+\dk\uk)-f(\xk) &=& \dk\uk^{\top}\nabla f(\xk) + \frac{1}{2}\dk^2\uk^{\top}\mathbf{A}^{+}_k\uk,\label{Eq43}\\
f(\xk-\dk\uk)-f(\xk) &=& -\dk\uk^{\top}\nabla f(\xk) + \frac{1}{2}\dk^2\uk^{\top}\mathbf{A}^{-}_k\uk,\label{Eq44}
\end{eqnarray}
where $\mathbf{A}^{+}_k\in \partial^2 f(\x^{+}_{ k})$ for some $\x^{+}_{k}\in (\xk, \xk+\dk\uk)$ and $\mathbf{A}^{-}_k\in \partial^2 f(\x^{-}_{ k})$ for some $\x^{-}_{ k}\in (\xk, \xk-\dk\uk)$. Recall Proposition~\ref{genHessSetVal}-{\it (i)--(iii)}. Since
{$\partial^2 f$ is locally bounded, then the sequence $\accolade{\mathbf{A}_k^{+}}_{k\in \mathcal{L}}$ is bounded and thus has} a subsequence {$\accolade{\mathbf{A}^{+}_k}_{k\in \mathcal{L}^{+}}$}, { with $\mathcal{L}^{+}\subseteq \mathcal{L}$, such that $\lim_{k\in \mathcal{L}^{+}}\mathbf{A}^{+}_k=:\mathbf{A}^{+}$}. Since {$\lim_{k\in \mathcal{L}^{+}}\dk=0$ and $\lim_{k\in \mathcal{L}^{+}}\xk=\hat{\x}$,  whence $\lim_{k\in \mathcal{L}^{+}}\x^{+}_{k}=\hat{\x}$, it follows from the upper-semicontinuity of $\partial^2 f$ that $\mathbf{A}^{+}\in \partial^2 f(\hat{\x})$.}
%{ $\lim_{k\in L}\Delta_k^m=0$ and $\lim_{k\in L}\xk=\hat{\x}$}, there is a subsequence { $\accolade{\mathbf{A}^{+}_k}_{k\in L^{+}}$}, { with $L^{+}\subseteq L$, such that $\lim_{k\in L^{+}}\mathbf{A}^{+}_k=:\mathbf{A}^{+}\in \partial^2 f(\hat{\x})$. 
Likewise, the boundedness of the sequence $\accolade{\mathbf{A}_k^{-}}_{k\in \mathcal{L}^{+}}$ implies the existence of a subsequence {$\accolade{\mathbf{A}^{-}_k}_{k\in \mathcal{L}^{0}}$}, {with $\mathcal{L}^{0}\subseteq \mathcal{L}^{+}$, such that $\lim_{k\in \mathcal{L}^{0}}\mathbf{A}^{-}_k=:\mathbf{A}^{-}\in \partial^2 f(\hat{\x})$, with the latter inclusion due to the upper-semicontinuity of $\partial^2 f$ along with the fact that $\lim_{k\in \mathcal{L}^{0}}\x^{-}_{k}=\hat{\x}$}.
%{Likewise, since $\lim_{k\in L^{+}}\Delta_k^m=0$ and $\lim_{k\in L^{+}}\xk=\hat{\x}$, there is a subsequence $\accolade{\mathbf{A}^{-}_k}_{k\in L^{0}}$, with $L^{0}\subseteq L^{+}$, such that $\lim_{k\in L^{0}}\mathbf{A}^{-}_k=:\mathbf{A}^{-}\in \partial^2 f(\hat{\x})$.} 
Moreover, since $\partial^2 f(\hat{\x})$ is a convex set, $\mathbf{A}_k:=\frac{1}{2}\left(\mathbf{A}^{+}_k+\mathbf{A}^{-}_k\right)$ {satisfies $\lim_{k\in \mathcal{L}^{0}} \mathbf{A}_k=$} $ \frac{1}{2}\left(\mathbf{A}^{+}+\mathbf{A}^{-}\right)=:{ \hat{\mathbf{A}}}\in \partial^2 f(\hat{\x})$. On the other hand, adding~\eqref{Eq43} and~\eqref{Eq44} yields
\begin{equation}\label{Eq45}
{\chi_k:=}\ \frac{1}{\dk}\left(\frac{f(\xk+\dk\uk)-f(\xk)}{\dk} + \frac{f(\xk-\dk\uk)-f(\xk)}{\dk}\right)=\uk^{\top}\mathbf{A}_k\uk,
\end{equation}
%The fact that $\hat{\ub}$ is a Clarke tangent vector to $\Omega$ at $\hat{\x}\in\mbox{cl}(\Omega)$ given that by assumption $\hat{\ub}\in \mbox{cl}\left(T^F_{\Omega}(\hat{\x})\right)=T^{Cl}_{\Omega}(\hat{\x})$, ensures that for $k\in L^0$ sufficiently large, $\xk+{ \Delta_k^m}\norme{\di_k}\wb_k = \xk+\Delta_k^m\di_k\in\Omega$. 
%Recalling that $\hat{\ub}$ is a Clarke tangent vector to $\bcX$ at $\hat{\x}\in\mbox{cl}(\bcX)$ and that $L^0$ is a subset of , then $f(\xk+{ \Delta_k^m}\norme{\di_k}\wb_k)-f(\xk)\geq 0$ for $k\in L^0$ sufficiently large. 
%The conditioning on $\mathscr{E}_{\delta,\ub,\nabla}$ ensures 
with $\underset{k\in \mathcal{L}^0}{\limsup}\cfrac{f(\xk+\dk\uk)-f(\xk)}{\dk}\geq 0$ as explained above.
Recalling that $\hat{\ub}^{\top}\nabla f(\hat{\x})=0$, we also observe that
%\[f^{\circ}(\hat{\x};-\hat{\ub})=\limsup_{k\in L^{0}} \frac{f(\xk-{\Delta_k^m}\norme{\di_k}\wb_k)-f(\xk)}{{\Delta_k^m}\norme{\di_k}} =(-f)^{\circ}(\hat{\x};\hat{\ub}) =(-\hat{\ub})^{\top}\nabla f(\hat{\x})=0.\]	
%By negation invariance of $\pollD_k$, it also holds that $f(\xk-{\Delta_k^m}\norme{\di_k}\wb_k)-f(\xk)\geq 0$ for $k\in L^0$ sufficiently large. 
\begin{equation}\label{Eqdirectional}
\lim_{k\in \mathcal{L}^0} \cfrac{f(\xk-\dk\uk)-f(\xk)}{\dk} =f'(\hat{\x}; -\hat{\ub})=(-\hat{\ub})^{\top}\nabla f(\hat{\x})=0.
\end{equation}
Thus, passing to the $\limsup$ in~\eqref{Eq45} when $k\in \mathcal{L}^0$ yields $0\leq \limsup_{k\in \mathcal{L}^0}\chi_k =\lim_{k\in \mathcal{L}^{0}}\uk^{\top}\mathbf{A}_k\uk = \hat{\ub}^{\top}\hat{\mathbf{A}}\hat{\ub} <0,$
%\[  \]
which is impossible, and the proof is completed.

\end{proof}

% \clearpage

\section{Numerical results}\label{sec8}
Inspired by~\cite[Section~6]{DzaWildSub2022} and~\cite[Section~4]{audet2019stomads} the performance of StoDARS is numerically investigated in this section, where many of its variants corresponding to various random subspaces 
%including the one corresponding to an instance of SDDS algorithms~\cite{dzahini2020expected}, 
are compared to an instance of SDDS algorithms~\cite{dzahini2020expected}, StoMADS~\cite{audet2019stomads} and STARS~\cite{DzaWildSub2022}. Stochastically noisy variants of~$40$ unconstrained deterministic problems of moderately large dimensions~$n$ are considered, where $n\in\intbracket{98, 125}$ (see~\cite{testprobs}). The objective functions defining these problems are sums of squares, that is, $f(\x):=\sum_{i=1}^m f_i(\x)^2$, where all $f_i$ are smooth functions\footnote{\bbl While the numerical experiments considered only contained true functions defined by the sum of squares of the generating functions, we have also experimented with the ``one-norm'' version (i.e., sum of absolute values of the generating functions) of this same test set and did not see significant changes for the variants reported on here.}. Two types of noise are considered: an {\it additive} noise and a {\it multiplicative} noise. In the former case (the only one considered in~\cite{audet2019stomads}), the functions~$f_i$ are additively perturbed respectively by independent random variables $\ubar{\theta}_i$, $i\in\intbracket{1,m}$, with mean zero, following the same distribution as $\ubar{\theta}$ leading to $f_{\ubar{\theta}}(\x):=\sum_{i=1}^m (f_i(\x)+\ubar{\theta}_i)^2$, so that $\Esp_{\ubar{\theta}}[f_{\ubar{\theta}}(\x)]=f(\x)+\sum_{i=1}^m \E{(\ubar{\theta}_i)^2}$, while in the latter case (the only one considered in~\cite{DzaWildSub2022}), $f_{\ubar{\theta}}(\x):=f(\x)(1+\ubar{\theta})$ so that $\Esp_{\ubar{\theta}}[f_{\ubar{\theta}}(\x)]=f(\x)$. We note however that in the additive case, optimization results are not affected by the constant bias term $\sum_{i=1}^m \E{(\ubar{\theta}_i)^2}$ since $\min_{\x} f(\x)=\min_{\x} \Esp_{\ubar{\theta}}[f_{\ubar{\theta}}(\x)]$. 
{\bl We emphasize the importance of examining both absolute and relative noise given that the initial function values $f(\x_0)$ vary significantly (from $10^{-6}$ to 
$10^{13}$).}
During the experiments, the random variable~$\ubar{\theta}$ with mean zero and standard deviation~$\sigma={\bl 10^{-3}}$ is either uniformly or normally distributed. Twenty replications corresponding to random seeds (common throughout the experiments) were performed for each of the~$40$ problems in order to take into account the variability due to the random subspaces or the stochasticity of $f_{\ubar{\theta}}$. Combined with the two aforementioned distributions, we have a total of~$1600$ {\it problem instances} for each type of noise.

For all numerical investigations of StoDARS, SDDS and StoMADS, only the poll step is used, with poll directions sorted relatively to the last successful direction in the noisy objective as was considered in~\cite{audet2019stomads} and~\cite{dzahini2020constrained}, that is, we do not use the optional search step. Moreover, the strategy known as the {\it opportunistic strategy}, consisting of an interruption of the poll step as soon as a better trial point is identified, is used during the experiments. For StoDARS variants, the random matrices $\Urandom_{k,\ntp}$ {\bl defining the random subspaces,} and the sets $\pollDrandom_k$ {\bl of} poll directions {\bl are} respectively obtained as follows. As described in Section~\ref{sec2Point2}, a ``QR-decomposition'' $\XfracRandom_k = \QfracRandom_k\RfracRandom_k$ via the Gram--Schmidt process is performed on random matrices $\XfracRandom_k\in\rnn$ from the real Ginibre ensemble, that is, with i.i.d. standard Gaussian entries. By considering the diagonal matrix with elements $({\bl \DfracRandomScal_k})_{ii}:=\sign((\RfracRandomScal_k)_{ii})$ where $\RfracRandom_k = ((\RfracRandomScal_k)_{ij})_{1\leq i,j\leq n}$, the matrices
$\Urandom_{k,\ntp}$ are defined by the first~$p$ columns of the Haar-distributed orthogonal matrices $\UfracRandom_k:=\QfracRandom_k{\bl \DfracRandom_k}$. On the other hand, denote by $\BrandomSet_k$ the orthogonal bases defined by the orthonormal columns of Haar-distributed matrices $\Brandom_k:=[{\ubar{\bm{b}}}_{k,1}\,\cdots\, \ubar{\bm{b}}_{k,p}]\in\mathbb{O}(p)$ generated {\bl independently of $\Urandom_{k,\ntp}$,} using the same strategy as above. Then the $\pollDrandom_k:=\BrandomSet_k\cup \accolade{-\ubar{\di}_s^p}\subset\mathbb{S}^{p-1}$ are minimal positive bases of~$\rp$, where $\ubar{\di}_s^p:=\left(\sum_{j=1}^p{\ubar{\bm{b}}}_{k,j}\right)/\normii{\sum_{j=1}^p{\ubar{\bm{b}}}_{k,j}}$. Thus, the poll sets $\mathbb{U}^n_k$ used by StoDARS are obtained as realizations of $\UrandomSet^n_k:=\accolade{\Urandom_{k,\ntp}\diRandom:\diRandom\in \pollDrandom_k}$, $k\in \N$. 

As first stochastic variant of MADS~\cite{AuDe2006}, StoMADS is a StoDARS-like algorithm where ``OrthoMADS~$2n$'' poll directions are constructed in full-space instead of subspaces, using e.g., a technique based on Householder matrices (see, e.g.,~\cite[Algorithm~2]{audet2019stomads} and~\cite[Algorithm~8.2]{AuHa2017}). This means that at each iteration, the algorithm uses a set of~$2n$ poll directions defining a maximal positive basis of~$\rn$. Moreover, StoMADS uses the same acceptance condition as StoDARS to accept new trial points which instead of having the freedom to be anywhere in the variable space, are constrained to belong to a {\it discretization} of that space called the {\it mesh}. On the other hand, SDDS is introduced in~\cite{dzahini2020expected} as a broad class of directional direct search algorithms that encompasses StoMADS and every stochastic direct search algorithm making use of poll directions defining a positive spanning set, and an acceptance condition of new points (that can be anywhere in the variable space) which in fact inspired that of StoDARS as clarified in the above sections. Thus, replacing in StoDARS the matrix $\Ukp$ by the identity matrix $\bm{I}_n\in\rnn$ while using a minimal positive basis $\mathbb{D}^n_k\in\mathbb{S}^{n-1}$ (that is, with $n+1$ directions) lead to an instance of SDDS which is therefore also referred to as SDDS in throughout the experiments. Recently introduced in~\cite{DzaWildSub2022} as a subspace variant of the stochastic trust-region STORM algorithm using full-space random models, and an extension to noisy objectives of a simplified version of the trust-region RSDFO algorithm~\cite{CRsubspace2021} using random  subspace models in a deterministic objective framework, STARS makes use of random models constructed in random subspaces. These subspaces, unlike StoDARS, can be constructed by means of various random matrices such as {\it hashing} matrices~\cite{DzaWildHashing2022,KaNel2014SparseLidenstrauss}, Haar distribution-based matrices as in the present manuscript, etc. Thus, motivated by the choices of $p\in\accolade{2,5}$ and Haar distribution-based subspaces  that yield the best observed performance of STARS, these variants of the method are compared to StoDARS, not only for the latter values of~$p$, but also for $p=n$ as was also considered in~\cite{DzaWildSub2022} in order to investigate the possible difference between the {\bl corresponding StoDARS} variant and ``StoDARS with $\Ukp:=\bm{I}_n$'' corresponding to SDDS as detailed above. {\bl We note that the STARS and StoDARS variants are respectively labeled ``STARS-$p$'' and ``StoDARS-$p$'', where $p\in\accolade{2,5,n}$.}

{\bl Let $n_k$ be the number of new noisy function evaluations performed at each iteration.} When estimating unknown function values using a Monte Carlo approach with $n_k=4$ (that yield an observed best performance of StoMADS) and $n_k=25$ (considered in STARS experiments~\cite[Section~6]{DzaWildSub2022}), we reuse available samples from previous iterations following the same approach described in the last paragraph of~\cite[Section~2.3]{audet2019stomads}, which was also used in~\cite[Section~6]{DzaWildSub2022} and~\cite[Section~5.2]{dzahini2020constrained}. Indeed, as recalled in~\cite[Section~6]{DzaWildSub2022}, 
let $\ubar{\theta}^0_\ell$, $\ell = 1,\dots,\pi_k$, and $\ubar{\theta}^{\ub}_\ell$, $\ell = 1,\dots,\pi_k$, be independent random samples of the independent random variables $\ubar{\theta}^0$ and $\ubar{\theta}^{\ub}$ following the same distribution as $\ubar{\theta}$. At iteration~$k$, let $n_0:=\pi_0$ and $n_k\leq \pi_k$ be the number of function evaluations used for an estimate's computation. {\bl When~$k$ is a} successful iteration, {\bl then the estimate at $\x_{k+1}$ is given by} $f^0_{k+1}=\frac{n_k\fuk+\sum_{\ell=n_k +1}^{\pi_{k+1}}f_{\theta^{\ub}_{\ell}}(\x_{k+1})}{\pi_{k+1}}\approx f(\x_{k+1})$, where $\fuk=\frac{1}{n_k}\sum_{\ell=1}^{n_k}f_{\theta^{\ub}_{\ell}}(\xk + \delta_k\ub)$ and $\pi_{k+1}=n_k+n_{k+1}$. {\bl This means that $f^0_{k+1}$ is obtained by averaging $n_k+n_{k+1}$ noisy function values including the~$n_k$ available samples at~$\xk+\dk\ub$ and the~$n_{k+1}$ newly generated samples at~$\x_{k+1}$.} When the iteration~{\bl $k$} is unsuccessful, then ${f^0_{k+1}=\frac{\pi_k\fok+\sum_{\ell=\pi_k +1}^{\pi_{k+1}}f_{\theta^0_{\ell}}(\x_{k+1})}{\pi_{k+1}}\approx f(\x_{k+1})}$, where $\pi_{k+1}=\pi_k+n_{k+1}$. {\bl This means that $\pi_k+n_{k+1}$ noisy function values, including the available $\pi_k$ samples at~$\xk$ and the~$n_{k+1}$ newly generated samples at~$\x_{k+1}$, are averaged to obtain~$f^0_{k+1}$.} While all STARS variants use the same algorithmic hyperparameters as in~\cite[Section~6]{DzaWildSub2022}, the direct-search algorithms share the same common algorithmic parameters. Thus, $\gamma\ef=1$, $\tau=0.5$ and $\delta_0=1$.

All the~$7$ algorithms are assessed by using data 
profiles~\cite{MoWi2009}. Next, for each type of noise and each of the 1,600 stochastically noisy problem instances, $\bm{x}_0$ is the starting point, $\bm{x}_N$ is the point with the best~$f$ {\bl (i.e., associated with a noise-free evaluation at $\bm{x}_N$)} function value found by an algorithm after~$N$ evaluations of the noisy function~$f_{\ubar{\theta}}$ while~$f^{\star}$ is the least such value found by all the algorithms. We consider a given problem to be successfully solved within a convergence tolerance~$\tau\in [0,1]$ after $N$ evaluations if
$%\label{convTest}
	f(\bm{x}_N)\leq f^{\star} + \tau(f(\bm{x}_0) - f^{\star}).
$
The vertical and horizontal axes of the data profiles show, respectively, the proportion of problems solved and the number of noisy function evaluations divided by~$(n+1)$. Throughout the experiments a budget of 1,500$(n+1)$ noisy function evaluations is allocated to all the algorithms.

The data profiles used to compare the algorithms are depicted in Figures~\ref{dataproft3s3ev4ty67}~--~\ref{dataproft3s3ev25ty45} for convergence tolerances~$\tau=10^{-2}$ and $\tau=10^{-3}$, and two sampling schemes defined by $n_k=4$ and $n_k=25$. All the data profiles show that in general, StoDARS-$5$ outperforms the other algorithms when given sufficient budget, especially when $n_k=4$ which results in very poor quality models (compared to $n_k=25$), affecting the performance of STARS variants. The latter, on the other hand, seem more robust for small budgets of function evaluations. As expected, regardless of the sampling schemes and the convergence tolerances, all the direct search algorithms using a full-space polling strategy (SDDS, StoDARS-$n$ and StoMADS) are outperformed by StoDARS-$2$ and StoDARS-$5$ which use fewer poll directions generated in random subspaces. When $p=n$, the matrix $\Ukp$ is orthogonal whence the vectors in the minimal positive bases used by StoDARS-$n$ are images, either by a reflection or a rotation, of the vectors in the positive bases used by SDDS, given that the transformation $\di\mapsto \Ukp\di$ can in that case be considered as a generalization of reflections and rotations. Thus, such positive bases of~$\rn$ have the same positive spanning properties as was also demonstrated in Proposition~\ref{Prop2Point3}. This explains the similarities observed in the performance of SDDS and StoDARS-$n$.

%Indeed, when p = n, the Haar orthogonal matrix multiplies the vectors in the minimal positive basis of R^n, leading to a new polling set which is not distorted compared to the one in "R^p". The new set will be 

\twocolumn 
\begin{figure}
\centering
\includegraphics[scale=0.22]{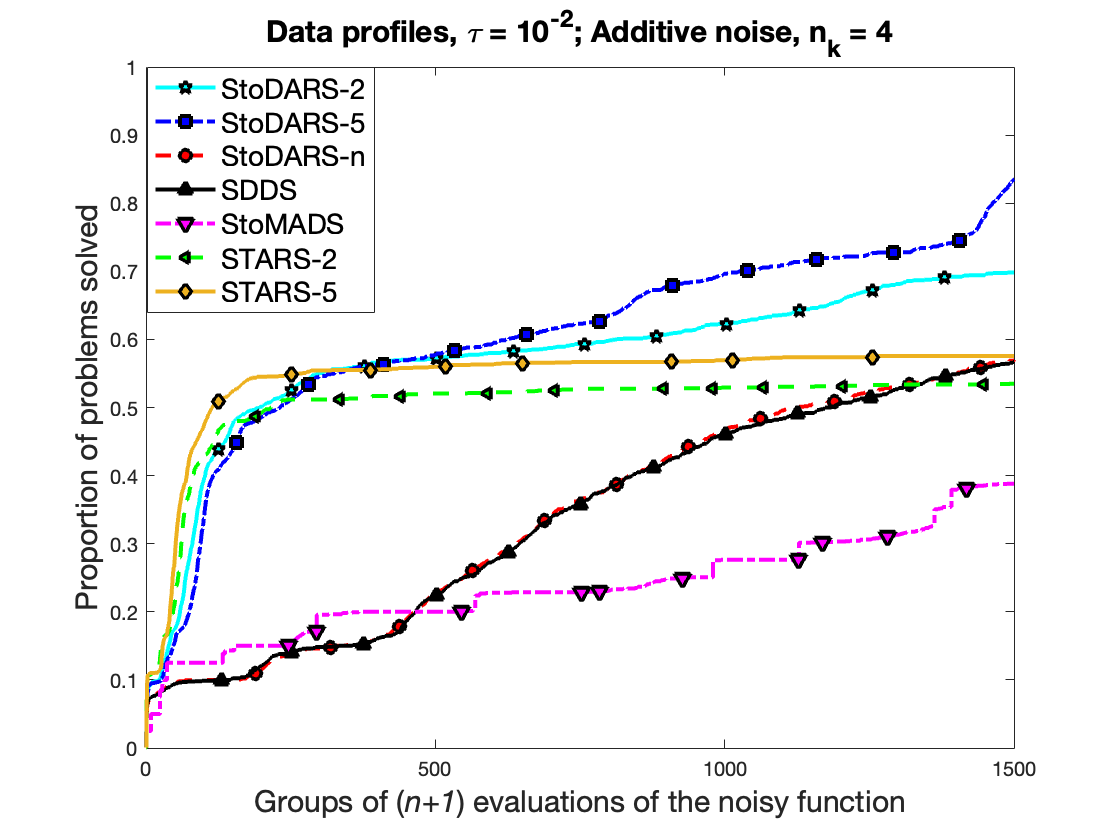}
\includegraphics[scale=0.22]{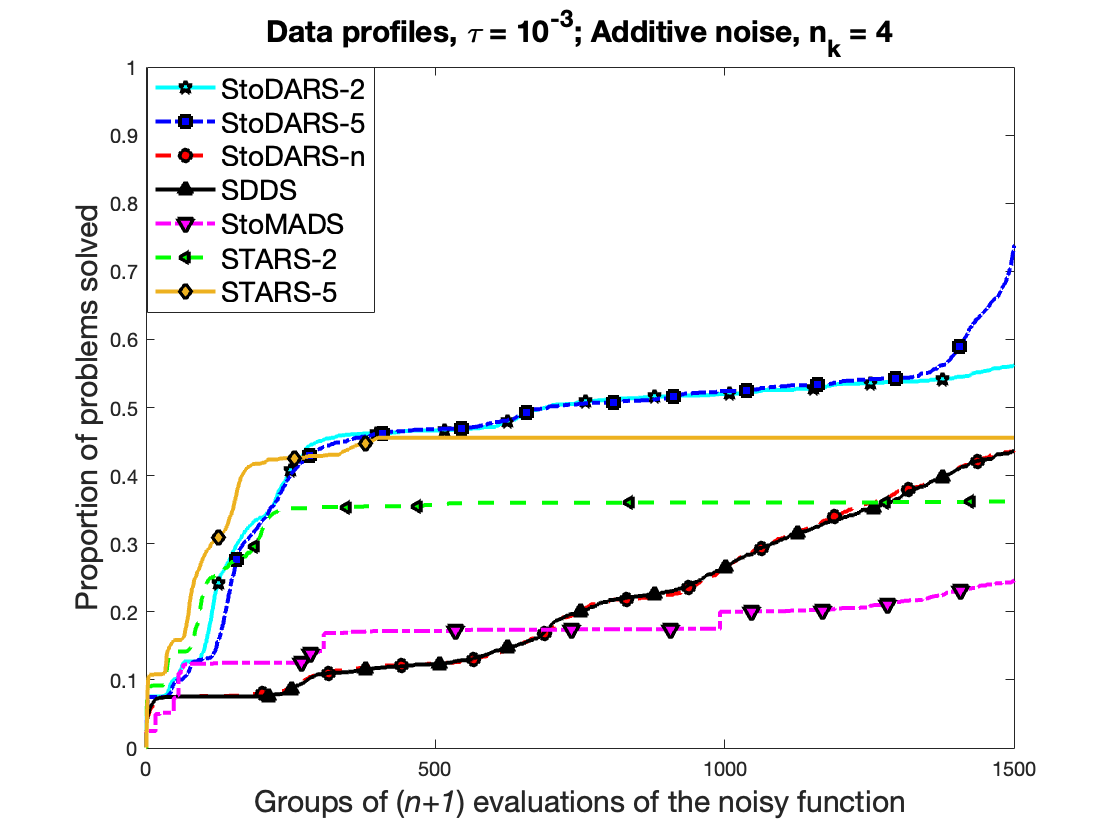}
\caption{\small{Data profiles for convergence tolerances $\tau=10^{-2}$ and $\tau=10^{-3}$, on $1,600$ problem instances for additive noise with standard deviation $\sigma=10^{-3}$, while using $n_k=4$ noisy function evaluations with reuse of available samples from previous iterations for the computation of estimates.}}
\label{dataproft3s3ev4ty67}
\end{figure}
\begin{figure}
\centering
\includegraphics[scale=0.22]{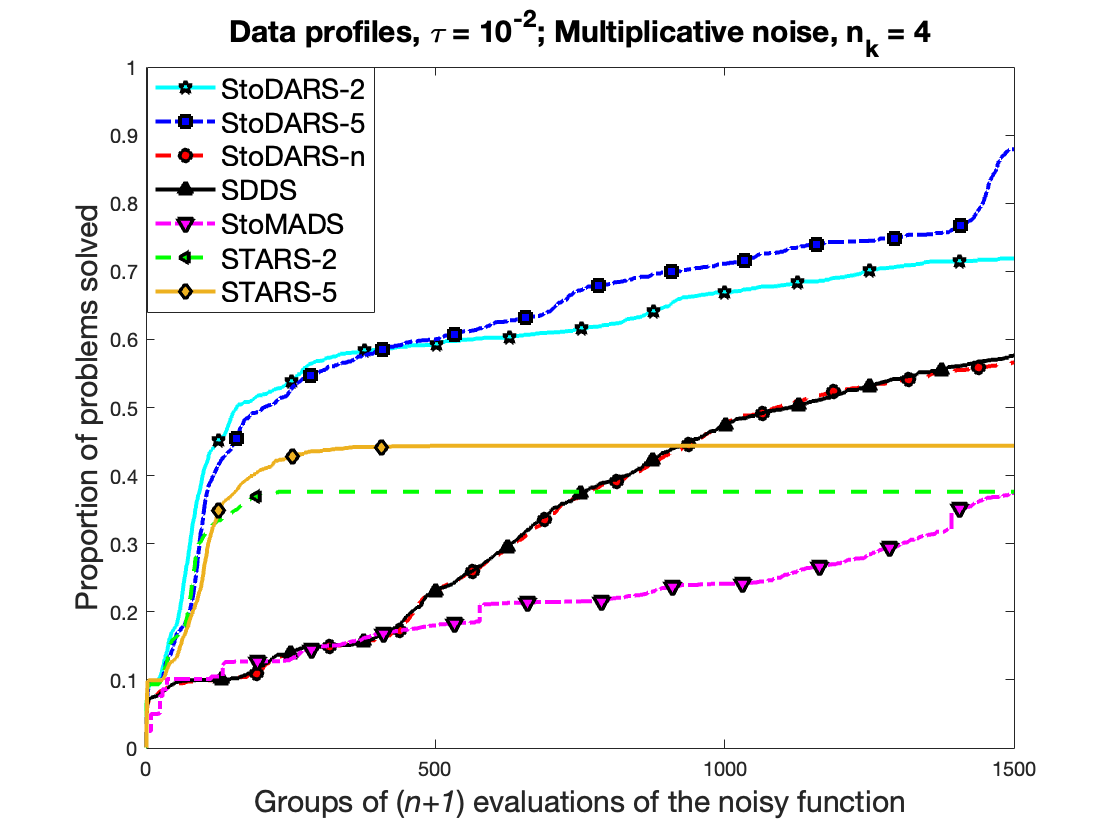}
\includegraphics[scale=0.22]{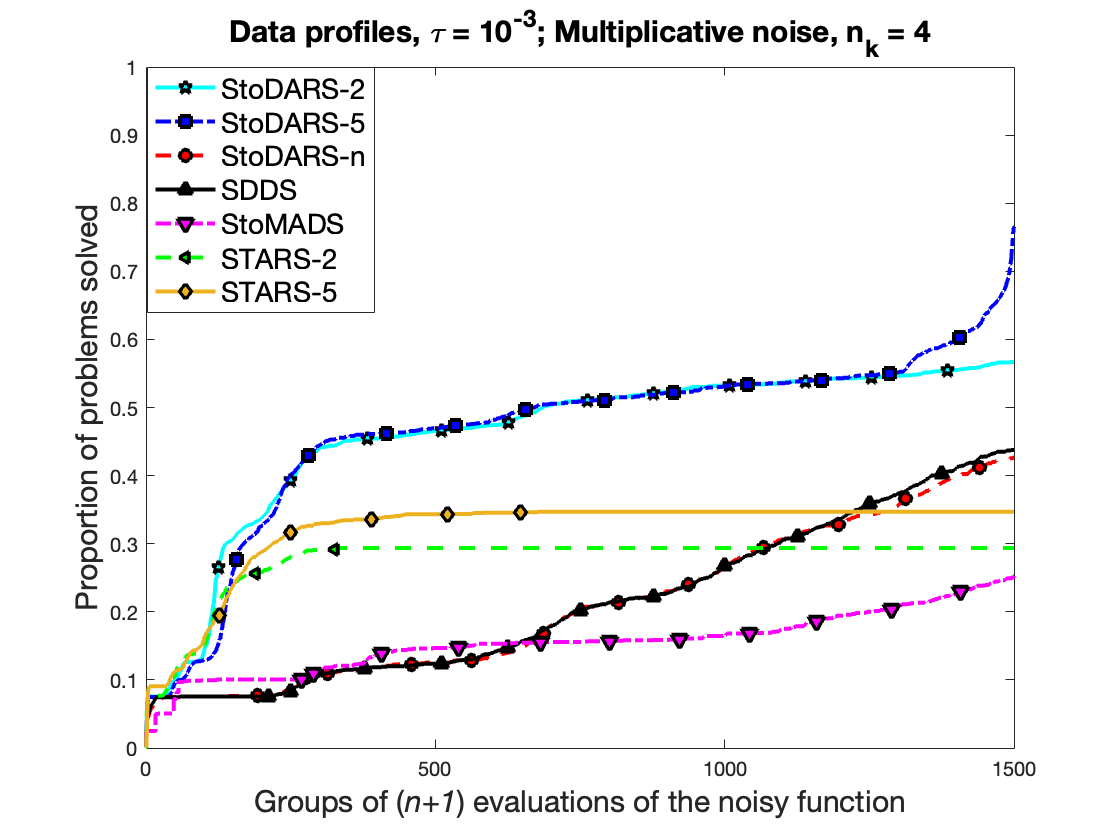}
\caption{\small{Data profiles for convergence tolerances $\tau=10^{-2}$ and $\tau=10^{-3}$, on $1,600$ problem instances for multiplicative noise with standard deviation $\sigma=10^{-3}$, while using $n_k=4$ noisy function evaluations with reuse of available samples from previous iterations for the computation of estimates.}}
\label{dataproft3s3ev4ty45}
\end{figure}
\onecolumn

\clearpage

\twocolumn 
%\clearpage
\begin{figure}
\centering
\includegraphics[scale=0.22]{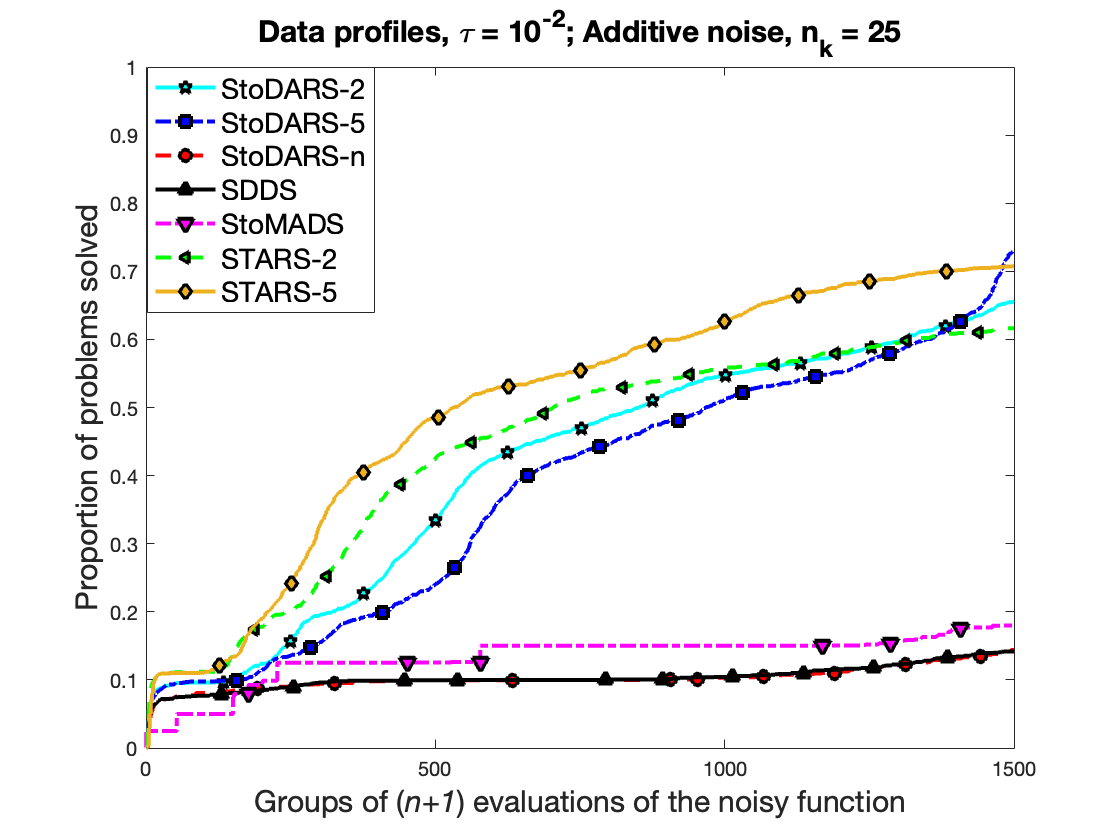}
\includegraphics[scale=0.22]{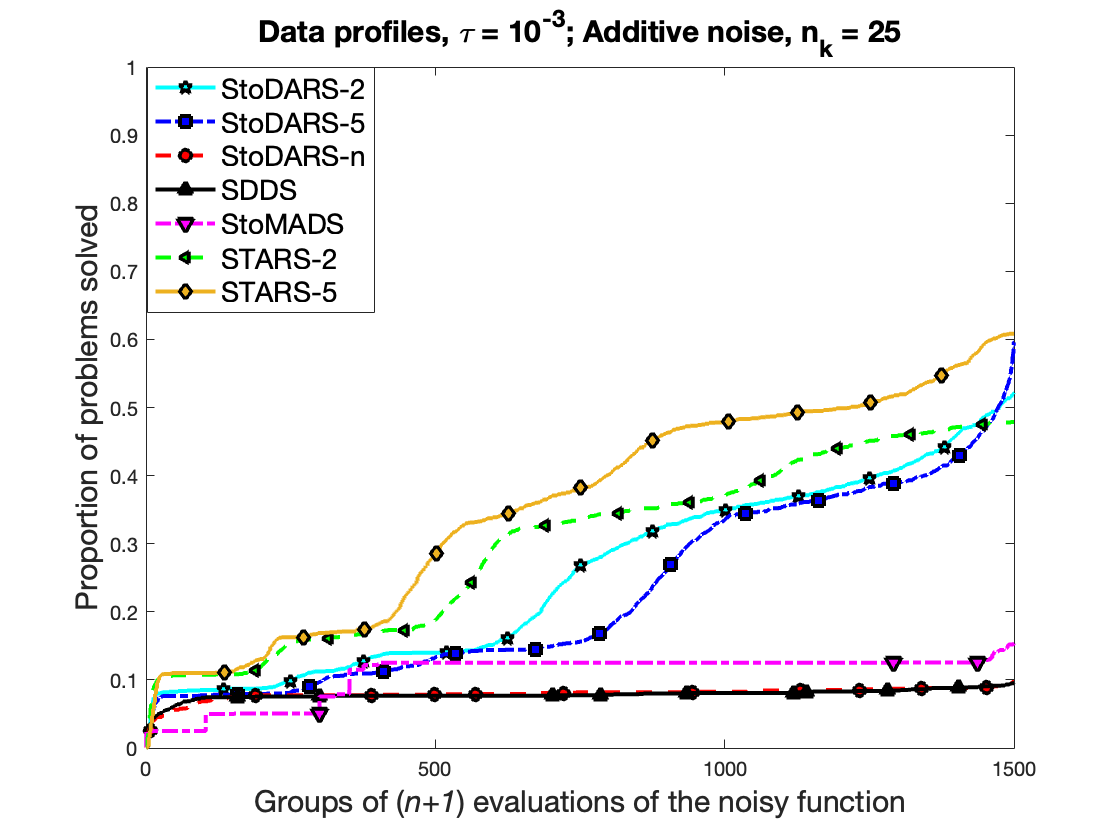}
\caption{\small{Data profiles for convergence tolerances $\tau=10^{-2}$ and $\tau=10^{-3}$, on $1,600$ problem instances for additive noise with standard deviation $\sigma=10^{-3}$, while using $n_k=25$ noisy function evaluations with reuse of available samples from previous iterations for the computation of estimates.}}
\label{dataproft3s3ev25ty67}
\end{figure}
\begin{figure}
\centering
\includegraphics[scale=0.22]{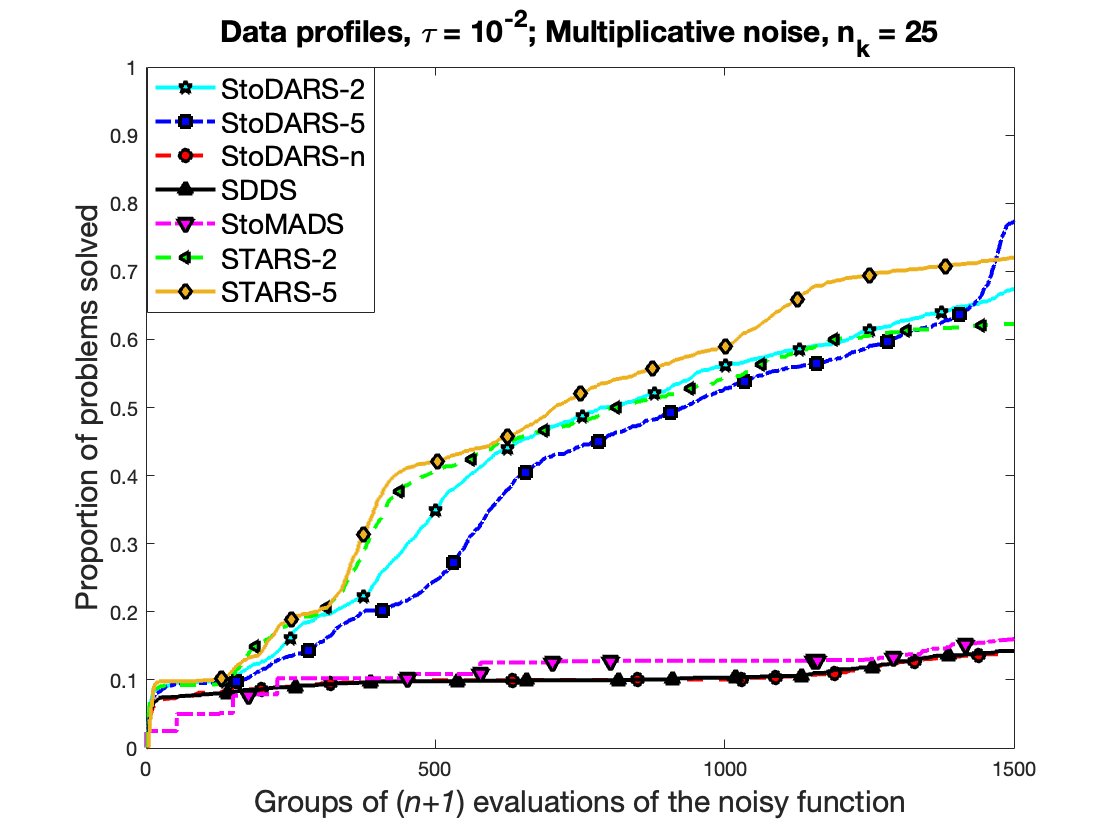}
\includegraphics[scale=0.22]{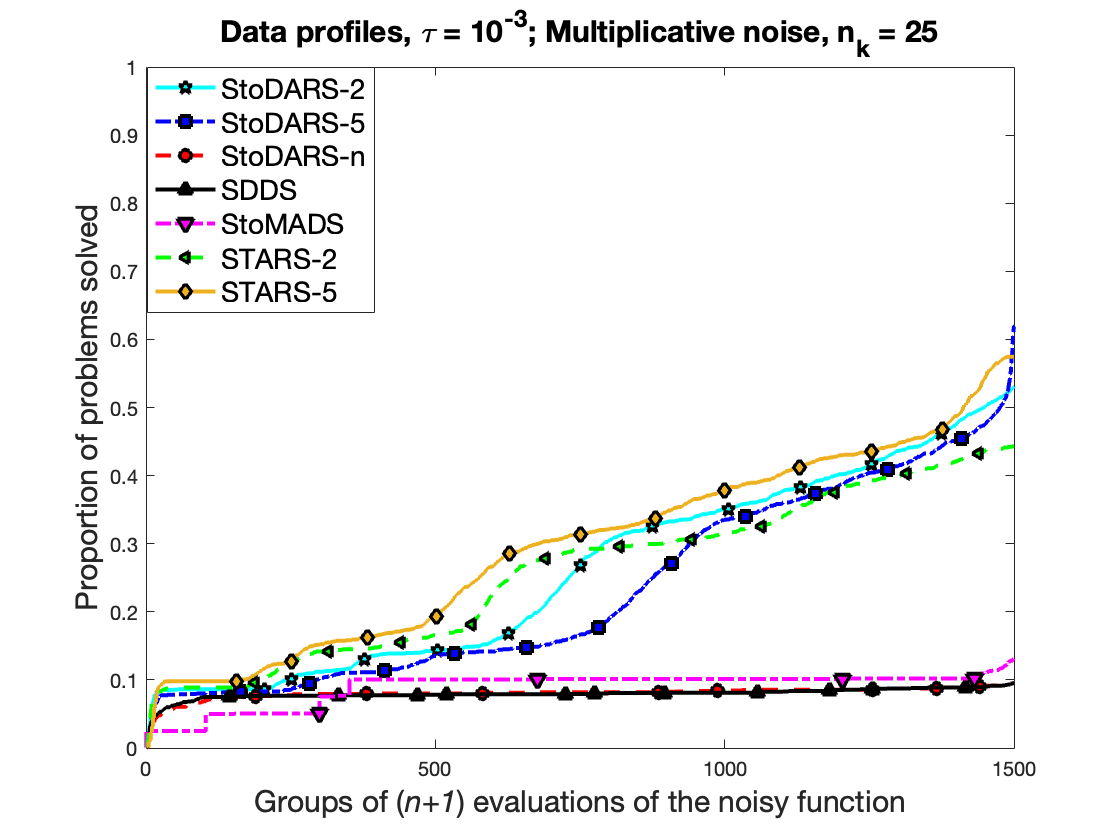}
\caption{\small{Data profiles for convergence tolerances $\tau=10^{-2}$ and $\tau=10^{-3}$, on $1,600$ problem instances for multiplicative noise with standard deviation $\sigma=10^{-3}$, while using $n_k=25$ noisy function evaluations with reuse of available samples from previous iterations for the computation of estimates.}}
\label{dataproft3s3ev25ty45}
\end{figure}

\onecolumn

\subsection*{Acknowledgments}
The first author thanks Professor Coralia Cartis for the discussions that motivated this work. This material was based upon work supported by the {\bl 
U.S.\ Department of Energy, Office of Science, Office of Advanced Scientific Computing Research applied mathematics and SciDAC (CAMPA) programs under Contract Nos.~DE-AC02-05CH11231 and DE-AC02-06CH11357.}
%\clearpage

\bibliography{dzahini-bibliography}
\bibliographystyle{abbrvnat} 

\clearpage

\appendix
\section{Proof of Theorem~\ref{JLforHaarTheor}}\label{theor22}
\begin{proof}
Without any loss of generality, it is sufficient to prove the result for $\x\in \mathbb{S}^{n-1}$, which is assumed throughout. Let $\bm{e_1}$ be the first standard basis vector of $\rn$. 
%By the translation invariance of Haar measure, 
For any $\x\in \mathbb{S}^{n-1}$ arbitrary, $\Urandom_{p\times n}\x \equald \Urandom_{p\times n} \bm{e_1}=(\UrandomScal_{11},\UrandomScal_{21},\dots, \UrandomScal_{p1})^{\top}$,
%. On the other hand, it follows from the definition of Haar measure that the first column of $\Urandom$ satisfies 
with $(\UrandomScal_{11}, \UrandomScal_{21},\dots, \UrandomScal_{n1})^{\top}\sim \mathcal{U}(\mathbb{S}^{n-1})$, hence $\Urandom_{p\times n}\x \equald \left(\xiRandom_1,\xiRandom_2,\dots, \xiRandom_p\right)^{\top}$, where $\xiRandomVec:=\left(\xiRandom_1,\xiRandom_2,\dots, \xiRandom_n\right)^{\top}\sim \mathcal{U}(\mathbb{S}^{n-1})$. 
%Consider the function $F:\mathbb{S}^{n-1}\to \R$ defined by 
Let \scalebox{0.88}{$F(x_1,x_2,\dots,x_n):=\sqrt{\frac{n}{p}}\normii{(x_1,x_2,\dots,x_p)}$}, so that \scalebox{0.88}{$F\left(\xiRandomVec\right)=\sqrt{\frac{n}{p}}\normii{\left(\xiRandom_1,\xiRandom_2,\dots, \xiRandom_p\right)}\equald \sqrt{\frac{n}{p}}\normii{\Urandom_{p\times n}\x}=\normii{\QrandomTranspose\x}$}. 
Then, for all $\x,\bm{y}\in \mathbb{S}^{n-1}$, $\abs{F(\x)-F(\bm{y})}\leq \sqrt{\frac{n}{p}}\normii{\x-\bm{y}}$,
% \begin{equation*}
% \begin{split}
% %\abs{F(\x)-F(\bm{y})}&=\sqrt{\frac{n}{p}}\abs{\normii{(x_1,x_2,\dots,x_p)}-\normii{(y_1,y_2,\dots,y_p)}}\\
% \abs{F(\x)-F(\bm{y})}&\leq \sqrt{\frac{n}{p}}\normii{(x_1-y_1),(x_2-y_2),\dots,(x_p-y_p)}\leq \sqrt{\frac{n}{p}}\normii{\x-\bm{y}},
% \end{split}
% \end{equation*}
which shows that $F$ is $\sqrt{\frac{n}{p}}$-Lipschitz on $\mathbb{S}^{n-1}$. Thus, it follows from Theorem~\ref{LeviTheorem} that there exists $\eta>0$ such that for all $\varepsilon_Q\in (0,1)$, 
{\footnotesize{
\begin{equation}\label{LeviEq1}
\prob{\abs{F\left(\xiRandomVec\right)-\E{F\left(\xiRandomVec\right)}}\geq \frac{\varepsilon_Q}{2}}\leq 2\exp{\left(-\frac{\eta n (\varepsilon_Q/2)^2}{(\sqrt{n/p})^2}\right)}=2e^{-\frac{\eta p\varepsilon_Q^2}{4}}.
\end{equation}
}}
%On the other hand, it follows from Lemma~\ref{HaarLemma} that $\E{(F\left(\xiRandomVec\right))^2}=1$, which implies $\var{F\left(\xiRandomVec\right)}+(\E{F\left(\xiRandomVec\right)})^2=1$. 
Moreover, it follows from the integral identity formula that 
{\footnotesize{
\begin{equation}\label{levEq2}
\var{F\left(\xiRandomVec\right)}=\E{\left(F\left(\xiRandomVec\right)-\E{F\left(\xiRandomVec\right)}\right)^2}=\int_0^{+\infty}\prob{\abs{F\left(\xiRandomVec\right)-\E{F\left(\xiRandomVec\right)}}^2\geq t} dt\leq\int_0^{+\infty} 2e^{-\eta p t} dt=\frac{2}{\eta p},
\end{equation}
}}
where the inequality follows from Theorem~\ref{LeviTheorem}. 
%It follows from~\eqref{levEq2} and the equality $\E{(F\left(\xiRandomVec\right))^2}=\var{F\left(\xiRandomVec\right)}+\left(\E{F\left(\xiRandomVec\right)}\right)^2=1$ (thanks to Lemma~\ref{HaarLemma}) that $\sqrt{1-\frac{2}{\eta p}}\leq \E{F\left(\xiRandomVec\right)}\leq 1$. 
From~\eqref{levEq2} and Lemma~\ref{HaarLemma}, we obtain $\sqrt{1-\frac{2}{\eta p}}\leq \E{F\left(\xiRandomVec\right)}\leq 1$.
Assume that $p$ is chosen so that $1-\frac{\varepsilon_Q}{2}\leq\sqrt{1-\frac{2}{\eta p}}$ and $2e^{-\frac{\eta p\varepsilon_Q^2}{4}}\leq \beta_Q$. Then $\abs{F\left(\xiRandomVec\right)-1}<\varepsilon_Q$ since $1-\frac{\varepsilon_Q}{2}\leq \E{F\left(\xiRandomVec\right)}\leq 1$.
%Since $1-\frac{\varepsilon_Q}{2}\leq \E{F\left(\xiRandomVec\right)}\leq 1$, then the inequality $\abs{F\left(\xiRandomVec\right)-\E{F\left(\xiRandomVec\right)}}<\frac{\varepsilon_Q}{2}$ implies that $\abs{F\left(\xiRandomVec\right)-1}<\varepsilon_Q$. 
Consequently, for all $\x\in \mathbb{S}^{n-1}$, it follows from~\eqref{LeviEq1} that
\[\prob{\abs{\normii{\QrandomTranspose\x}-1}\geq \varepsilon_Q}=\prob{\abs{F\left(\xiRandomVec\right)-1}\geq \varepsilon_Q}\leq\prob{\abs{F\left(\xiRandomVec\right)-\E{F\left(\xiRandomVec\right)}}\geq\frac{\varepsilon_Q}{2}}\leq2e^{-\frac{\eta p\varepsilon_Q^2}{4}}\leq \beta_Q,\]
where the equality is due to the equality in distribution $F\left(\xiRandomVec\right)\equald \normii{\QrandomTranspose\x}$ as explained above. Equivalently, it holds that $\prob{1-\varepsilon_Q\leq \normii{\QrandomTranspose\x}\leq 1+\varepsilon_Q}\geq 1-\beta_Q$ for all $\x\in \mathbb{S}^{n-1}$. It remains to prove that choosing $p\geq \max\accolade{\frac{8}{3\eta}\varepsilon_Q^{-2}, \frac{4}{\eta}\varepsilon_Q^{-2}\log\left(\frac{2}{\beta_Q}\right)}=\frac{4}{\eta}\varepsilon_Q^{-2}\log\left(\frac{2}{\beta_Q}\right)$ ensures $1-\frac{\varepsilon_Q}{2}\leq\sqrt{1-\frac{2}{\eta p}}$ and $2e^{-\frac{\eta p\varepsilon_Q^2}{4}}\leq \beta_Q$ as assumed above, where the latter equality trivially holds for any $\varepsilon_Q,\beta_Q\in (0,1)$. We first observe that $2e^{-\frac{\eta p\varepsilon_Q^2}{4}}\leq \beta_Q\iff p\geq \frac{4}{\eta}\varepsilon_Q^{-2}\log\left(\frac{2}{\beta_Q}\right)$. On the other hand, $\left(1-\frac{\varepsilon_Q}{2}\right)^2\leq 1-\frac{2}{\eta p}$ if $p\geq \frac{2}{\eta\left(\varepsilon_Q-\frac{\varepsilon_Q^2}{4}\right)}$.
%, which completes the proof. 
Since $\varepsilon_Q\in (0,1)$, then $\varepsilon_Q-\frac{\varepsilon_Q^2}{4}>\frac{3}{4}\varepsilon_Q^2$, which implies that $\frac{8}{3\eta}\varepsilon_Q^{-2}\geq \frac{2}{\eta\left(\varepsilon_Q-\frac{\varepsilon_Q^2}{4}\right)}$. Hence,  choosing $p\geq \frac{8}{3\eta}\varepsilon_Q^{-2}$ ensures $p\geq \frac{2}{\eta\left(\varepsilon_Q-\frac{\varepsilon_Q^2}{4}\right)}$.
%, and the proof is completed.
\end{proof}

\section{Proof of Theorem~\ref{zerothOrdTheor}}\label{th41}
\begin{proof}
The proof is similar to that of~\cite[Theorem~2]{dzahini2020expected}, using ideas derived in~\cite{blanchet2016convergence,chen2018stochastic,DzaWildSub2022,paquette2018stochastic}. It is presented here for the sake of clarity and self-containedness. The overall goal is to prove~\eqref{PhikDynamics} since taking the expectation on both its sides leads to~\eqref{finiteSeries} as detailed in~\cite[Theorem~3]{dzahini2020expected}, given that $\Phik>0$ for all~$k\in\N$. The proof considers two cases, {\it good estimates} and {\it bad estimates}, each of which is broken into whether the iteration is successful or not. Throughout, 
\[\mathfrak{S}_k:=\accolade{\mbox{The}\ k\mbox{th iteration is successful}}.\]
{\bf Case~1:} {\it (Good estimates, $\ijk=1$).} The overall goal is to prove that almost surely, 
\begin{equation}\label{treizeEq}
\E{\ijk\left(\Phikun-\Phik\right)|\fqdalgebra}\leq -\beta_f(1-\nu)(1-\tau^2)\Dk^2.
\end{equation}
{{Subcase~$1i$}:} {\it (Successful iteration, $\isk=1$).} A decrease occurs in~$f$ according to~\eqref{decreaseEqu}, and $\Dk$ is updated according to~$\Dkun=\min\accolade{\tau^{-1}\Dk,\dmax}$, whence $\ijk\isk\frac{\nu}{\ef}(f(\Xkun)-f(\Xk))\leq -\ijk\isk\nu(\gamma-2)\Dk^2$ and $\ijk\isk(1-\nu)(\Dkun^2-\Dk^2)\leq \ijk\isk(1-\nu)(\tau^{-2}-1)\Dk^2$, leading to 
\begin{equation}\label{dixseptEq}
\ijk\isk\left(\Phikun-\Phik\right)\leq\ijk\isk\left(-\nu(\gamma-2)+(1-\nu)(\tau^{-2}-1)\right)\Dk^2\leq -\ijk\isk(1-\nu)(1-\tau^2)\Dk^2,
\end{equation}
where the last inequality follows from the above choice of~$\nu$ along with the fact that $1-\tau^2<\tau^{-2}-1$. \\
{{Subcase~$1ii$}:} {\it (Unsuccessful iteration, $\iskc=1$).} The parameter~$\Dk$ is decreased according to $\Dkun=\tau\Dk$, and $f(\Xkun)=f(\Xk)$, whence
\begin{equation}\label{dixhuitEq}
\ijk\iskc\left(\Phikun-\Phik\right)=-\ijk\iskc(1-\nu)(1-\tau^2)\Dk^2.
\end{equation}
Equations~\eqref{dixseptEq} and~\eqref{dixhuitEq} lead to
\begin{equation}\label{vingtEq}
\ijk\left(\Phikun-\Phik\right)\leq -\ijk(1-\nu)(1-\tau^2)\Dk^2.
\end{equation}
Taking the conditional expectation with respect to $\fqdalgebra$ on both sides of~\eqref{vingtEq} and using Assumption~\ref{efkfAssump}-{\it (i)} lead to~\eqref{treizeEq}.
{\bf Case~2:} {\it (Bad estimates, $\ijkc=1$).} Because of bad estimates, an iterate leading to an increase in~$f$ and~$\Dk$, and hence in~$\Phik$, can be accepted by Algorithm~\ref{algoStoDARS}. Such an increase is controlled by using Assumption~\ref{efkfAssump}-{\it (ii)}. The overall goal is to show that almost surely,
\begin{equation}\label{vingtEtUnEq}
\E{\ijkc\left(\Phikun-\Phik\right)|\fqdalgebra}\leq 2\nu(1-\beta_f)\Dk^2.
\end{equation}
{{Subcase~$2i$}:} {\it (Successful iteration, $\isk=1$).} Since $\Fuk-\Fok\leq-\gamma\ef\Dk^2$, then 
\[\ijkc\isk\frac{\nu}{\ef}(f(\Xkun)-f(\Xk))\leq\ijkc\isk\nu\left(-\gamma\Dk^2+\frac{1}{\ef}\left(\abs{f(\Xkun)-\Fuk} + \abs{\Fok-f(\Xk)}\right)\right).\]
Recalling that $\Dkun=\min\accolade{\tau^{-1}\Dk,\dmax}$, and by noticing that $-\nu\gamma+(1-\nu)(\tau^{-2}-1)\leq 0$ due to the above choice of $\nu$, it holds that
% \begin{equation}\label{vingtcinqEq}
% \begin{split}
% \ijkc\isk\left(\Phikun-\Phik\right)&\leq \ijkc\isk\left(\left(-\nu\gamma+(1-\nu)(\tau^{-2}-1) \right)\Dk^2+\frac{\nu}{\ef}\left(\abs{f(\Xkun)-\Fuk} + \abs{\Fok-f(\Xk)}\right) \right)\\
% &\leq \ijkc\isk\left(\frac{\nu}{\ef}\left(\abs{f(\Xkun)-\Fuk} + \abs{\Fok-f(\Xk)}\right) \right)
% \end{split}
% \end{equation}
\begin{equation}\label{vingtcinqEq}
\ijkc\isk\left(\Phikun-\Phik\right)\leq \ijkc\isk\frac{\nu}{\ef}\left(\abs{f(\Xkun)-\Fuk} + \abs{\Fok-f(\Xk)}\right).
\end{equation}
{{Subcase~$2ii$}:} {\it (Unsuccessful iteration, $\iskc=1$).} The upper bound on the change in~$\Phik$ follows from~\eqref{dixhuitEq} by replacing~$\ijk$ with~$\ijkc$. Thus, $\ijkc\iskc\left(\Phikun-\Phik\right)\leq \ijkc\iskc\frac{\nu}{\ef}\left(\abs{f(\Xkun)-\Fuk} + \abs{\Fok-f(\Xk)}\right)$, which, combined with~\eqref{vingtcinqEq}, yields
\begin{equation}\label{vingtseptEq}
\ijkc\left(\Phikun-\Phik\right)\leq \ijkc\frac{\nu}{\ef}\left(\abs{f(\Xkun)-\Fuk} + \abs{\Fok-f(\Xk)}\right).
\end{equation}
Taking the conditional expectation with respect to $\fqdalgebra$ on both sides of~\eqref{vingtseptEq} and applying Lemma~\ref{SDDSlemma1} lead to~\eqref{vingtEtUnEq}. Then, combining~\eqref{treizeEq} and~\eqref{vingtEtUnEq} yields
\[\E{\left(\Phikun-\Phik\right)|\fqdalgebra}\leq (-\beta_f(1-\nu)(1-\tau^2)+2\nu(1-\beta_f))\Dk^2\leq -\varrho\Dk^2, \]
where the last inequality follows from the above choice of $\beta_f$, and the proof is completed.
\end{proof}

\section{Proof of Lemma~\ref{smallLemmaHaar}}\label{lem52}

\begin{proof}
The proof is inspired by that of~\cite[Lemma~2]{dzahini2020expected} using elements derived in~\cite{KoLeTo03a}. Suppose that $\dk\leq \depsilon$, that the event $A_k$ occurs, and that the estimates $\fok$ and $\fuk$ are $\ef$-accurate. Then, assume by way of contradiction that $\fuk-\fok> -\gamma\ef\dk^2$ for all $\ub\in \mathbb{U}_k^n$ of the form $\ub:=\Ukp\di^p$, with $\di^p\in \pollD_k$. Since $A_k$ occurs, then the latter inequality holds in particular for $\ub=\Ukp\hat{\di}^p_k$, with $\hat{\di}^p_k\in\argmax\accolade{-\nabla f(\xk)^{\top}\left(\Qk\di^p\right):\di^p\in\pollD_k}$ satisfying 
\begin{equation}\label{AkRealization}
\alpha_Q\kappa(\pollD_k)\normii{\nabla f(\xk)}\leq -\sqrt{\frac{n}{p}}\nabla f(\xk)^{\top}\left(\Ukp{\hat{\di}^p_k}\right).
\end{equation}
It follows from the $\ef$-accuracy assumption, along with the equality 
\[f(\xk+\dk\ub)-f(\xk)=f(\xk+\dk\ub)-\fuk+(\fuk-\fok)+\fok-f(\xk),\]
that 
\begin{equation}\label{eq39Acc}
f\left(\xk+\dk\Ukp\hat{\di}^p_k\right)-f(\xk)+(\gamma+2)\ef\dk^2\geq 0.
\end{equation}
It follows from the mean value theorem and~\eqref{eq39Acc} that there exists $\mu_k\in[0,1]$ such that 
\begin{equation}\label{eq40MeanVal}
0\leq \dk\nabla f\left(\xk+\mu_k\dk\Ukp\hat{\di}^p_k\right)^{\top}\left(\Ukp{\hat{\di}^p_k}\right)+(\gamma+2)\ef\dk^2.
\end{equation}
Multiplying both sides of~\eqref{eq40MeanVal} by $\frac{1}{\dk}\sqrt{\frac{n}{p}}$ and subtracting $\sqrt{\frac{n}{p}}\nabla f(\xk)^{\top}\left(\Ukp{\hat{\di}^p_k}\right)$ yield
\begin{equation}\label{eq41yield}
-\sqrt{\frac{n}{p}}\nabla f(\xk)^{\top}\left(\Ukp{\hat{\di}^p_k}\right)\leq\sqrt{\frac{n}{p}}\left(\nabla f\left(\xk+\mu_k\dk\Ukp\hat{\di}^p_k\right)- \nabla f(\xk)\right)^{\top}\left(\Ukp{\hat{\di}^p_k}\right)+(\gamma+2)\ef\dk,
\end{equation}
which, combined with~\eqref{AkRealization}, leads to
\begin{equation}\label{eq42Norm}
\begin{split}
\alpha_Q\kappa(\pollD_k)\normii{\nabla f(\xk)}&\leq L_g\sqrt{\frac{n}{p}}\normii{\mu_k\dk\Ukp{\hat{\di}^p_k}}\normii{\Ukp{\hat{\di}^p_k}}+(\gamma+2)\ef\dk \leq \left(L_g\sqrt{\frac{n}{p}}+(\gamma+2)\ef\right)\depsilon,%\\
%&\leq L_g\sqrt{\frac{n}{p}}\dk+(\gamma+2)\ef\dk
\end{split}
\end{equation}
where the last inequality follows from the $L_g$-Lipschitz continuity of $\nabla f$, the fact that $\normii{\Ukp{\hat{\di}^p_k}}=1$ and $\dk\leq\depsilon$. It follows from~\eqref{eq42Norm} that 
\[\normii{\nabla f(\xk)}\leq \alpha_Q^{-1}\kappa_{\min}^{-1}\left(L_g\sqrt{\frac{n}{p}}+(\gamma+2)\ef\right)\times \frac{\epsilon}{\xi}< \epsilon,\]
where the last inequality follows from~\eqref{depsEq}, and the proof is completed.
\end{proof}

% \section{Proof of Lemma~\ref{dynamicsLemma}}\label{lem53}

\section{Proof of Lemma~\ref{limsupWeakTail}}\label{lem61}

\begin{proof}
The proof is inspired by those of~\cite[Proposition~2.10 and Lemma~3.2]{rinaldi2022weak} and, unlike~\cite{rinaldi2022weak}, provides  an explanation of why it suffices to prove that $\forall\vartheta\in\N\setminus\accolade{0}$, the events $\mathscr{F}_{\vartheta}:=\accolade{\underset{k\in L}{\limsup} \frac{f\left(\Xk+\Dk\ukRandom\right)-f(\Xk)}{\Dk}\geq  -\frac{1}{\vartheta}}$ satisfy $\prob{\mathscr{F}_{\vartheta}}=1$. Indeed, in the latter case, since $\accolade{\mathscr{F}_{\vartheta}}_{\vartheta\geq 1}$ is a decreasing sequence, that is, $\mathscr{F}_{\vartheta+1}\subset \mathscr{F}_{\vartheta}$ for all $\vartheta\geq 1$, then the event $\mathscr{F}:=\underset{\vartheta\geq 1}{\bigcap} \mathscr{F}_{\vartheta}=\accolade{\underset{k\in L}{\limsup} \frac{f\left(\Xk+\Dk\ukRandom\right)-f(\Xk)}{\Dk}\geq 0}$ satisfies $1=\underset{\vartheta\to\infty}{\lim}\prob{\mathscr{F}_{\vartheta}}=\prob{\bigcap_{\vartheta\geq 1} \mathscr{F}_{\vartheta}}=\prob{\mathscr{F}}$, where the second equality follows from the continuity of $\pr$ (see, e.g.,~\cite[Proposition~I.1.4-$(vii)$]{franchi2013}).

Before diving into the proof of~\eqref{limsupWeakTresult}, we observe that the existence of an almost sure limit~$\hat{\urandom}$ of $\accoladekinL{\ukRandom}$ such that the almost sure event $\mathscr{E}_{\delta,\ub, {\bbbl \mathbb{U}}}\subseteq \accolade{\hat{\urandom}\in T_{\bcX}^H(\hat{\X})}$ ensures (by definition of the hypertangent cone) that $\Xk(\omega)+\Dk(\omega)\ukRandom(\omega)\in \bcX$ for infinitely many $k\in L(\omega)$ sufficiently large and all $\omega\in \mathscr{E}_{\delta,\ub, {\bbbl \mathbb{U}}}$. In other words, with probability one, $\hat{\urandom}$ is a refining direction for~$\hat{\X}$.

The remainder of the proof distinguishes three separate parts. Part~1 shows that for all $\vartheta\geq 1$, the following weak tail probabilistic inequality holds:
\begin{equation}\label{weakTailBound1}
\prob{\abs{\Fuk-\Fok-\left(f(\Xk+\Dk\ukRandom)-f(\Xk)\right)}\geq \left.\frac{\Dk}{\vartheta}\right\rvert\falgebra}\leq 4\vartheta^2\ef^2(1-\beta_f)\Dk^2\quad\mbox{almost surely.}
\end{equation}
The latter result is used in Part~2 to show that $\prob{\mathscr{F}_{\vartheta}}=1$, leading to~\eqref{limsupWeakTresult} as explained above. Part~3 explains how~\eqref{limsupWeakTresultPrime} follows from the proof of~\eqref{limsupWeakTresult}. \\
{\bf Part~1} {\it (Proving~\eqref{weakTailBound1})}. Consider the event $\mathscr{G}_{k,\vartheta}:=\accolade{\abs{\Fuk-\Fok-\left(f(\Xk+\Dk\ukRandom)-f(\Xk)\right)}\geq \frac{\Dk}{\vartheta}}$. It follows from the conditional Markov inequality that almost surely
\begin{eqnarray*}
\prob{\mathscr{G}_{k,\vartheta}\left\lvert\fqdalgebra\right.}&\leq& \Dk^{-2}\vartheta^2\E{\left.\abs{\Fuk-\Fok-\left(f(\Xk+\Dk\ukRandom)-f(\Xk)\right)}^2\right\rvert\fqdalgebra}\\
&\leq& 2\Dk^{-2}\vartheta^2\left(\E{\left.\abs{\Fuk-f(\Xk+\Dk\ukRandom)}^2\right\rvert\fqdalgebra}+\E{\left.\abs{\Fok-f(\Xk)}^2\right\rvert\fqdalgebra}\right)\\
&\leq&  4\vartheta^2\ef^2(1-\beta_f)\Dk^2.
\end{eqnarray*}
Thus, since $\falgebra\subseteq\fqdalgebra$ and $\Dk$ is $\falgebra$-measurable by construction, it holds that almost surely
\begin{eqnarray*}
\prob{\mathscr{G}_{k,\vartheta}\left\lvert\falgebra\right.}&=&\E{\mathds{1}_{\mathscr{G}_{k,\vartheta}}\left\lvert\falgebra\right.}=\E{\left.\E{\mathds{1}_{\mathscr{G}_{k,\vartheta}}\left\lvert\fqdalgebra\right.}\right\rvert\falgebra}=\E{\left.\prob{\mathscr{G}_{k,\vartheta}\left\lvert\fqdalgebra\right.}\right\rvert\falgebra}\\
&\leq& 4\vartheta^2\ef^2(1-\beta_f)\E{\Dk^2|\falgebra}=4\vartheta^2\ef^2(1-\beta_f)\Dk^2.
\end{eqnarray*}
{\bf Part~2} {\it (Getting~\eqref{limsupWeakTresult} using~\eqref{weakTailBound1})}. Taking the expectation on both sides of~\eqref{weakTailBound1} and using Theorem~\ref{zerothOrdTheor} lead to \[\sum_{k=0}^{\infty}\prob{\mathscr{G}_{k,\vartheta}}\leq 4\vartheta^2\ef^2(1-\beta_f)\sum_{k=0}^{\infty}\E{\Dk^2}=  4\vartheta^2\ef^2(1-\beta_f)\E{\sum_{k=0}^{\infty}\Dk^2}<\infty.\]
Then $\prob{\accolade{\abs{\Fuk-\Fok-\left(f(\Xk+\Dk\ukRandom)-f(\Xk)\right)}\Dk^{-1}\geq \frac{1}{\vartheta}}\ \ \mbox{i.o.}}{\bl =0\ }$ thanks to the Borel--Cantelli's first lemma, 
where ``i.o.'' stands for {\it infinitely often}. This implies that almost surely, for all $k$ sufficiently large, say $k\in L'\subset L$, $\abs{\Fuk-\Fok-\left(f(\Xk+\Dk\ukRandom)-f(\Xk)\right)}\Dk^{-1}\leq \frac{1}{\vartheta}$. In other words, \[\mathscr{G}_{\vartheta}:=\accolade{\underset{k\in L'}{\liminf}\abs{\Fuk-\Fok-\left(f(\Xk+\Dk\ukRandom)-f(\Xk)\right)}\Dk^{-1}\leq \frac{1}{\vartheta}}\ \mbox{satisfies}\ \prob{\mathscr{G}_{\vartheta}}=1,\]
which implies that the event $\mathscr{H}_{\vartheta}:=\mathscr{G}_{\vartheta}\cap\accolade{\lim_{k\in L'}\Dk=0}$ is almost sure. Then, passing to the $\limsup$ in the inequality
\[\frac{f(\xk+\dk\uk)-f(\xk)}{\dk}\geq \frac{\fuk-\fok-\abs{\fuk-\fok-(f(\xk+\dk\uk)-f(\xk))}}{\dk}\]
and using the inequality $\limsup(a_k+b_k)\geq \liminf a_k+\limsup b_k$ lead to
\begin{equation*}
\begin{split}
\underset{k\in L'}{\limsup} \frac{f(\Xk+\Dk\ukRandom)-f(\Xk)}{\Dk}&\geq \underset{k\in L'}{\liminf}\frac{\Fuk-\Fok}{\Dk}-\underset{k\in L'}{\liminf} \frac{\abs{\Fuk-\Fok-(f(\Xk+\Dk\ukRandom)-f(\Xk))}}{\Dk}\\
&\geq \underset{k\in L'}{\liminf}\left(-\gamma\ef\Dk\right)-\frac{1}{\vartheta}=-\frac{1}{\vartheta},
\end{split}
\end{equation*}
where the last inequality always holds on $\mathscr{H}_{\vartheta}$, thus implying that the event $\mathscr{F}_{\vartheta}$ is almost sure.\\
{\bf Part~3} {\it (Proving~\eqref{limsupWeakTresultPrime}).} If $L^0$ is an infinite subset of $L$, then $\hat{\urandom}$ is also the almost sure limit of $\accolade{\ukRandom}_{k\in L^0}$, which does not change the fact that $\prob{\hat{\urandom}\in T^H_{\bcX}(\hat{\X})}=1$ or that $\hat{\urandom}$ is a refining direction for $\hat{\X}$ with probability one. One can easily see that all the reasoning in Parts~1 and~2 remains true with $L$ replaced by $L^0$. This demonstrates~\eqref{limsupWeakTresultPrime}, which completes the proof.
\end{proof}

%-------------------------------------------------------------

% This can be deleted in the version that appears, it is only needed for the Argonne PANDA submission and ArXiv submission:
% \framebox{\parbox{\linewidth}{
% The submitted manuscript has been created by UChicago Argonne, LLC, Operator of 
% Argonne National Laboratory (``Argonne''). Argonne, a U.S.\ Department of 
% Energy Office of Science laboratory, is operated under Contract No.\ 
% DE-AC02-06CH11357. 
% The U.S.\ Government retains for itself, and others acting on its behalf, a 
% paid-up nonexclusive, irrevocable worldwide license in said article to 
% reproduce, prepare derivative works, distribute copies to the public, and 
% perform publicly and display publicly, by or on behalf of the Government.  The 
% Department of Energy will provide public access to these results of federally 
% sponsored research in accordance with the DOE Public Access Plan. 
% http://energy.gov/downloads/doe-public-access-plan.}}

\end{document}

%% file: preamble.tex
\usepackage{eufrak}
\usepackage[breaklinks=true, pdftex, 
pdfborder={0 0 0}, colorlinks=true, linkcolor=blue, citecolor=blue, urlcolor=blue, 
bookmarks=false, pdfpagelabels=false]{hyperref}
\usepackage{fullpage}
\usepackage{times}
\usepackage{ifthen}   % for conditionals
\usepackage{graphicx,graphics}
\usepackage{setspace}
\usepackage{amssymb,amsmath,amsthm}
\usepackage{dsfont}
\usepackage{color}
\usepackage{wrapfig}
\usepackage{multirow}
\usepackage{multicol}
\usepackage{pdfpages}
\usepackage{longtable} 	
\usepackage{booktabs}
\usepackage{datetime}
\usepackage{xspace}
\usepackage{paralist} %% This gives enumeration environments & compactenum 
\usepackage{enumitem} % For tighter lists
\usepackage[it,small,belowskip=-15pt]{caption}
\usepackage[top=1in,bottom=1in,left=0.95in,right=0.95in]{geometry} % Set margins
\usepackage[T1]{fontenc}
\usepackage{doi}
\usepackage[numbers,compress,sort]{natbib} %,sort
\usepackage{algorithmic}
\usepackage[noresetcount,vlined,boxed]{algorithm2e} % ... for algorithms
\let\oldnl\nl% Store \nl in \oldnl
\newcommand{\nonl}{\renewcommand{\nl}{\let\nl\oldnl}}
\usepackage{caption}
\usepackage{float}   
\usepackage{subcaption}
\usepackage[framemethod=tikz]{mdframed}

\usepackage{comment}
\usepackage{cleveref} % for cref
\usepackage{fontawesome5}

\usepackage{listings}
\usepackage{upgreek}
\usepackage[bbgreekl]{mathbbol}
\DeclareSymbolFontAlphabet{\mathbb}{AMSb}
\DeclareSymbolFontAlphabet{\mathbbl}{bbold}

\usepackage{lineno}  % To make it easier for reviewers

\definecolor{codegreen}{rgb}{0,0.6,0}
\definecolor{codegray}{rgb}{0.5,0.5,0.5}
\definecolor{codepurple}{rgb}{0.58,0,0.82}
\definecolor{backcolour}{rgb}{0.95,0.95,0.92}
\definecolor{ForestGreen}{RGB}{34,139,34}

\usepackage{xcolor}
\usepackage{listings}

\usepackage{xparse}
\usepackage{adjustbox}
\usepackage{tikz}
\usetikzlibrary{fit, arrows, calc, positioning}
\usetikzlibrary{arrows, automata}
\usetikzlibrary{positioning}
\usetikzlibrary{shadows}
\tikzset{box/.style={draw,rectangle,rounded  corners=0pt, fill=blue!15, minimum  width=3cm, minimum  height=2cm}}
\tikzset{loz/.style={draw,diamond, aspect=1.5, fill=blue!15}}

\tikzstyle{rnd}=[circle,fill=blue!25,minimum  width=5em]
\tikzset{
	quote/.style={{|[width=1mm]}-{|[width=1mm]}}
}
\usepackage{stackengine}   %  By Joseph
\newcommand{\ubar}[1]{\stackunder[1.0pt]{$#1$}{\rule{.8ex}{.075ex}}}    %  By Joseph

\NewDocumentCommand{\codeword}{v}{%
	\texttt{\textcolor{blue}{#1}}%
}

\lstdefinestyle{mystyle}{
	backgroundcolor=\color{backcolour},   
	commentstyle=\color{codegreen},
	keywordstyle=\color{magenta},
	numberstyle=\tiny\color{codegray},
	stringstyle=\color{codepurple},
	basicstyle=\ttfamily\footnotesize,
	breakatwhitespace=false,         
	breaklines=true,                 
	captionpos=b,                    
	keepspaces=true,                 
	numbers=left,                    
	numbersep=5pt,                  
	showspaces=false,                
	showstringspaces=false,
	showtabs=false,                  
	tabsize=2
}

\usepackage{titlesec}
\usepackage{mdframed}  % creating color text boxes
\usepackage{colortbl}
\usepackage{mathrsfs} % for \mathscr
\definecolor{highlight}{HTML}{DFEFF9}
\mdfdefinestyle{highlightbox}{leftmargin=0pt,rightmargin=0pt,%
	innerleftmargin=3pt,innerrightmargin=2pt,%
	innertopmargin=4pt,innerbottommargin=4pt,%
	backgroundcolor=highlight,linecolor=white}
\definecolor{algm}{HTML}{D4EFDF} % background color
\mdfdefinestyle{highlightalg}{leftmargin=0pt,rightmargin=0pt,%
	innerleftmargin=3pt,innerrightmargin=2pt,%
	innertopmargin=4pt,innerbottommargin=4pt,%
	backgroundcolor=algm,linecolor=white}
\definecolor{soft}{HTML}{F2D7D5} % background color
\mdfdefinestyle{highlightstruct}{leftmargin=0pt,rightmargin=0pt,%
	innerleftmargin=3pt,innerrightmargin=2pt,%
	innertopmargin=4pt,innerbottommargin=4pt,%
	backgroundcolor=soft,linecolor=white}

\usepackage{soul}

\makeatletter
\def\namedlabel#1#2{\begingroup
	#2%
	\def\@currentlabel{#2}%
	\phantomsection\label{#1}\endgroup
}
\makeatother

\newcommand{\R}{\Reals} % Real numbers x2
\newcommand{\eps}{\epsilon} 			% Nice looking epsilon
\DeclareMathOperator*{\argmax}{arg\,max}

\newcommand{\E}[2][]{\operatorname{\mathbb{E}}_{#1}\left[ #2\right]} % Exp value
 
\DeclareMathOperator\sign{sign}

\usepackage{bm} % For bold math symbols

\newcommand{\ub}{\bm{u}}
\newcommand{\Ub}{\bm{U}}

\newcommand{\vb}{\bm{v}}
\newcommand{\wb}{\bm{w}}

\newcommand{\yb}{\bm{y}}

\newcommand{\bcX}{\bm{\mathcal{X}}}
\newtheorem{definition}{Definition}[section]
\newtheorem{theorem}{Theorem}[section]
\newtheorem{lemma}{Lemma}[section]
\newtheorem{proposition}{Proposition}[section]
\newtheorem{assumption}{Assumption}[section]

\newtheorem{remark}{Remark}[section]

\usepackage{tikz}
\usetikzlibrary{positioning}
\usetikzlibrary{calc}
\usetikzlibrary{shapes.geometric}
\usetikzlibrary{topaths}
\usetikzlibrary{external}
\usetikzlibrary{fadings,decorations.pathreplacing} 

\usepackage{pgfplots}
\usetikzlibrary{shapes,decorations.pathmorphing,patterns}
\pgfplotsset{compat=newest}
\usepgfplotslibrary{fillbetween}

\graphicspath{{figures/}}
\usepackage{hyperref}
\hypersetup{
unicode = false,
pdftoolbar = true,
pdfmenubar = true,
pdffitwindow = true,
pdftitle = {Direct search for stochastic optimization in random subspaces},
pdfauthor = {Dzahini},
pdfsubject = {RandomMatrix},
pdfnewwindow = true,
pdfkeywords = {Random, Matrix},
colorlinks = true,
linkcolor = blue,
citecolor = blue,
filecolor = black,
urlcolor = blue,
breaklinks = true
}

\newcommand{\Z}{\mathbb{Z}}
\newcommand{\Esp}{\mathbb{E}}

\newcommand{\snOne}{\mathbb{S}^{n-1}}
\newcommand{\spOne}{\mathbb{S}^{p-1}}

\newcommand{\N}{\mathbb{N}}

\newcommand{\Pk}{\mathcal{P}_k}

\newcommand{\F}{\mathcal{F}}

\newcommand{\n}{\mathcal{N}}

\newcommand{\pr}{\mathbb{P}}

\newcommand{\dmax}{\delta_{\max}}

\newcommand{\ijk}{\mathds{1}_{J_k}}
\newcommand{\isk}{\mathds{1}_{\mathfrak{S}_k}}
\newcommand{\iskc}{\mathds{1}_{\bar{\mathfrak{S}}_k}}

\newcommand{\ijkc}{\mathds{1}_{\bar{J}_k}}

\newcommand{\abs}[1]{\left\lvert#1\right\rvert}
\newcommand{\intbracket}[1]{\left[\!\left[#1\right]\!\right]}
\newcommand{\accolade}[1]{\left\lbrace#1\right\rbrace}
\newcommand{\accoladekinN}[1]{{\left\lbrace#1\right\rbrace}_{k\in\mathbb{N}}}

\newcommand{\prob}[1]{\mathbb{P}\left(#1\right)}
\newcommand{\var}[1]{\mathbb{V}\left(#1\right)}

\newcommand{\norme}[1]{\left\lVert#1\right\rVert}
\newcommand{\scal}[2]{\left\langle#1,#2\right\rangle}

\newcommand{\normii}[1]{{\left\lVert#1\right\rVert}_{2}}

\newcommand{\rn}{\mathbb{R}^n}

\newcommand{\rp}{\mathbb{R}^p}
\newcommand{\rnn}{\mathbb{R}^{n\times n}}

\newcommand{\rnp}{\mathbb{R}^{n\times p}}

\newcommand{\rpn}{\mathbb{R}^{p\times n}}
\newcommand{\rpp}{\mathbb{R}^{p\times p}}
\newcommand{\rpm}{\mathbb{R}^{p\times m}}

\newcommand{\s}{\bm{s}}

\newcommand{\vi}{\bm{v}}

\newcommand{\Qk}{\bm{Q}_k} 
\newcommand{\QkRandom}{\bm{{\stackunder[0.6pt]{$Q$}{\rule{1.3ex}{.075ex}}}}_k}

\newcommand{\fok}{f_k({\bm{0}})}
\newcommand{\Fok}{{\stackunder[0.6pt]{$f$}{\rule{.8ex}{.075ex}}}_k({\bm{0}})}

\newcommand{\dkun}{{\delta_{k+1}}}

\newcommand{\iak}{\mathds{1}_{A_k}}
\newcommand{\Q}{\bm{Q}}

\newcommand{\Xk}{\bm{{\stackunder[1.0pt]{$x$}{\rule{1.0ex}{.075ex}}}}_k}
\newcommand{\X}{{\stackunder[1.0pt]{$\bm{x}$}{\rule{1.0ex}{.075ex}}}}
\newcommand{\xk}{\bm{x}_k}
\newcommand{\xkun}{\bm{x}_{k+1}}
\newcommand{\Xkun}{\bm{{\stackunder[1.0pt]{$x$}{\rule{1.0ex}{.075ex}}}}_{k+1}}
\newcommand{\x}{\bm{x}}

\newcommand{\di}{\bm{d}}

\newcommand{\falgebra}{\mathcal{F}_{k-1}}

\newcommand{\ef}{\varepsilon_f}
\newcommand{\Dk}{{\stackunder[1.0pt]{$\delta$}{\rule{.8ex}{.075ex}}}_k}
\newcommand{\Dkun}{{\stackunder[1.0pt]{$\delta$}{\rule{.8ex}{.075ex}}}_{k+1}}
\newcommand{\Dkmun}{{\stackunder[1.0pt]{$\delta$}{\rule{.8ex}{.075ex}}}_{k-1}}
\newcommand{\dk}{\delta_k}

\newcommand{\Phik}{{\stackunder[1.0pt]{$\phi$}{\rule{0.8ex}{.075ex}}}_k}
\newcommand{\Phikun}{{\stackunder[1.0pt]{$\phi$}{\rule{0.8ex}{.075ex}}}_{k+1}}
\newcommand{\Tepsilon}{{{\stackunder[1.0pt]{\scalebox{0.7}{$T$}}{\rule{0.8ex}{.075ex}}}_\epsilon}}
\newcommand{\Wtildek}{{{\stackunder[1.0pt]{$w$}{\rule{0.8ex}{.075ex}}}_k}}
\newcommand{\Wtildekun}{{{\stackunder[1.0pt]{$w$}{\rule{0.8ex}{.075ex}}}_{k+1}}}
\newcommand{\PhiZero}{{\stackunder[1.0pt]{$\phi$}{\rule{0.8ex}{.075ex}}}_0}

\newcommand{\Dzero}{{\stackunder[1.0pt]{$\delta$}{\rule{.8ex}{.075ex}}}_0}

\newcommand{\WtildeZero}{{{\stackunder[1.0pt]{$w$}{\rule{0.8ex}{.075ex}}}_0}}

\newcommand{\depsilon}{\delta_{\epsilon}}

\newcommand{\Wtildekmun}{{{\stackunder[1.0pt]{$w$}{\rule{0.8ex}{.075ex}}}_{k-1}}}
\newcommand{\Phikmun}{{\stackunder[1.0pt]{$\phi$}{\rule{0.8ex}{.075ex}}}_{k-1}}
\newcommand{\Amatrix}{\mathbf{A}}

\newcommand{\Qrandom}{\bm{{\stackunder[0.6pt]{$Q$}{\rule{1.3ex}{.075ex}}}}}
\newcommand{\QrandomTranspose}{\bm{{\stackunder[0.6pt]{$Q$}{\rule{1.3ex}{.075ex}}}}^{\top}}
\newcommand{\Qmax}{Q_{\max}}
\newcommand{\pollD}{\mathbb{D}^p}
\newcommand{\pollDrandom}{{\stackunder[0.6pt]{$\mathbb{D}$}{\rule{1.2ex}{.075ex}}}^p}

\newcommand{\Urandom}{\bm{{\stackunder[0.7pt]{$U$}{\rule{1.1ex}{.075ex}}}}}
\newcommand{\Drandom}{{\stackunder[0.6pt]{$\bm{D}$}{\rule{1.2ex}{.075ex}}}^p} 
\newcommand{\equald}{\overset{d}{=}}
\newcommand{\qrandom}{\scalebox{0.63}{$\Qrandom$}}
\newcommand{\drandom}{\scalebox{0.6}{$\Drandom$}}
\newcommand{\Mrandom}{\bm{{\stackunder[0.7pt]{$M$}{\rule{1.8ex}{.075ex}}}}}
\newcommand{\Pkfeas}{\mathcal{P}_k^{\mathcal{X}}}
\newcommand{\bl}{\color{black}}
\newcommand{\bll}{\color{black}}
\newcommand{\bbl}{\color{black}}
\newcommand{\bbbl}{\color{black}}
\newcommand{\uk}{\bm{u}_k}
\newcommand{\ukRandom}{\bm{{\stackunder[1.0pt]{$\ub$}{\rule{1.0ex}{.075ex}}}}_k}
\newcommand{\fuk}{f_k(\ub)}
\newcommand{\Fuk}{{\stackunder[0.6pt]{$f$}{\rule{.8ex}{.075ex}}}_k(\ubar{\ub})}
\newcommand{\Urandomtilde}{{\stackunder[0.7pt]{$\tilde{\bm{U}}$}{\rule{1.1ex}{.075ex}}}}

\newcommand{\BrandomSet}{{\stackunder[0.6pt]{$\mathbb{B}$}{\rule{1.1ex}{.075ex}}}}
\newcommand{\Brandom}{{\stackunder[0.6pt]{$\bm{B}$}{\rule{1.1ex}{.075ex}}}}
\newcommand{\urandom}{\bm{{\stackunder[1.0pt]{$u$}{\rule{1.0ex}{.075ex}}}}}
\newcommand{\urandomtilde}{{\stackunder[1.0pt]{$\tilde{\bm{u}}$}{\rule{1.0ex}{.075ex}}}}
\newcommand{\diRandom}{\bm{{\stackunder[1.0pt]{$d$}{\rule{1.0ex}{.075ex}}}}^p}
\newcommand{\dikRandom}{\bm{{\stackunder[1.0pt]{$\di$}{\rule{1.0ex}{.075ex}}}}_k^p}
\newcommand{\diRandomtilde}{{\stackunder[1.0pt]{$\tilde{\di}$}{\rule{1.0ex}{.075ex}}}}

\newcommand{\vrandom}{\bm{{\stackunder[1.0pt]{$v$}{\rule{0.7ex}{.075ex}}}}}

\newcommand{\fqdalgebra}{\mathcal{F}^{Q\cdot D}_{k-1}}
\newcommand{\accoladekinK}[1]{{\left\lbrace#1\right\rbrace}_{k\in K}}

\newcommand{\accoladekinL}[1]{{\left\lbrace#1\right\rbrace}_{k\in L}}

\newcommand{\fdk}{f_k(\di)}
\newcommand{\xiRandomVec}{\ubar{\bm{\xi}}}
\newcommand{\xiRandom}{\ubar{\xi}}
\newcommand{\UrandomScal}{{\stackunder[0.7pt]{$U$}{\rule{1.1ex}{.075ex}}}}
\newcommand{\XfracRandom}{\bm{{\stackunder[0.9pt]{$\mathfrak{X}$}{\rule{1.1ex}{.075ex}}}}}
\newcommand{\QfracRandom}{\bm{{\stackunder[0.9pt]{$\mathfrak{Q}$}{\rule{1.1ex}{.075ex}}}}}
\newcommand{\RfracRandom}{\bm{{\stackunder[0.9pt]{$\mathfrak{R}$}{\rule{1.1ex}{.075ex}}}}}
\newcommand{\DfracRandom}{\bm{{\stackunder[0.9pt]{$\mathfrak{A}$}{\rule{1.1ex}{.075ex}}}}}
\newcommand{\DfracRandomScal}{{\stackunder[0.9pt]{$\mathfrak{A}$}{\rule{1.1ex}{.075ex}}}}
\newcommand{\RfracRandomScal}{{\stackunder[0.9pt]{$\mathfrak{R}$}{\rule{1.1ex}{.075ex}}}}
\newcommand{\UfracRandom}{\bm{{\stackunder[0.9pt]{$\mathfrak{U}$}{\rule{1.1ex}{.075ex}}}}}
\newcommand{\ntp}{n \times p}
\newcommand{\ntnmp}{n \times (n-p)}

\newcommand{\pollDrandomtilde}{{\tilde{\stackunder[0.6pt]{$\mathbb{D}$}{\rule{1.2ex}{.075ex}}}}^p} 
\newcommand{\Drandomtildep}{\tilde{\stackunder[0.6pt]{$\bm{D}$}{\rule{1.2ex}{.075ex}}}^p}
\newcommand{\Ukp}{\Ub_{k,\ntp}}
\newcommand{\Ukprandom}{{{\stackunder[1.0pt]{$\Ub$}{\rule{1.0ex}{.075ex}}}}_{k,\ntp}}

\newcommand{\itksup}{\mathds{1}_{\accolade{\Tepsilon > k}}}
\newcommand{\UrandomSet}{{\stackunder[0.6pt]{$\mathbb{U}$}{\rule{1.1ex}{.075ex}}}}
\newcommand{\ukjoRandom}{\bm{{\stackunder[1.0pt]{$\ub$}{\rule{1.0ex}{.075ex}}}}_{k,j_0}}

\newcommand{\dikjoRandom}{\bm{{\stackunder[1.0pt]{$\bm{d}$}{\rule{1.0ex}{.075ex}}}}_{k,j_0}^p}
\newcommand{\UsetDrandom}{{\stackunder[0.6pt]{$\mathbb{U}$}{\rule{1.2ex}{.075ex}}}}